\documentclass{amsart}
\usepackage{amsmath,amssymb,mathrsfs,bbm,MnSymbol,longtable,mathtools}
\usepackage{amsthm,rotating,epsfig,color}

\usepackage{graphicx}
\usepackage{hyperref}
\allowdisplaybreaks

1

\date{}

\def\cabl#1{\lefthalfcup#1}
\def\cabr#1{\righthalfcup#1}
\def\catl#1{\lefthalfcap#1}
\def\catr#1{\righthalfcap#1}

\theoremstyle{theorem}
\newtheorem{theo}{Theorem}[section]
\newtheorem{lemm}[theo]{Lemma}
\newtheorem{prop}[theo]{Proposition}

\theoremstyle{remark}
\newtheorem{rema}[theo]{Remark}
\newtheorem{exam}[theo]{Example}

\theoremstyle{definition}
\newtheorem{defi}[theo]{Definition}

\numberwithin{equation}{section}
\numberwithin{figure}{section}

\author{Mikhail Chernavskikh, Ivan Dynnikov}
\thanks{}
\address{\noindent Steklov Mathematical Institute of Russian Academy of Sciences, 8 Gubkina Str., Moscow 119991, Russia}
\email{mike.chernavskikh@gmail.com}
\address{\noindent Steklov Mathematical Institute of Russian Academy of Sciences, 8 Gubkina Str., Moscow 119991, Russia}
\email{dynnikov@mech.math.msu.su}

\email{}

\graphicspath{{./pic/}}

\title{Rectangular diagrams of foliations}

\begin{document}
\maketitle

\begin{abstract}
A concept of a rectangular diagram of a foliation in the three-sphere
is introduced. It is shown that any co-orientable finite depth foliation in
the complement of a link admits a presentation by a rectangular diagram
compatible with the given rectangular diagram of the link.
\end{abstract}

\tableofcontents

\section{Introduction}
The aim of this paper is to propose a systematic way to represent
codimension-one foliations in link complements in the three-sphere. Among such foliations,
taut ones have received most attention in the literature.

Recall that a codimension-one foliation in a compact three-manifold~$M$ is \emph{taut}
if, for any leaf~$F$ of the foliation, there is a closed curve in~$M$
intersecting~$F$ and transverse to the foliation. If~$M$
has a non-empty boundary~$\partial M$, one typically requires that
the foliation is transverse to~$\partial M$. The boundary~$\partial M$
then must be a union of tori.

The importance of taut foliations comes, in particular,
from the fact that they can be used to detect the Seifert
genus of a link. Namely, a Seifert surface of a non-split link
can be included as a leaf into a taut foliation in the link
complement if and only if it has smallest possible genus.
In the `only if' direction this is a result of W.\,Thurston~\cite{thurston86},
and the converse was proved by D.\,Gabai~\cite{gabai83}.
If a link is split, then, by Novikov's theorem~\cite{novikov65}, its
complement supports no taut foliation.

In order to use taut foliations in applications, it would
be convenient to have a uniform way of representing them
by diagrams. Since foliations are quite complicated objects,
this is not obvious how to do, and to date, such a way has not been worked out. There
are two mutually related techniques that are commonly used
to confirm the \emph{existence} of a taut foliation in the given three-manifold:
sutured manifold hierarchies and branched surfaces. Such objects are usually presented
by a few pictures accompanied by a verbal description that contain
some combinatorial information about the foliation structure, but don't
specify a unique foliation up to isotopy.

The method proposed in this paper for presenting foliations allows one
to put into a single diagram all information needed to recover a foliation,
which includes not only combinatorial data, but also
holonomy maps. We achieve this by extending to foliations the formalism
of so called rectangular diagrams developed previously for links and surfaces in the
three-sphere.

The idea behind the rectangular diagram formalism is that we can represent
one-dimensional objects in~$\mathbb S^3$ such as links and graphs by
finite families of points in the two-torus~$\mathbb T^2$.
To do so, we view~$\mathbb S^3$ as a join of two circles, which is, by definition,
a quotient space of~$\mathbb T^2\times[0;1]$. So, every point of~$\mathbb T^2$
represents an arc in~$\mathbb S^3$. In this way, by following certain rules,
we can represent the one-dimensional skeleton of a surface so that the entire
surface is recovered uniquely. Once we have managed to represent individual
surfaces by collections of points in~$\mathbb T^2$, it is natural to try to represent
one-parametric families of surfaces and foliations by collections of curves.

There are, however, two technical difficulties in realizing this idea. First,
in order to be presented by a rectangular diagram, a surface must be `normalized'.
Normalizing an individual surface is not hard, but normalizing all leaves
of a foliation simultaneously is never possible. To overcome this, we allow finitely many leaves
to branch and to leave `cavities' between the branches. The genuine foliation
is obtained by deflating the cavities.

Another difficulty occurs with allowing the leaves to have non-empty boundary. Instead,
we consider foliations having each boundary component of the link complement as a leaf.
The definition of a taut foliation is modified by exempting such leaves from the requirement to
have an intersection with a closed transversal. If a foliation in the link complement
is taut in this sense, then by removing a small open collar neighborhood of the boundary
of the manifold one can obtain a taut foliation that is transverse to the new boundary.

With these two reservations in mind, we prove that any
finite depth taut foliation in the complement of a link
can be represented by what we call a rectangular diagram of a foliation.
As our anonymous referee rightly points out, the tautness of a foliation is not important
for the proof; our result extends, with almost no additional work, to any finite depth foliation of~$\mathbb S^3$.
We concentrate mostly on taut foliations because of their particular importance.

Our interest to this approach is due to the previously discovered wonderful properties
of rectangular diagrams of links and surfaces. Most remarkable is the fact
that rectangular diagrams of links allow one to recognize the unknot, split and composite links
by a monotonic simplification, that is, by searching sequences
of elementary moves that do not increase the complexity of the diagram~\cite{dynn06}.

Another important property of rectangular diagrams is their relation to contact topology:
all links represented by rectangular diagrams are Legendrian, and all surfaces are Giroux convex,
with respect to two contact structures on~$\mathbb S^3$: the standard  one and its mirror image.
These facts as well as the interpretation of elementary
transformations of rectangular diagrams in the contact topology terms have been used to solve the problem
of algorithmic classification of Legendrian and transverse links in~$\mathbb S^3$
(see~\cite{dynpras25,dynshast}).

One of the goals we have in mind when developing the rectangular diagram formalism
for foliations is an extension of the monotonic simplification approach of~\cite{dynn06}
to general links. It sounds plausible that, for every topological link type,
there are only finitely many combinatorial types of rectangular diagrams
representing this link type that do not admit a simplification by elementary moves.
As follows from~~\cite{dynpras25,dynshast}, classfying such diagrams is almost equivalent to classifying Legendrian link types
that do not admit a Legendrian destabilization. We think that a diagrammatic
presentation of taut foliations whose leaves, with finitely many exceptions, are Giroux
convex could be useful for that, since it might allow for a far-reaching generalization
of the method of the work~\cite{etnyrehonda}, where Legendrian Figure Eight knots
have been classified by analyzing the fibration in the knot complement
from the contact topology point of view.

The paper is organized as follows. In Section~\ref{links-and-surfaces-sec} we recall
the notions of a rectangular diagram of a link and a rectangular diagram of a surface,
and generalize the latter slightly. In Section~\ref{deform-sec} we discuss representing
thickened surfaces by means of rectangular diagrams. In Section~\ref{tube-sec}
we show how to represent the boundary of a tubular neighborhood
of a link by a rectangular diagram. In Section~\ref{pack-sec} we introduce
the building blocks from which rectangular diagrams of foliations
will be built. In Section~\ref{discs-sec} we consider a toy example, the
Reeb foliation in a solid torus, and reveal difficulties
with representing all leaves simultaneously by rectangular diagrams.
In Section~\ref{cavit-sec} we formalize the main obstacle to that in
general and explain how to cope with it. In Section~\ref{foli-sec}
we introduce our key object, a rectangular diagram of a foliation,
and formulate the main result, which states that every co-orientable finite depth
taut foliation in the complement to a non-split link admits a presentation
by a rectangular diagram of a foliation. Three examples are considered in Section~\ref{examp-sec}.
Section~\ref{normal-sec} is devoted to the proof of the main result.
In Section~\ref{finite-depth-sec} we sketch the argument that allows to extend it
to arbitrary finite depth foliations in~$\mathbb S^3$.

\subsection*{Acknowledgement}
This work was performed at the Steklov International Mathematical Center and supported by the Ministry of Science and Higher Education of the Russian Federation (agreement no. 075-15-2025-303). We are very grateful to the referee for pointing out several
inaccuracies in the original version of the paper.

\section{Rectangular diagrams of links and (quasi-)surfaces}\label{links-and-surfaces-sec}
In this section, we recall definitions of rectangular diagrams and some of their transformations,
mostly following~\cite{Representability,Distinguishing,Basic_moves}.
In addition to surfaces with corners considered previously in~\cite{Representability}, here we will need to deal with
not necessarily compact surfaces with a few identifications made
between boundary points, and call such objects \emph{quasi-surfaces}. The concept
of a rectangular diagram of a surface is extended here accordingly.

We denote by~$\mathbb T^2$ the two-torus~$\mathbb S^1\times\mathbb S^1$ and use~$\theta$ and~$\varphi$
for coordinates on the first and second copies of the circle~$\mathbb S^1$, respectively.

\begin{defi}
A \emph{rectangular diagram of a link} is a finite set $R$ of points in $\mathbb T^2$ such that every
meridian~$\{\theta\}\times\mathbb S^1$ and every longitude~$\mathbb S^1\times\{\varphi\}$ of $\mathbb T^2$ contains either no or exactly two points from $R$. The points of $R$ are called \emph{vertices} of $R$. If two vertices~$v$, $w$ of~$R$ lie
on the same meridian, the pair~$\{v,w\}$ is called \emph{a vertical edge} of~$R$, and if they are
on the same longitude, the pair~$\{v,w\}$ is called \emph{a horizontal edge} of~$R$.
\end{defi}

With any rectangular diagram of a link we associate a link in $\mathbb S^3$ as follows.  We view the three-sphere $\mathbb S^3$ as the join of two circles:
\begin{equation}\label{join-eq}
    \mathbb S^3 = \mathbb S^1 * \mathbb S^1 = \mathbb S^1 \times \mathbb S^1 \times [0; 1] /\bigl((\theta, \varphi, 0) \sim (\theta', \varphi, 0), (\theta, \varphi, 1)\sim (\theta, \varphi', 1)\ \forall\theta,\theta',\varphi,\varphi'\in\mathbb S^1\bigr),
\end{equation}
and use~$\tau$ to denote the coordinate on~$[0;1]$. The image of a point~$(\theta,\varphi,\tau)$
under the projection~$\mathbb S^1\times\mathbb S^1\times[0;1]\rightarrow\mathbb S^3$ will be denoted by~$[\theta,\varphi,\tau]$.

For a point $v = (\theta, \varphi) \in \mathbb T^2$, we denote by $\widehat v$  the arc~$\{[\theta,\varphi,\tau]\}_{\tau\in[0;1]}$. For any rectangular diagram of a link $R$, the \emph{associated link}~$\widehat R$
is defined as
$$\widehat R = \bigcup_{v\in R} \widehat v.$$

\begin{defi}
A \emph{rectangle} $r$
in the torus $\mathbb T^2 = \mathbb S^1 \times \mathbb S^1$ is a subset of the form $r=[\theta_1; \theta_2] \times [\varphi_1; \varphi_2]$, where $\theta_1, \theta_2, \varphi_1, \varphi_2 \in \mathbb S^1$ and $\theta_1 \ne \theta_2$, $\varphi_1 \ne \varphi_2$.
In this case, we also denote
$$\begin{aligned}
\theta_-(r)&= \theta_1,&\theta_+(r)&= \theta_2,&\varphi_-(r)&= \varphi_1,&\varphi_+(r)&= \varphi_2,\\
\cabl r&=(\theta_1,\varphi_1),&\catl r&=(\theta_1,\varphi_2),&\cabr r&=(\theta_2,\varphi_1),&\catr r&=(\theta_2,\varphi_2).
\end{aligned}$$
\end{defi}

\begin{defi}
We call two rectangles $r_1$, $r_2$ \emph{compatible} if their intersection $r_1\cap r_2$ satisfies one of the following conditions:
\begin{enumerate}
\item $r_1\cap r_2 = \varnothing$;
\item $r_1 \cap r_2$ is a subset of the set of vertices of $r_1$ (which is then a subset of the set of vertices of~$r_2$);
\item $r_1 \cap r_2$ is a rectangle free of vertices of rectangles $r_1$ and $r_2$.
\end{enumerate}	
\end{defi}

\begin{defi}
\label{def:diagram_quasi-surface}
A \emph{rectangular diagram of a quasi-surface} is a countable (possibly finite) collection $\Pi$ of pairwise
compatible rectangles in $\mathbb T^2$ such that
\begin{enumerate}
\item[(i)]
the two subsets of~$\mathbb S^1$
\begin{equation}\label{ThetaPhi-eq}
\Theta(\Pi)=\bigcup_{r \in \Pi} \{\theta_-(r), \theta_+(r)\},\quad\text{and}\quad
\Phi(\Pi)=\bigcup_{r \in \Pi} \{\varphi_-(r), \varphi_+(r)\}
\end{equation}
are discrete;
\item[(ii)]
for any~$x\in\mathbb S^1$, the number of rectangles~$r$ in~$\Pi$ such that

$$x\in\{\theta_-(r),\theta_+(r),\varphi_-(r),\varphi_+(r)\}$$
is finite.
\end{enumerate}

When~$\Pi$ is a rectangular diagram of a quasi-surface, any vertex of any rectangle~$r\in\Pi$
is called \emph{a vertex of}~$\Pi$. A vertex is \emph{free} if there is exactly one
such rectangle~$r$ in~$\Pi$.
By the \emph{boundary} of a rectangular diagram of a quasi-surface $\Pi$ we mean the set of all the free vertices of $\Pi$.
We denote the boundary of $\Pi$ by $\partial \Pi$. If~$\partial \Pi$ is a rectangular diagram of a link, then $\Pi$ is called a \emph{rectangular diagram of a surface}.
\end{defi}

\begin{rema}
In earlier papers~\cite{Representability,Distinguishing,Basic_moves} a rectangular diagram of a surface was required to be a finite collection of rectangles.
Here we allow it to be countable.
\end{rema}

\begin{defi}
Let~$\Pi$ be a rectangular diagram of a quasi-surface. By \emph{an orientation} on~$\Pi$ we mean
a function~$\epsilon:\Pi\to \{\pm 1\}$ such that
\begin{enumerate}
\item[(i)]
whenever~$r_1,r_2\in\Pi$, and any of the following conditions holds:
\begin{equation}\label{orientation1-eq}
\theta_-(r_1)=\theta_-(r_2)\quad\text{or}\quad
\theta_+(r_1)=\theta_+(r_2)\quad\text{or}\quad
\varphi_-(r_1)=\varphi_-(r_2)\quad\text{or}\quad
\varphi_+(r_1)=\varphi_+(r_2),
\end{equation}
we have~$\epsilon(r_1)=\epsilon(r_2)$;
\item[(ii)]
whenever~$r_1,r_2\in\Pi$, and any of the following conditions hold:
\begin{equation}\label{orientation2-eq}
\theta_-(r_1)=\theta_+(r_2)\quad\text{or}\quad
\varphi_-(r_1)=\varphi_+(r_2)
\end{equation}
we have~$\epsilon(r_1)\ne\epsilon(r_2)$.
\end{enumerate}
A pair~$(\Pi,\epsilon)$ in which~$\Pi$ is a rectangular diagram of a (quasi-)surface, and~$\epsilon$ is an
orientation on~$\Pi$ is called \emph{an oriented rectangular diagram of a \emph(quasi-\emph)surface}.
Rectangles~$r\in\Pi$ with~$\epsilon(r)=1$ (respectively, $\epsilon(r)=-1$) are then referred to as \emph{positive}
(respectively, \emph{negative}).
\end{defi}

For a subset~$X\in\mathbb R^3$ we denote by~$[0;1)\cdot X$ the open cone~$\{\lambda p\;:\;\lambda\in[0;1),\ p\in X\}$.

\begin{defi}\label{quasi-surf-def}
We say that a subset $F\subset \mathbb S^3$ is a \emph{quasi-surface} if, for any~$p\in F$, there exist
an open neighborhood~$U$ of~$p$ in~$\mathbb S^3$ and a $C^1$-diffeomorphism~$f$ from~$U$
to the open unit ball in~$\mathbb R^3$ such that~$f(F\cap U)$ has the form~$[0;1)\cdot X$,
where~$X$ is a compact one-dimensional submanifold of the equator of the unit sphere in~$\mathbb R^3$.

When~$F\subset\mathbb S^3$ is a quasi-surface, by \emph{the boundary} of~$F$ denoted~$\partial F$
we mean the set of all points in~$F$ that do not have a neighborhood in~$F$ homeomorphic to an open two-disc.
\end{defi}

The concept of a quasi-surface generalizes the notion of a surface with corners used in~\cite{Representability}
in two ways. First, a quasi-surface is not required to be compact. Second, it need not be, topologically, a two-dimensional
manifold, but it is always contained in a two-dimensional manifold and bounded in this manifold by a
union of smooth arcs.

With any rectangular diagram of a quasi-surface~$\Pi$ we associate a quasi-surface denoted by~$\widehat\Pi$
exactly in the same way as a surface with corners is associated
with a rectangular diagram of a surface in~\cite{Representability}. We recall the construction briefly.

For any rectangle~$r=[\theta_1;\theta_2]\times[\varphi_1;\varphi_2]\subset\mathbb T^2$,
we denote by~$h_r$ a bounded harmonic function on~$r\setminus\partial\{r\}$
(we use the notation~$\partial\{r\}$ for the set of vertices of~$r$, which agrees with Definition~\ref{def:diagram_quasi-surface})
such that~$h_r$ is identically zero at~$(\theta_1;\theta_2)\times\{\varphi_1,\varphi_2\}$
and identically one at~$\{\theta_1,\theta_2\}\times(\varphi_1;\varphi_2)$. Such a harmonic function
is unique, and an exact formula for it using the Weierstrass function can be found in~\cite{Representability}. By means of separation of variables one can get another
explicit expression for~$h_r$, which is the folowing one:
\begin{equation}\label{h_r-series-eq}
h_r(\theta,\varphi)=\frac4\pi\sum_{n=0}^\infty\frac{\sin\bigl(c_n(\varphi-\varphi_1)\bigr)\bigl(e^{c_n(\theta-\theta_1)}+e^{c_n(\theta_2-\theta)}\bigr)}
{(2n+1)\bigl(1+e^{c_n(\theta_2-\theta_1)}\bigr)},
\end{equation}
where
$$c_n=\frac{2n+1}{\varphi_2-\varphi_1}\pi.$$

We compose~$h_r$ with the function
\begin{equation}\label{zeta-func-eq}\zeta(x)=(2/\pi)\cdot\arctan\sqrt{\tan(\pi x/2)}
\end{equation}
(whose graph is shown in
Figure~\ref{atan_sqrt_tan-fig}) and denote the result by~$\widetilde h_r$:
$$\widetilde h_r(v)=\zeta(h_r(v)).$$
Let~$\widetilde r\subset\mathbb T^2\times[0;1]$ be the graph of the function~$\widetilde h_r$ on~$r\setminus\partial\{r\}$.
We define the \emph{tile}~$\widehat r$ associated with~$r$ as the closure of the image of~$\widetilde r$
in~$\mathbb S^3$
under identifications~\eqref{join-eq}.

\begin{figure}[ht]
    \centerline{\includegraphics[scale =.65]{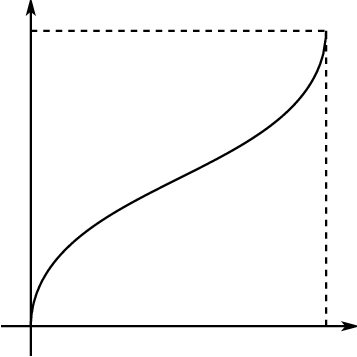}
    \put(-13, 0){$1$}
    \put(-110, 100){$1$}
    \put(-5, 13){$x$}
    \put(-98, 110){$\zeta(x)$}
    }
    \caption{The graph of $\zeta$}\label{atan_sqrt_tan-fig}
\end{figure}
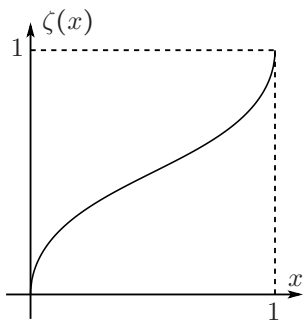

The association~$r\mapsto\widehat r$ defined above has the following properties (see the proof of Proposition~1
in~\cite{Representability}):
\begin{enumerate}
\item
the disc~$\widehat r$ is contained in the three-simplex~$[\theta_-(r);\theta_+(r)]*[\varphi_-(r);
\varphi_+(r)]\subset\mathbb S^1*\mathbb S^1=\mathbb S^3$;
\item
the boundary~$\partial\widehat r$ is the unknot~$\widehat{\partial\{r\}}=
\{\theta_-(r),\theta_+(r)\}*\{\varphi_-(r),\varphi_+(r)\}$;
\item
when two compatible rectangles~$r_1$ and~$r_2$ share a vertex~$v$, the tangent plane to the discs~$\widehat r_1$
and~$\widehat r_2$ at any point of the arc~$\widehat v$ coincide, and, moreover these
discs approach~$\widehat v$ from opposite sides, so~$\widehat r_1\cup\widehat r_2$ is a surface
with corners;
\item
when two rectangles~$r_1$ and~$r_2$ are compatible, share no corners, and a vertical (respectively, horizontal) side of~$r_1$
lies on the same meridian (respectively, longitude) as a vertical (respectively, horizontal) side of~$r_2$,
the discs~$\widehat r_1$ and~$\widehat r_2$ meet at the corresponding point at~$\mathbb S^1_{\tau=1}$
(respectively, $\mathbb S^1_{\tau=0}$), have the same tangent plane at
this point, and approach it in such a way that~$\widehat r_1\cup\widehat r_2$ is a quasi-surface;
\item
when two rectangles~$r_1$ and~$r_2$ are compatible and~$\{\theta_-(r_1),\theta_+(r_1)\}
\cap\{\theta_-(r_2),\theta_+(r_2)\}=\varnothing=
\{\varphi_-(r_1),\varphi_+(r_1)\}\cap\{\varphi_-(r_2),\varphi_+(r_2)\}$,
the discs~$\widehat r_1$ and~$\widehat r_2$ are disjoint.
\end{enumerate}
If $r$ and $r'$ are two compatible rectangles such that~$r\cap r'$ is a rectangle
and~$r$ has smaller $\theta$-dimension than~$r'$, then~$h_r(u)>h_{r'}(u)$
for any~$u\in r\cap r'$. So, with respect to the projection to the torus~$\mathbb T^2_{\tau=1/2}$
along the $\tau$-direction, the rectangle~$r$ `overpasses' $r'$.
For this reason we say, in this situation, that~$r$ \emph{overlays}~$r'$,
and draw such rectangles in the pictures accordingly.

\begin{prop}
    Let $\Pi$ be a rectangular diagram of a quasi-surface. Then the union $\widehat \Pi = \cup_{r\in\Pi}\widehat{r}$ is a quasi-surface.
\end{prop}
\begin{proof}
In the case when~$\Pi$ is finite the assertion of the proposition follows immediately from the listed above properties
of the construction~$r\mapsto\widehat r$. To establish the general case it suffices
to show that, for any~$r\in\Pi$, there is an open neighborhood~$U$ of~$\widehat r$ that intersects
only finitely many tiles~$\widehat r'$ with~$r'\in\Pi$.

Let~$r=[\theta_1; \theta_2]\times[\varphi_1; \varphi_2]\in\Pi$. Since the subsets~\eqref{ThetaPhi-eq}
are discrete, there exists~$\varepsilon>0$ such that
$$\begin{aligned}\relax
[\theta_1-\varepsilon;\theta_1+\varepsilon]\cap\Theta(\Pi)&=\{\theta_1\},&
[\theta_2-\varepsilon;\theta_2+\varepsilon]\cap\Theta(\Pi)&=\{\theta_2\},\\
[\varphi_1-\varepsilon;\varphi_1+\varepsilon]\cap\Phi(\Pi)&=\{\varphi_1\},&
[\varphi_2-\varepsilon;\varphi_2+\varepsilon]\cap\Phi(\Pi)&=\{\varphi_2\}.
\end{aligned}
$$

Define~$\Sigma=\{r_1,r_2,r_3,r_4,r_5,r_6\}$, where
$$    \begin{aligned}
    r_1&= [\theta_1 - \varepsilon; \theta_2 + \varepsilon] \times [\varphi_1 + \varepsilon; \varphi_2-\varepsilon]   ,\\
    r_2&= [\theta_1 + \varepsilon; \theta_2 - \varepsilon] \times [\varphi_1 - \varepsilon; \varphi_2 + \varepsilon] ,\\
    r_3&= [\theta_1 - \varepsilon; \theta_1 + \varepsilon] \times [\varphi_2 + \varepsilon; \varphi_1 - \varepsilon] ,\\
    r_4&= [\theta_2 - \varepsilon; \theta_2 + \varepsilon] \times [\varphi_2 + \varepsilon; \varphi_1 - \varepsilon] ,\\
    r_5&= [\theta_2  + \varepsilon; \theta_1 - \varepsilon] \times [\varphi_1 - \varepsilon; \varphi_1 + \varepsilon],\\
    r_6&= [\theta_2  + \varepsilon; \theta_1 - \varepsilon] \times [\varphi_2 - \varepsilon; \varphi_2 + \varepsilon].
    \end{aligned}$$
This diagram is shown in Figure~\ref{sigma-fig}.

    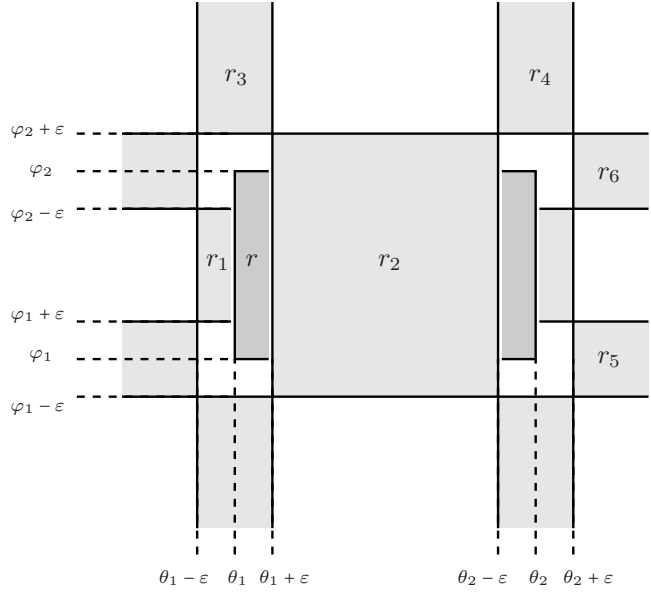
\begin{figure}[ht]
        \centerline{
            \includegraphics{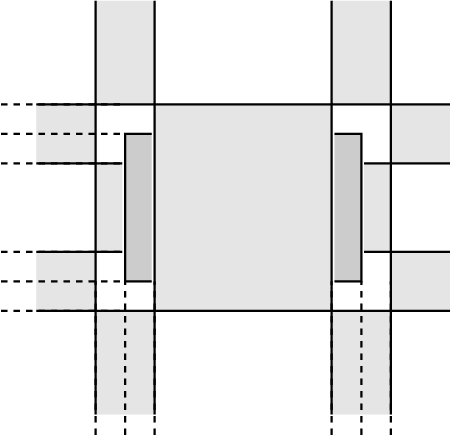}
           \put(-46, -10){\scriptsize$\theta_2$}
           \put(-73, -10){\scriptsize$\theta_2 - \varepsilon$}
           \put(-33, -10){\scriptsize$\theta_2 + \varepsilon$}
           \put(-148, -10){\scriptsize$\theta_1 + \varepsilon$}
           \put(-186, -10){\scriptsize$\theta_1 - \varepsilon$}
           \put(-160, -10){\scriptsize$\theta_1$}
           \put(-235, 145){\scriptsize$\varphi_2$}
           \put(-242, 128){\scriptsize$\varphi_2 - \varepsilon$}
           \put(-242, 160){\scriptsize$\varphi_2 + \varepsilon$}
           \put(-235, 75){\scriptsize$\varphi_1$}
           \put(-242, 55){\scriptsize$\varphi_1 - \varepsilon$}
           \put(-242, 90){\scriptsize$\varphi_1 + \varepsilon$}
           \put(-103, 110){$r_2$}
           \put(-153, 110){$r$}
           \put(-168, 110){$r_1$}
           \put(-161, 180){$r_3$}
           \put(-46, 180){$r_4$}
           \put(-20, 143){$r_6$}
           \put(-20, 73){$r_5$}
        }
        \caption{The rectangular diagram~$\Sigma$ of a sphere around~$\widehat r$}\label{sigma-fig}

    \end{figure}

By the Euler characteristic argument the surface $\widehat{\Sigma}$ is a two-sphere (the diagram
has~$6$ rectangles, $12$ vertices, $8$ meridians and longitudes containing
the vertices; $6-12+8=2$).
Each of the rectangles~$r_i$, $i=1,\ldots,6$, is compatible with~$r$, therefore, $\widehat\Sigma$ is disjoint
from~$\widehat r$. Denote by~$B$ the three-ball bounded by~$\widehat\Sigma$ whose interior contains~$\widehat r$.
We claim that~$B$ intersects only finitely many tiles of the form~$\widehat r'$, $r'\in\Pi$.

Indeed, due to the choice of~$\varepsilon$ and Condition~(ii) in Definition~\ref{def:diagram_quasi-surface},
there are only finitely many rectangles~$r'$ in~$\Pi$ such that
$$\{\theta_-(r'),\theta_+(r')\}\cap\{\theta_1,\theta_2\}\ne\varnothing
\quad\text{or}\quad
\{\varphi_-(r'),\varphi_+(r')\}\cap\{\varphi_1,\varphi_2\}\ne\varnothing.$$
Suppose that~$r'\in\Pi$ is such that none of these conditions holds. Then one can see that~$r'$
is compatible with~$r_i$, $i=1,\ldots,6$, and hence~$\widehat r'\cap\widehat\Sigma=\varnothing$.
It remains to ensure that~$\widehat r'$ is not contained in the interior of~$B$.

The intersection of~$\widehat\Sigma$ with~$\mathbb S^1_{\tau=1}$ is, by construction, the four-point set~$\{\theta_1\pm\varepsilon,
\theta_2\pm\varepsilon\}$, so the intersection~$B\cap\mathbb S^1_{\tau=1}$ is a union of two disjoint arcs.
Since~$\widehat r\subset B$ and~$\widehat r\cap\mathbb S^1_{\tau=1}=\{\theta_1,\theta_2\}$,
we have~$B\cap\mathbb S^1_{\tau=1}=[\theta_1-\varepsilon;\theta_1+\varepsilon]
\cup[\theta_2-\varepsilon;\theta_2+\varepsilon]$.
By assumption, the intersection of~$\widehat r'$ with~$\mathbb S^1_{\tau=1}$
is outside of these arcs, which implies~$\widehat r'\not\subset B$. Hence~$\widehat r'\cap B=\varnothing$.
\end{proof}

    The quasi-surface associated with an oriented rectangular diagram of a surface~$(\Pi,\epsilon)$ is oriented so that the coordinate pair $(\theta,\varphi)$ is positively oriented in the interior of~$\widehat r$ for any positive rectangle~$r\in\Pi$,
    and negatively oriented for any negative~$r\in\Pi$. Such an orientation of~$\widehat\Pi$
    does exist, since whenever two rectangles~$r_1,r_2\in\Pi$ satisfy one of the conditions~\eqref{orientation1-eq}
    or~\eqref{orientation2-eq} (which means that~$\widehat r_1\cap\widehat r_2$ is not empty)
    the orientations of~$\widehat r_1$ and~$\widehat r_2$ at common points agree.

In order to specify an orientation of a rectangular diagram of a quasi-surface~$\Pi$, we prefer to indicate the
corresponding coorienation of the surface~$\widehat\Pi$ at the intersection points with~$\mathbb S^1_{\tau=0}$
and~$\mathbb S^1_{\tau=1}$ by marking the longitudes and meridians containing
the sides of rectangles from~$\Pi$ with arrows perpendicular to them. If a rectangle~$r\in\Pi$
is positive, then the arrows at the meridians containing the vertical sides of~$r$ will
point inward the region containing~$r$, and the arrows at the longitudes containing
the horizontal sides of~$r$ will point outward. For a negative~$r$ the rule will be opposite;
see Figure~\ref{pos-neg-rect-fig}.
\begin{figure}
    \includegraphics[scale=1.5]{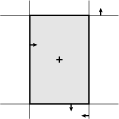}
    \hspace{1cm}
    \includegraphics[scale=1.5]{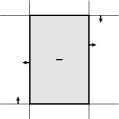}
\caption{Positive and negative rectangles}\label{pos-neg-rect-fig}
\end{figure}

\section{Deformations of rectangular diagrams of quasi-surfaces}\label{deform-sec}

For any rectangle~$r=[\theta_1;\theta_2]\times[\varphi_1;\varphi_2]$,
we define a parametrization
$$u_r:[0;1]\times[0;1]\rightarrow\widehat r$$
of the tile~$\widehat r$
as follows. First, introduce the following maps~$\iota_r$, $\lambda_r$, $\mu_r$, and~$\eta_r$ from~$r\setminus\partial r$
to~$\mathbb S^3$, $[0;1]$, $[0;1]$, and~$[0;1]\times[0;1]$, respectively:
$$\begin{aligned}
\iota_r(\theta,\varphi)&=[\theta,\varphi,\widetilde h_r(\theta,\varphi)],\\
\lambda_r(\theta,\varphi)&=\sin\Bigl(\frac\pi2\,\widetilde h_r(\theta,\varphi)\Bigr),\quad
\mu_r(\theta,\varphi)=\cos\Bigl(\frac\pi2\,\widetilde h_r(\theta,\varphi)\Bigr),\\
\eta_r(\theta,\varphi)&=\frac1{\lambda_r(\theta,\varphi)+\mu_r(\theta,\varphi)}
\left(\frac{\lambda_r(\theta,\varphi)}{\theta_2-\theta_1}(\theta-\theta_1,\theta_2-\theta)+
\frac{\mu_r(\theta,\varphi)}{\varphi_2-\varphi_1}(\varphi-\varphi_1,\varphi-\varphi_1)\right).\end{aligned}$$
The maps~$\iota_r:(r\setminus\partial r)\rightarrow\widehat r$ and~$\eta_r:(r\setminus\partial r)\rightarrow[0;1]\times[0;1]$ cannot be extended continuously to the entire
rectangle~$r$. They admit continuous extensions to the sides of~$r$, but not to the corners. The point
is that these maps deflate the sides of~$r$ to the vertices,
and `inflate' the vertices of~$r$ to the sides of~$\widehat r$ and~$[0;1]\times[0;1]$, respectively.
Figure~\ref{coordinate_lines} shows the image of a square grid in~$(0;1)\times(0;1)$ under the map~$\eta_r^{-1}$.

\begin{figure}[h]
    \includegraphics[scale = 4]{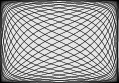}
    \caption{The image of the coordinate grid in~$(0;1)\times(0;1)$ under the map~$\eta_r^{-1}$}\label{coordinate_lines}
\end{figure}

 However, the composition~$\iota_r\circ\eta_r^{-1}$ can be shown to be extendable
to a continuous embedding~$[0;1]\times[0;1]\rightarrow\mathbb S^3$ which is smooth everywhere except
at the corners of the square~$[0;1]\times[0;1]$. We take this composition extended to the entire square~$[0;1]\times[0;1]$
for~$u_r$. The explicit formulas for the restriction of~$u_r$ to the boundary~$\partial r$
are as follows:
\begin{equation}\label{boundary-param-eq}\begin{aligned}
u_r(x,0)&=\left[\theta_2,\varphi_1,\frac2\pi\arctan\Bigl(\frac x{1-x}\Bigr)\right],\quad&
u_r(x,1)&=\left[\theta_1,\varphi_2,\frac2\pi\arctan\Bigl(\frac{1-x}x\Bigr)\right],\\
u_r(0,y)&=\left[\theta_1,\varphi_1,\frac2\pi\arctan\Bigl(\frac y{1-y}\Bigr)\right],\quad&
u_r(1,y)&=\left[\theta_2,\varphi_2,\frac2\pi\arctan\Bigl(\frac{1-y}y\Bigr)\right].
\end{aligned}\end{equation}

\begin{prop}\label{octagon-prop}
Let~$\boldsymbol\theta_1,\boldsymbol\theta_2,\boldsymbol\varphi_1,\boldsymbol\varphi_2$ be continuous maps from~$[0;1]$ to~$\mathbb S^1$ such that
\begin{enumerate}
\item[(i)]
$\boldsymbol\theta_1$ and~$\boldsymbol\varphi_2$ are strictly increasing\emph;
\item[(ii)]
$\boldsymbol\theta_2$ and~$\boldsymbol\varphi_1$ are strictly decreasing\emph;
\item[(iii)]
for all~$t\in[0;1]$, we have~$\boldsymbol\theta_1(t)\ne\boldsymbol\theta_2(t)$, $\boldsymbol\varphi_1(t)\ne\boldsymbol\varphi_2(t)$.
\end{enumerate}
Then the map
\begin{equation}\label{cube-map-eq}
(t,x,y)\mapsto u_{r(t)}(x,y),
\end{equation}
where~$r(t)=[\boldsymbol\theta_1(t);\boldsymbol\theta_2(t)]\times[\boldsymbol\varphi_1(t);\boldsymbol\varphi_2(t)]$
is a continuous embedding of the cube~$[0;1]^3$ to~$\mathbb S^3$.
\end{prop}
\begin{proof}
The continuity of this map follows easily from the definition. For any~$t\in[0;1]$ the
map~$u_{r(t)}:[0;1]^2\rightarrow\mathbb S^3$
is already known to be an embedding whose image is~$\widehat r(t)$. It remains to notice that, for any distinct~$t,t'\in[0;1]$,
the rectangles~$r(t)$ and~$r(t')$ are compatible and their intersection is a rectangle, which
implies that the discs~$\widehat r(t)$ and~$\widehat r(t')$ are disjoint.
\end{proof}

Figure~\ref{curved-cube-fig} demonstrates how the image of the map~\eqref{cube-map-eq}
projected stereographically to~$\mathbb R^3$ may look like.

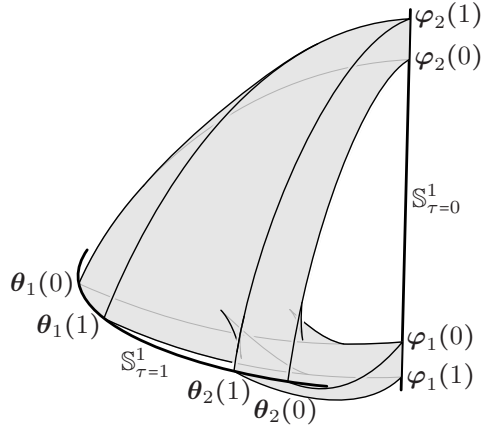
\begin{figure}[ht]
\includegraphics[scale=.7]{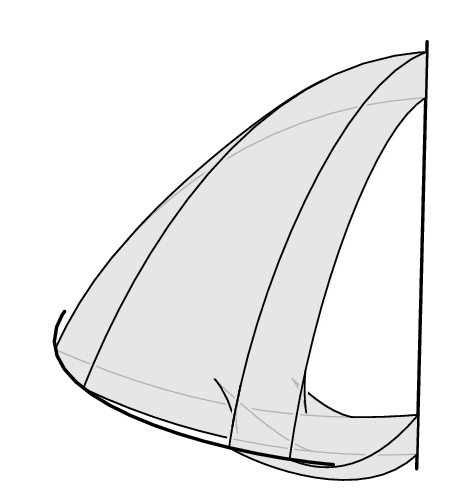}
\put(-10,80){$\mathbb S^1_{\tau=0}$}
\put(-12,13){$\boldsymbol\varphi_1(1)$}
\put(-12,28){$\boldsymbol\varphi_1(0)$}
\put(-9,134){$\boldsymbol\varphi_2(0)$}
\put(-9,150){$\boldsymbol\varphi_2(1)$}
\put(-120,18){$\mathbb S^1_{\tau=1}$}
\put(-162,48){$\boldsymbol\theta_1(0)$}
\put(-152,32){$\boldsymbol\theta_1(1)$}
\put(-95,7){$\boldsymbol\theta_2(1)$}
\put(-70,0){$\boldsymbol\theta_2(0)$}
\caption{The image of the map~\eqref{cube-map-eq}}\label{curved-cube-fig}
\end{figure}

For an oriented rectangular diagram of a quasi-surface~$(\Pi,\epsilon)$, introduce the following notation:
$$\begin{aligned}
\Theta_+(\Pi,\epsilon)&=\{\theta_-(r):r\in\Pi,\ \epsilon(r)=1\}\cup
\{\theta_+(r):r\in\Pi,\ \epsilon(r)=-1\},\\
\Theta_-(\Pi,\epsilon)&=\{\theta_-(r):r\in\Pi,\ \epsilon(r)=-1\}\cup
\{\theta_+(r):r\in\Pi,\ \epsilon(r)=1\},\\
\Phi_+(\Pi,\epsilon)&=\{\varphi_+(r):r\in\Pi,\ \epsilon(r)=1\}\cup
\{\varphi_-(r):r\in\Pi,\ \epsilon(r)=-1\},\\
\Phi_-(\Pi,\epsilon)&=\{\varphi_+(r):r\in\Pi,\ \epsilon(r)=-1\}\cup
\{\varphi_-(r):r\in\Pi,\ \epsilon(r)=1\}.
\end{aligned}$$
In other words, $\Theta_+(\Pi,\epsilon)$ (respectively, $\Theta_-(\Pi,\epsilon)$)
is the set of points~$p\in\widehat\Pi\cap\mathbb S^1_{\tau=1}$
at which the intersection index of~$\widehat\Pi$ and~$\mathbb S^1_{\tau=1}$ is~$+1$ (respectively, $-1$).
Similarly, $\Phi_+(\Pi,\epsilon)$ (respectively, $\Phi_-(\Pi,\epsilon)$)
is the set of points~$p\in\widehat\Pi\cap\mathbb S^1_{\tau=0}$
at which the intersection index of~$\widehat\Pi$ and~$\mathbb S^1_{\tau=0}$ is~$+1$ (respectively, $-1$).

\begin{defi}
Let~$(\Pi,\epsilon)$ be an oriented rectangular diagram of a quasi-surface. By \emph{a positive deformation}
of~$(\Pi,\epsilon)$ we mean a couple~$f=(f^0,f^1)$ of continuous maps
$$f^0:\Phi(\Pi)\times[0;1]\rightarrow\mathbb S^1,\quad f^1:\Theta(\Pi)\times[0;1]\rightarrow\mathbb S^1$$
such that
\begin{enumerate}
\item
$f^0(\varphi,0)=\varphi$ and~$f^1(\theta,0)=\theta$ for any~$\varphi\in\Phi(\Pi)$ and~$\theta\in\Theta(\Pi)$;
\item
for any~$t\in[0;1]$, the maps~$f^0(\cdot,t)$ and~$f^1(\cdot,t)$ are embeddings (so~$f^0$ and~$f^1$ are isotopies);
\item
for any~$\theta\in\Theta_+(\Pi,\epsilon)$ and~$\varphi\in\Phi_+(\Pi,\epsilon)$ the functions~$f^0(\varphi,\cdot)$
and~$f^1(\theta,\cdot)$ are strictly increasing;
\item
for any~$\theta\in\Theta_-(\Pi,\epsilon)$ and~$\varphi\in\Phi_-(\Pi,\epsilon)$ the functions~$f^0(\varphi,\cdot)$
and~$f^1(\theta,\cdot)$ are strictly decreasing.
\end{enumerate}
When these conditions hold we denote by~$f_t(r)$, where~$r\in\Pi$ and~$t\in[0;1]$,
the following rectangle:
$$f_t(r)=[f^1(\theta_-(r),t);f^1(\theta_+(r),t)]\times[f^0(\varphi_-(r),t);f^0(\varphi_+(r),t)],$$
and by~$f_t(\Pi)$
the following rectangular diagram of a quasi-surface:
$$f_t(\Pi)=\{f_t(r):r\in\Pi\}.$$
We also denote by~$f_t(\Pi,\epsilon)$ the oriented
rectangular diagram of a quasi-surface~$(f_t(\Pi),\epsilon')$, where
$$\begin{aligned}
\epsilon'(r')&=\epsilon(r)\text{ if }r\in\Pi,\text{ and }r'=f_t(r).
\end{aligned}$$
We will say that rectangular diagrams~$f_t(\Pi)$ and~$f_t(\Pi,\epsilon)$, $t\in(0;1]$, are obtained from~$\Pi$ and~$(\Pi,\epsilon)$,
respectively, by a \emph{positive deformation}.
If the functions~$f^0(\cdot,1)$ and~$f^1(\cdot,1)$ are close to~$f^0(\cdot,0)$ and~$f^1(\cdot,0)$, respectively,
we say that the positive deformation is \emph{small}.
\end{defi}

Proposition~\ref{octagon-prop} can be generalized as follows.

\begin{prop}\label{thickened_surface-prop}
Let~$f=(f^0,f^1)$ be a positive deformation of an oriented rectangular diagram of a quasi-surface~$(\Pi,\epsilon)$,
and let~$(\Pi_t,\epsilon_t)=f_t(\Pi,\epsilon)$. Then there is an orientation preserving immersion
$$I_f:\widehat\Pi\times[0;1]\rightarrow\mathbb S^3$$
taking~$\widehat\Pi\times\{t\}$ to~$\widehat\Pi_t$.
If, additionally, both maps~$f^0$ and~$f^1$ are embeddings,
then~$I_f$ is also an embedding.
\end{prop}

\begin{proof}
First, we define a map~$\widetilde I_f:\widehat\Pi\times[0;1]\rightarrow\mathbb S^3$
by
$$\widetilde I_f(p,t)=(u_{f_t(r)}\circ u_r^{-1})(p)\text{ if }r\in\Pi\text{ and }p\in\widehat r.$$
It follows from the parametrization~\eqref{boundary-param-eq} that this map
is well defined at the intersections~$\widehat r\cap\widehat r'$ whenever~$r$ and~$r'$
are rectangles from~$\Pi$ sharing a corner. One can see that~$\widetilde I_f$
takes~$\widehat\Pi\times\{t\}$ to~$\widehat\Pi_t$.

Proposition~\ref{octagon-prop} implies that~$\widetilde I_f$ is a topological immersion which
is smooth everywhere except at~$\bigcup_{r\in\Pi}\partial\widehat r\times[0;1]$.
Clearly, $\widetilde I_f$ can be $C^0$-approximated by a smooth immersion~$I_f$
also taking~$\widehat\Pi\times\{t\}$ to~$\widehat\Pi_t$ for all~$t\in[0;1]$.

If~$f^0$ and~$f^1$ are embeddings, then the surfaces~$\widehat\Pi_t$, $t\in[0;1]$ are
pairwise disjoint, so the maps~$\widetilde I_f$ and~$I_f$ are embeddings.
\end{proof}

\section{Thin tubes around a link}\label{tube-sec}

\begin{prop}\label{tube-prop}
Let $R$ be a rectangular diagram of a link, and let~$d$ be the smallest distance between
two neighboring meridians or two neighboring longitudes containing vertices of~$R$. Denote by~$\mathfrak v$
\emph(respectively, by $\mathfrak h$\emph)
the bijection~$R\rightarrow R$ exchanging vertices in each vertical \emph(respectively, horizontal\emph) edge.
Then, for any~$t\in(0;d/2)$, the following collection of rectangles
$$\Omega_t(R)=\bigcup_{v=(\theta_v,\varphi_v)\in R}\bigl\{[\theta_v+t;\theta_{\mathfrak h(v)}-t]\times[\varphi_v-t;\varphi_v+t],
[\theta_v-t;\theta_v+t]\times[\varphi_v+t;\varphi_{\mathfrak v(v)}-t]\bigr\}$$
is a rectangular diagram of a surface \emph(consult Figure~\ref{tube-fig}\emph). The associated surface~$\widehat\Omega_t(R)$
is the boundary of a tubular neighborhood of the link~$\widehat R$,
and this tubular neighborhood is
\begin{equation}\label{N_t(R)-eq}
N_t(R)=\widehat R\cup\bigcup_{0<s<t}\widehat\Omega_s(R).
\end{equation}
\end{prop} 
\begin{figure}[ht]
\centering{
\includegraphics[scale=0.75]{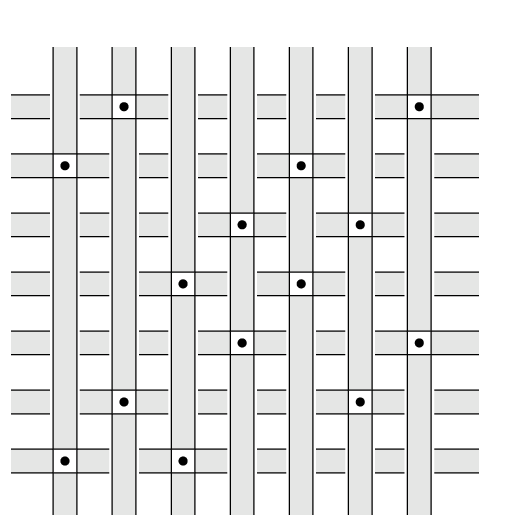}
}
\caption{Rectangular diagram of a link~$R$ and the rectangular diagram~$\Omega_t(R)$
of a tube around~$\widehat R$}
\label{tube-fig}
\end{figure}

To prove the proposition we need the following
\begin{lemm}\label{lemm:neighborhood}
Let~$v_1=(\theta_1,\varphi_1)$ and~$v_2=(\theta_2,\varphi_2)$ be two distinct points of~$\mathbb T^2$
lying on the same longitude \emph(respectively, on the same meridian\emph), and let~$U$
be an open neighborhood of the arc~$\widehat v_1\cup\widehat v_2$.
Then, for small enough~$\varepsilon>0$,
the disk~$\widehat r_\varepsilon$ associated with the rectangle $r_\varepsilon=[\theta_1+\varepsilon;\theta_2-\varepsilon]\times[\varphi_1-\varepsilon;\varphi_1+\varepsilon]$
\emph(respectively, $r_\varepsilon= [\theta_1 - \varepsilon; \theta_1 + \varepsilon] \times[\varphi_1 + \varepsilon; \varphi_2 - \varepsilon]$\emph)
is contained in~$U$.
\end{lemm}
\begin{proof}
Due to symmetry it suffices to consider the case when~$v_1$ and~$v_2$ lie on the same longitude,
which means $\varphi_1 = \varphi_2$.

Now if~$\varepsilon>0$ is small, then the subset
$$V=\bigl\{[\theta,\varphi,\tau]\in\mathbb S^3:
(\theta,\varphi,\tau)\in[\theta_1 + \varepsilon; \theta_2 - \varepsilon]\times[\varphi_1 - \varepsilon; \varphi_1 + \varepsilon]\times\bigl[0; \zeta(8\sqrt\varepsilon/\pi^2)\bigr]\bigr\}$$
is contained in a small neighborhood of~$[\theta_1,\varphi_1,0]{}=[\theta_2,\varphi_1,0]=\widehat v_1\cap\widehat v_2$
(recall that the function~$\zeta$
is defined by~\eqref{zeta-func-eq}), and the subset
$$W=\bigl[\theta_1+\varepsilon; \theta_1 + \varepsilon + \sqrt\varepsilon\bigr]*[\varphi_1 - \varepsilon; \varphi_1 + \varepsilon] \cup
\bigl[\theta_2-\varepsilon - \sqrt\varepsilon; \theta_2 - \varepsilon\bigr]*[\varphi_1 - \varepsilon; \varphi_1 + \varepsilon]\subset\mathbb S^1*\mathbb S^1=
\mathbb S^3$$
is contained in a small neighborhood of~$\widehat v_1\cup\widehat v_2$.
Thus, for small enough~$\varepsilon>0$ the subset~$V\cup W$ is contained in~$U$.

Now we claim that~$\widehat r_\varepsilon\subset V\cup W$. Indeed, from~\eqref{h_r-series-eq} we have for~$(\theta,\varphi)\in r_\varepsilon\setminus\partial\{r_\varepsilon\}$:
\begin{equation}\label{h_r_eps-eq}
h_{r_\varepsilon}(\theta,\varphi)=\frac4\pi\sum_{n=0}^\infty
\frac{\sin(c_n(\varphi-\varphi_1+\varepsilon))\bigl(e^{c_n(\theta-\theta_1-\varepsilon)}+e^{c_n(\theta_2-\theta-\varepsilon)}\bigr)}
{(2n+1)\bigl(1+e^{c_n(\theta_2-\theta_1-2\varepsilon)}\bigr)},
\end{equation}
where
$$c_n=\frac{2n+1}{2\varepsilon}\pi.$$
If~$\theta\in[\theta_1+\varepsilon+\sqrt\varepsilon;\theta_2-\varepsilon-\sqrt\varepsilon]$, then the sum~\eqref{h_r_eps-eq}
is bounded from above by
$$\frac8\pi\sum_{n=0}^\infty e^{-c_n\sqrt\varepsilon}=\frac4{\pi\sinh(\pi/(2\sqrt\varepsilon))}<\frac{8\sqrt\varepsilon}{\pi^2},$$
hence~$\bigl[\theta,\varphi,\widetilde h_{r_\varepsilon}(\theta,\varphi)\bigr]\in W$. Otherwise,
$\bigl[\theta,\varphi,\widetilde h_{r_\varepsilon}(\theta,\varphi)\bigr]\in V$.
\end{proof}

\begin{proof}[Proof of Proposition~\ref{tube-prop}]
It follows from Lemma~\ref{lemm:neighborhood} that, for any open neighborhood~$U$ of~$\widehat R$,
the surface~$\widehat\Omega_t(R)$ is contained in~$U$ provided that~$t$ is small enough.

For~$t<d/2$, define an orientation~$\epsilon_t$ of the diagram~$\Omega_t(R)$ by putting~$\epsilon_t(r)=-1$
if~$r$ has the form~$[\theta_v-t;\theta_v+t]\times[\varphi_v+t;\varphi_{\mathfrak v(v)}-t]$, and~$\epsilon_t(r)=1$
if~$r$ has the form~$[\theta_v+t;\theta_{\mathfrak h(v)}-t]\times[\varphi_v-t;\varphi_v+t]$.
Then whenever~$0<s<t<d/2$, the oriented rectangular diagram~$(\Omega_t,\epsilon_t)$
is obtained from~$(\Omega_s,\epsilon_s)$ by a positive deformation, and the respective
surfaces~$\widehat\Omega_t$ and~$\widehat\Omega_s$ are disjoint.
Moreover, it follows from Proposition~\ref{thickened_surface-prop} that~$\bigcup_{s\leqslant z\leqslant t}\widehat\Omega_z$
is diffeomorphic to~$\Omega_s\times[0;1]$. The claim follows.
\end{proof}

\section{Packs of rectangles}\label{pack-sec}
\begin{defi}\label{pack-rect-def}
By a \emph{positive pack of rectangles} we call an ordered family~$P$
of rectangles in~$\mathbb T^2$ having the form~$\{r(t)\}_{t\in[0;1]}$,
with~$r(t)$ as in Proposition~\ref{octagon-prop}, such that
the order in~$P$ agrees with the one induced by the parametrization~$t\mapsto r(t)$.
The four arcs
$$\{(\boldsymbol\theta_1(t),\boldsymbol\varphi_1(t))\}_{t\in[0;1]},\quad
\{(\boldsymbol\theta_2(t),\boldsymbol\varphi_1(t))\}_{t\in[0;1]},\quad
\{(\boldsymbol\theta_1(t),\boldsymbol\varphi_2(t))\}_{t\in[0;1]},\quad
\{(\boldsymbol\theta_2(t),\boldsymbol\varphi_2(t))\}_{t\in[0;1]}$$
oriented according to the parametrization, where~$\boldsymbol\theta_{1,2}$ and~$\boldsymbol\varphi_{1,2}$
are as in Proposition~\ref{octagon-prop}, are called the \emph{corner arcs} of~$P$
and denoted by~$\cabl P$, $\cabr P$,
$\catl P$, and $\catr P$,
respectively.

By reversing the order in a positive pack of rectangles one obtains a \emph{negative
pack of rectangles}. The orientation of corner arcs is reversed accordingly.

The least and the largest elements in~$P$ will be referred to as the \emph{first}
and the \emph{last} rectangles of~$P$ and denoted~$r_{\mathrm{min}}(P)$
and~$r_{\mathrm{max}}(P)$, respectively.
\end{defi}

Observe that a pack of rectangles does not carry any fixed parametrization. 
As we will now see, to represent
a pack of rectangles~$P$, it suffices to specify the corner arcs of~$P$,
indicate their orientations, and say which one is which.

\begin{prop}\label{pack-defined-by-arcs-prop}
Let~$P$ and~$P'$ be packs of rectangles such that three of the four corner
arcs of~$P$ coincide with the respective corner arcs of~$P'$. Then~$P=P'$.
\end{prop}

\begin{proof}
Due to symmetry, we may assume that~$\cabl P=\cabl{P'}$, $\cabr P=\cabr{P'}$, and~$\catl P=\catl{P'}$.
Let
$$P=\{[\boldsymbol\theta_1(t);\boldsymbol\theta_2(t)]\times[\boldsymbol\varphi_1(t);\boldsymbol\varphi_2(t)]\}_{t\in[0;1]},\quad
P'=\{[\boldsymbol\theta_1'(t);\boldsymbol\theta_2'(t)]\times[\boldsymbol\varphi_1'(t);\boldsymbol\varphi_2'(t)]\}_{t\in[0;1]}.$$
Since~$\cabl P=\cabl{P'}$, we can reparametrize~$P'$ so that the obtained parametrization of~$\cabl P'$ will coincide
with that of~$\cabl P$. This means that we may assume~$\boldsymbol\theta_1=\boldsymbol\theta_1'$ and~$\boldsymbol\varphi_1=\boldsymbol\varphi_1'$, without loss of generality.

Now~$\cabr P=\cabr {P'}$ reads
$$\{(\boldsymbol\theta_2(t),\boldsymbol\varphi_1(t))\}_{t\in[0;1]}=
\{(\boldsymbol\theta_2'(t),\boldsymbol\varphi_1(t))\}_{t\in[0;1]},$$
which implies~$\boldsymbol\theta_2=\boldsymbol\theta_2'$, since~$\boldsymbol\varphi_1$ is a monotone function.
Similarly, $\catl P=\catl{P'}$ implies~$\boldsymbol\varphi_2=\boldsymbol\varphi_2'$.
\end{proof}

Thus, to define a pack of rectangles, we don't need to specify all four corner arcs, since any of them
can be recovered from the other three.

Clearly, each corner arc of a pack of rectangles is a simple arc that has the form of a graph of a monotonic
function~$\varphi=\varphi(\theta)$.
We call such curves \emph{sloped arcs}. The next statement gives a necessary and sufficient condition for three
sloped arcs to be corner arcs of a pack of rectangles.

\begin{prop}
Let~$r_1$ and~$r_2$ be two compatible rectangles such that~$r_1\cap r_2$ is also a rectangle. Let~$r_0$
be a minimal rectangle containing~$r_1\cup r_2$. Then, for any three sloped oriented arcs~$\alpha_1$, $\alpha_2$, $\alpha_3$ starting
at three corners of~$r_1$ and ending at the respective corners of~$r_2$, such that~$\alpha_i\subset r_0$, $i=1,2,3$,
there is a unique pack of rectangles~$P$ for which these arcs are corner arcs, and~$r_1$, $r_2$
are the first and the last rectangles, respectively.
\end{prop}

\begin{proof}
Due to symmetry, it suffices to consider the case when~$\alpha_1$, $\alpha_2$, and~$\alpha_3$
start at~$\cabl{r_1}$, $\cabr{r_1}$, and~$\catl{r_1}$, respectively. We may also assume~$\theta_-(r_1)<\theta_-(r_2)$
(otherwise exchange~$r_1$ and~$r_2$, and reverse the ordering in~$P$ afterwards).
The configuration is shown in Figure~\ref{recover-pack-fig}.
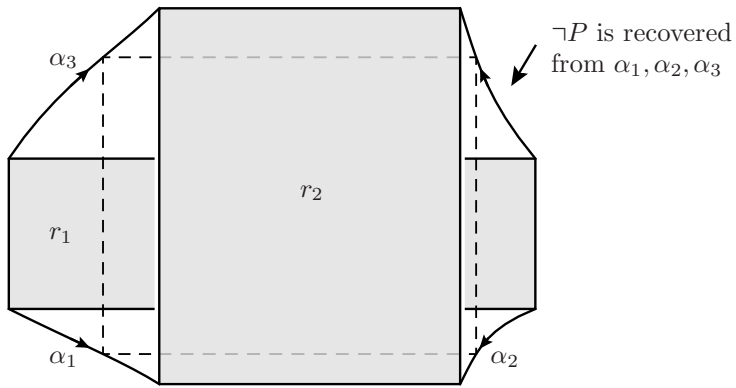
\begin{figure}[ht]
\includegraphics[scale=.5]{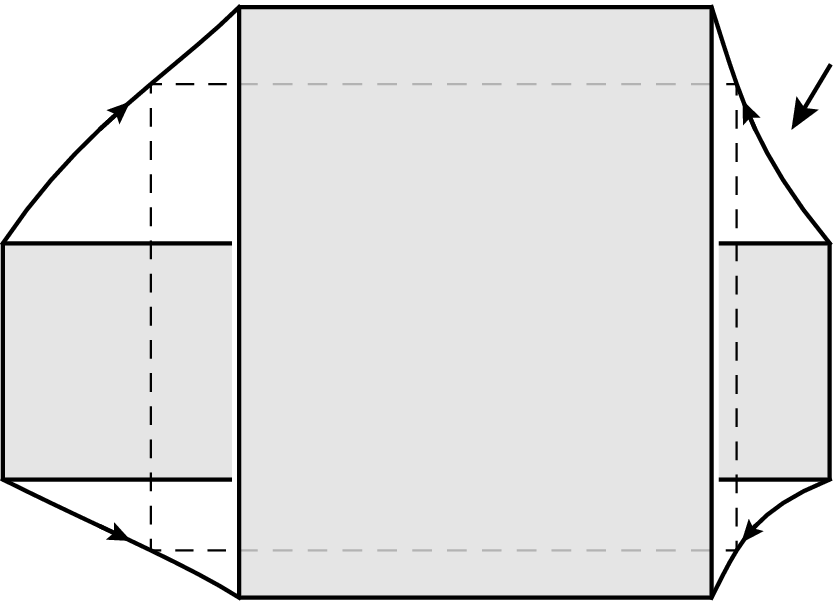}
\put(5, 126){\parbox{5cm}{$\catr P$ is recovered\\from~$\alpha_1,\alpha_2,\alpha_3$}}
\put(-18, 10){$\alpha_2$}
\put(-185, 10){$\alpha_1$}
\put(-185, 123){$\alpha_3$}
\put(-90, 73){$r_2$}
\put(-185, 57){$r_1$}
\caption{Recovering a pack of rectangles from three corner arcs}\label{recover-pack-fig}
\end{figure}

Since~$\alpha_1$ is a sloped arc connecting~$\cabl{r_1}$ with~$\cabl{r_2}$ and contained in~$r_0$,
it has the form of the graph of a strictly decreasing map from~$[\theta_-(r_1);\theta_-(r_2)]$
to~$[\varphi_-(r_2);\varphi_-(r_1)]$. Therefore, it admits a parametrization~$\alpha_1(t)=(\boldsymbol\theta_1(t),
\boldsymbol\varphi_1(t))$, $t\in[0;1]$, where~$\boldsymbol\theta_1$ and~$\boldsymbol\varphi_1$ are a strictly increasing
and a strictly decreasing functions, respectively.

The arc $\alpha_2$ is the graph of a strictly increasing function taking~$[\theta_+(r_2);\theta_+(r_1)]$
to~$[\varphi_-(r_2);\varphi_-(r_1)]$. We already have the function~$\boldsymbol\varphi_1$
that parametrizes the interval~$[\varphi_-(r_2);\varphi_-(r_1)]$. A parametrization of~$\alpha_2$
can be chosen in the form~$\alpha_2(t)=(\boldsymbol\theta_2(t),\boldsymbol\varphi_1(t))$.

Similarly, a parametrization of~$\alpha_3$ can be chosen in the form~$\alpha_3(t)=(\boldsymbol\theta_1(t),\boldsymbol\varphi_2(t))$.
The family~$\{[\boldsymbol\theta_1(t);\boldsymbol\theta_2(t)]\times[\boldsymbol\varphi_1(t);\boldsymbol\varphi_2(t)]\}_{t\in[0;1]}$
will then be the required pack of rectangles.
\end{proof}

\begin{defi}\label{compat-pack-def}
Two packs of rectangles~$P$ and~$P'$, say, are called \emph{compatible}
if any rectangle in~$P$ is compatible with any rectangle in~$P'$,
and whenever a corner arc~$\alpha$ of~$P$ meets a corner arc~$\alpha'$,
the union~$\alpha\cup\alpha'$ is a sloped arc, and the orientations
of~$\alpha$ and~$\alpha'$ agree at the intersection.

Figure~\ref{compat-pack-fig} shows examples of pairs of compatible packs of rectangles.
Each pack of rectangles is represented in the figure by the first and last rectangles and the corner arcs.
\begin{figure}[ht]
\begin{tabular}{ccc}
\includegraphics[scale=0.35]{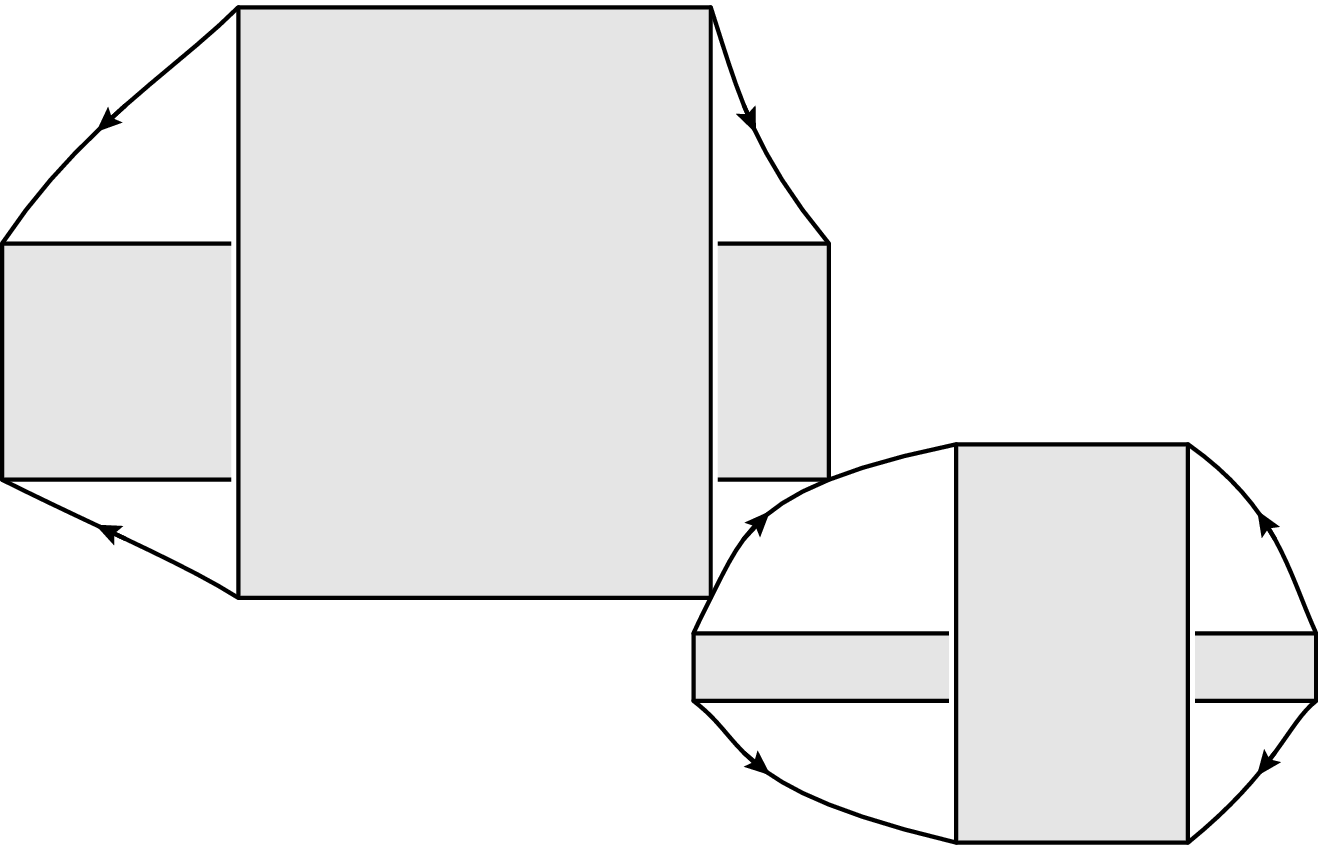}&\hbox to .5cm{}&
\includegraphics[scale=0.3]{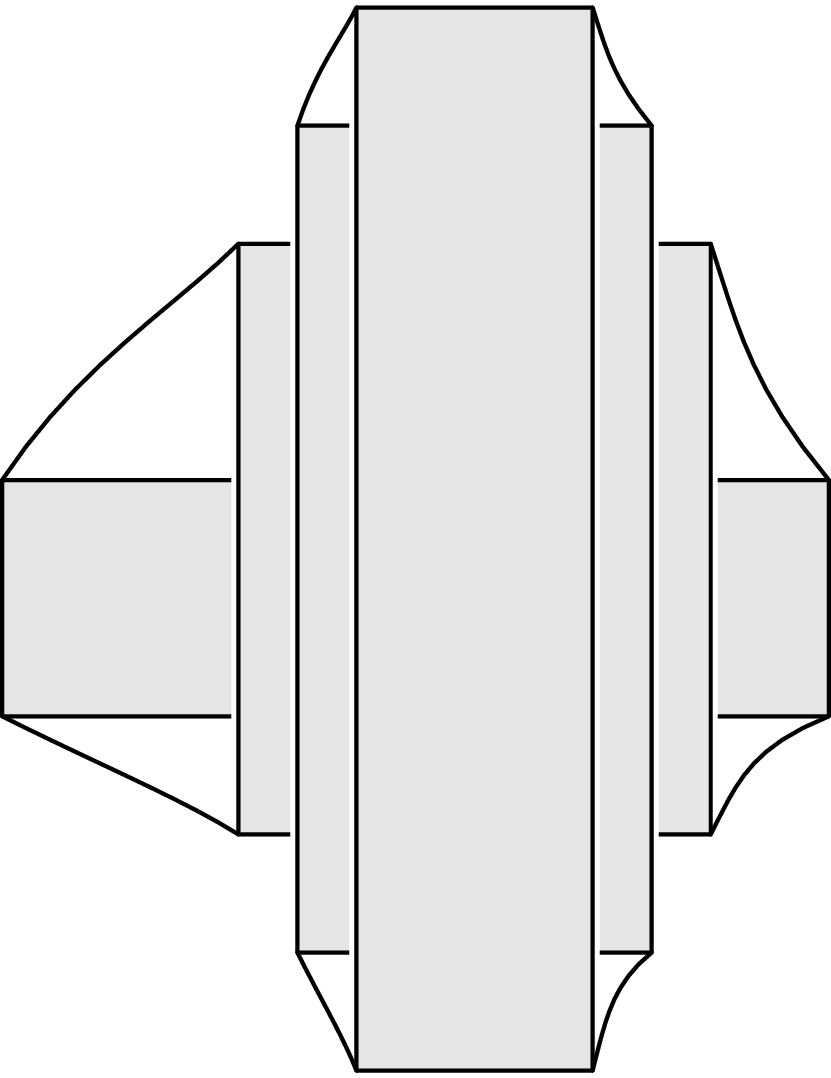}\\[1cm]
\includegraphics[scale=0.3]{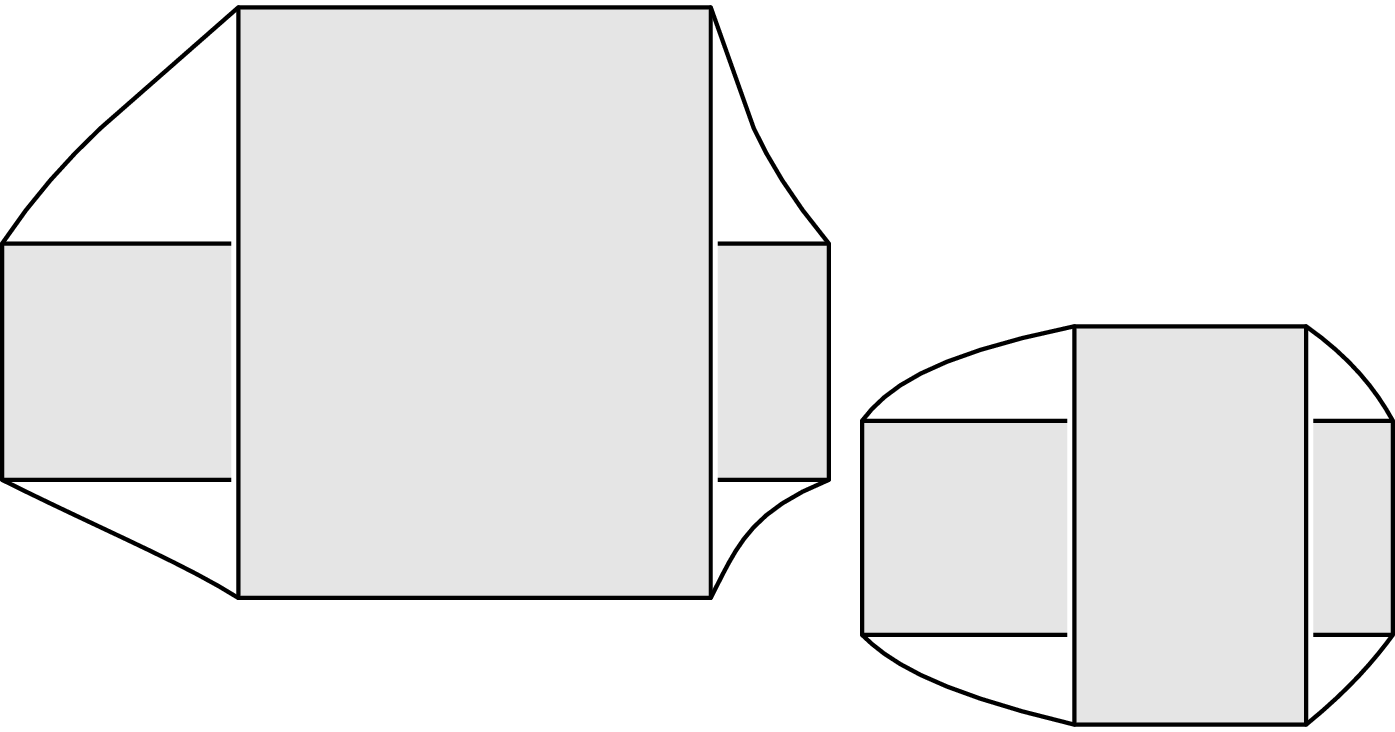}&&
\includegraphics[scale=0.3]{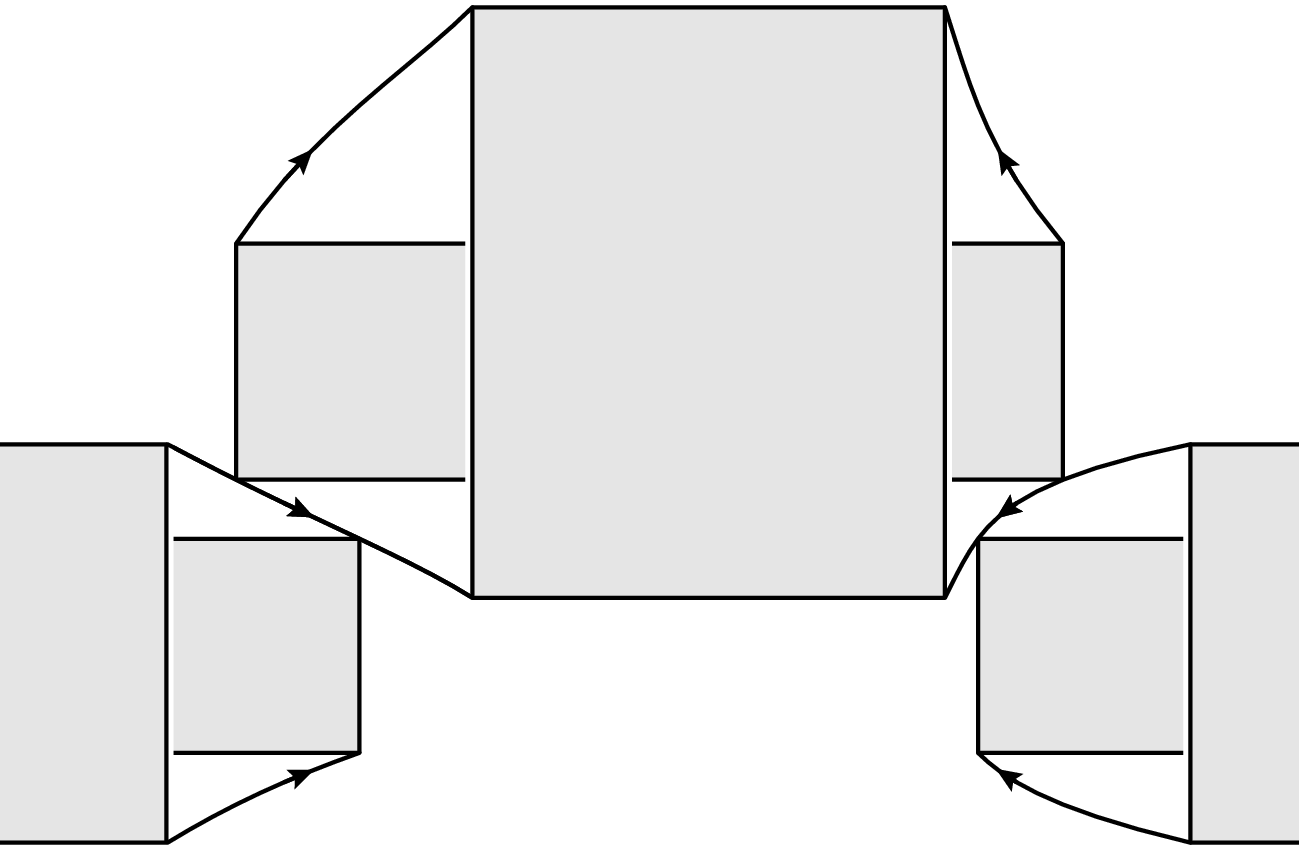}
\end{tabular}
\caption{Compatible packs of rectangles}\label{compat-pack-fig}
\end{figure}
\end{defi}

For a pack of rectangles~$P$, we denote by~$\widehat P$ the following subset of~$\mathbb S^3$:
$$\widehat P=\bigcup_{r\in P}\widehat r.$$
According to Proposition~\ref{octagon-prop}
the set~$\widehat P=\bigcup_{r\in P}\widehat r\subset\mathbb S^3$ is a homeomorphic image of the cube~$[0;1]^3$,
and the decomposition of this image into the discs~$\widehat r$, $r\in P$, is a foliation (even a fibration).
The ordering in~$P$ endows this foliation with a coorientation. We denote this cooriented foliation by~$\mathscr F(P)$.

Observe that the images of the four edges~$[0;1]\times\{0,1\}\times\{0,1\}$ of the cube~$[0;1]^3$
under the map~\eqref{cube-map-eq} depend only on the first and the last rectangles in~$P$.
Indeed, if~$P$ is a positive pack of rectangles, these images are the arcs
$$[\theta_-(r_{\mathrm{min}}(P));\theta_-(r_{\mathrm{max}}(P))],\quad
[\theta_+(r_{\mathrm{max}}(P));\theta_+(r_{\mathrm{min}}(P))]$$
of~$\mathbb S^1_{\tau=1}$ and
$$[\varphi_-(r_{\mathrm{max}}(P));\varphi_-(r_{\mathrm{min}}(P))],\quad
[\varphi_+(r_{\mathrm{min}}(P));\varphi_+(r_{\mathrm{max}}(P))]$$
of~$\mathbb S^1_{\tau=0}$. In the case of a negative pack, $r_{\mathrm{min}}$ and~$r_{\mathrm{max}}$ are exchanged.

These four arcs are transverse to the foliation~$\mathscr F(P)$, and
the corner arcs of~$P$ are the graphs of the holonomy maps between these transversals.

When two compatible packs of rectangles~$P$ and~$P'$ share a part of a corner arc,
then the respective `curved cubes'~$\widehat P$ and~$\widehat P'$ share
a part of a two-face and the foliations~$\mathscr F(P)$, $\mathscr F(P')$ agree
on the common part of these two-faces.

\section{Representing a Reeb component}\label{discs-sec}
In this section, we try to foliate the complement of an unknot by surfaces represented by a continuous family of rectangular diagrams,
and see the difficulties that arise in such an approach.

For convenience, we rescale the~$\theta$ and~$\varphi$ coordinates so that they take values in~$\mathbb R/\mathbb Z$
instead of~$\mathbb R/(2\pi\mathbb Z)$.
Let $R$ be the following rectangular diagram of the unknot:
$$R = \left\{\left(\frac{1}{4}, \frac{1}{4}\right), \left(\frac{1}{4}, \frac{3}{4}\right), \left(\frac{3}{4}, \frac{1}{4}\right), \left(\frac{3}{4}, \frac{3}{4}\right)\right\}.$$
We will now try to foliate the complement~$\mathbb S^3\setminus N_{1/16}(R)$, where~$N_t(R)$
is defined by~\eqref{N_t(R)-eq}, so that~$\Omega_{1/16}(R)$ is a fiber and~~$\mathbb S^3\setminus\overline{N_{1/16}(R)}$
is foliated by open discs represented by rectangular diagrams. This foliation would be nothing else but a Reeb component.

We are trying to construct a foliation which is tangent to
the boundary instead of being transverse to it because constructing a foliation transverse to the boundary
would involve surfaces with non-empty boundary, which we want to avoid
for the following two reasons. First, the boundary of a surface represented by a rectangular
diagram is typically not $C^1$-smooth. Second, at every boundary point, such a surface
must be tangent either to the standard contact structure or to its mirror image, which
is too restrictive.

So, instead of dealing with surfaces with boundary, we consider open surfaces that are wound
on a tubular neighborhood of the link
as shown in Figure~\ref{wound_surface}. The leaves of the foliations discussed below will always
behave like this.

\begin{figure}[h]
    \centering
    \includegraphics[scale=0.25]{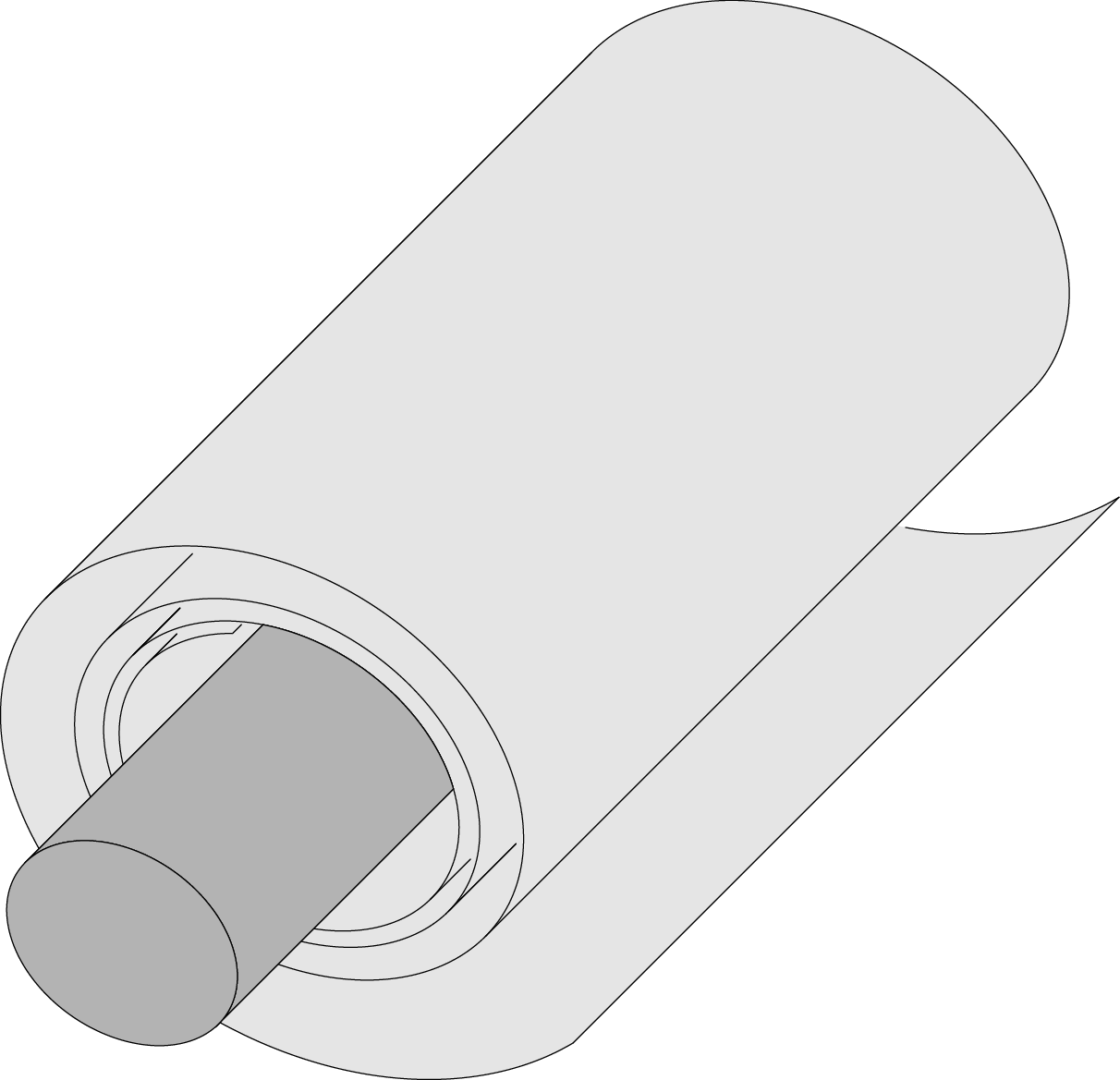}
    \caption{}
    \label{wound_surface}
\end{figure}

Let~$\Pi$ be the collection~$\{r_0,r_1,r_2,\ldots\}$ of the following rectangles:
\begin{align*}
    r_0 &= \left[\frac38; \frac58\right] \times\left[\frac18; \frac78\right],\\
    r_{8k-7} &= \left[\frac{3}{16} - \frac{1}{2^{k + 3}}; \frac{5}{16} + \frac{1}{2^{k + 3}}\right] \times\left[\frac{13}{16} + \frac{1}{2^{k + 3}}; \frac{3}{16} - \frac{1}{2^{k + 3}}\right],\\
    r_{8k-6} &= \left[\frac{11}{16} - \frac{1}{2^{k + 3}}; \frac{13}{16} + \frac{1}{2^{k + 3}}\right] \times\left[\frac{13}{16} + \frac{1}{2^{k + 3}}; \frac{3}{16} - \frac{1}{2^{k + 3}}\right],\\
    r_{8k-5} &= \left[\frac{13}{16} + \frac{1}{2^{k + 3}}; \frac{3}{16} - \frac{1}{2^{k + 3}}\right] \times\left[\frac{3}{16} - \frac{1}{2^{k + 3}}; \frac{5}{16} + \frac{1}{2^{k + 3}}\right],\\
    r_{8k-4} &= \left[\frac{13}{16} + \frac{1}{2^{k + 3}}; \frac{3}{16} - \frac{1}{2^{k + 3}}\right] \times\left[\frac{11}{16} - \frac{1}{2^{k + 3}}; \frac{13}{16} + \frac{1}{2^{k + 3}}\right],\\
    r_{8k-3} &= \left[\frac{3}{16} - \frac{1}{2^{k + 3}}; \frac{5}{16} + \frac{1}{2^{k + 4}}\right] \times\left[\frac{5}{16} + \frac{1}{2^{k + 3}}; \frac{11}{16} - \frac{1}{2^{k + 3}}\right],\\
    r_{8k-2} &= \left[\frac{11}{16} - \frac{1}{2^{k + 4}}; \frac{13}{16} + \frac{1}{2^{k + 3}}\right] \times\left[\frac{5}{16} + \frac{1}{2^{k + 3}}; \frac{11}{16} - \frac{1}{2^{k + 3}}\right],\\
    r_{8k-1} &= \left[\frac{5}{16} + \frac{1}{2^{k + 4}}; \frac{11}{16} - \frac{1}{2^{k + 4}}\right] \times\left[\frac{3}{16} - \frac{1}{2^{k + 4}}; \frac{5}{16} + \frac{1}{2^{k + 3}}\right],\\
    r_{8k}   &= \left[\frac{5}{16} + \frac{1}{2^{k + 4}}; \frac{11}{16} - \frac{1}{2^{k + 4}}\right] \times\left[\frac{11}{16} - \frac{1}{2^{k + 3}}; \frac{13}{16} + \frac{1}{2^{k + 4}}\right],
\end{align*}

where~$k\in\mathbb N$.
    Define a function~$\epsilon:\Pi\rightarrow\{-1,1\}$ as follows:
    \begin{equation*}
    \epsilon(r_i) = \left\{
        \begin{aligned}
            -1& \text{ if } i\equiv 1,2,5,\text{ or }6\ (\mathrm{mod}\ 8),\\
            1& \text{ if } i\equiv 0,3,4,\text{ or }7\ (\mathrm{mod}\ 8).
        \end{aligned}
    \right.
    \end{equation*}
    
    One can verify that the pair $(\Pi, \epsilon)$ is an oriented rectangular diagram of a surface.
The diagram $\Pi$ and the corresponding tiling of $\widehat \Pi$ are shown in the Fig~\ref{fibred_unknot-fig}.
\begin{figure}[ht]
            \includegraphics[scale = 7]{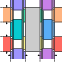}
           \put(-182, -10){$\frac1 8$}
            \put(-132, -10){$\frac3 8$}
            \put(-83, -10){$\frac5 8$}
            \put(-33, -10){$\frac7 8$}
            \put(-223, 36){$\frac1 8$}
            \put(-223, 86){$\frac3 8$}
            \put(-223, 136){$\frac5 8$}
            \put(-223, 186){$\frac7 8$}
            \put(-107, 112){$r_0$}
            \put(-195, 60){$r_{3}$}
            \put(-195, 160){$r_{4}$}    
            \put(-136, 60){\tiny $r_{7}$}
            \put(-136, 160){\tiny $r_{8}$}
            \put(-180, 110){\tiny $r_{5}$}
            \put(-35, 110){\tiny $r_{6}$}
            \put(-136, 200){\tiny $r_{1}$}
            \put(-80, 200){\tiny $r_{2}$}
            \hskip1cm
            \includegraphics[scale=0.6]{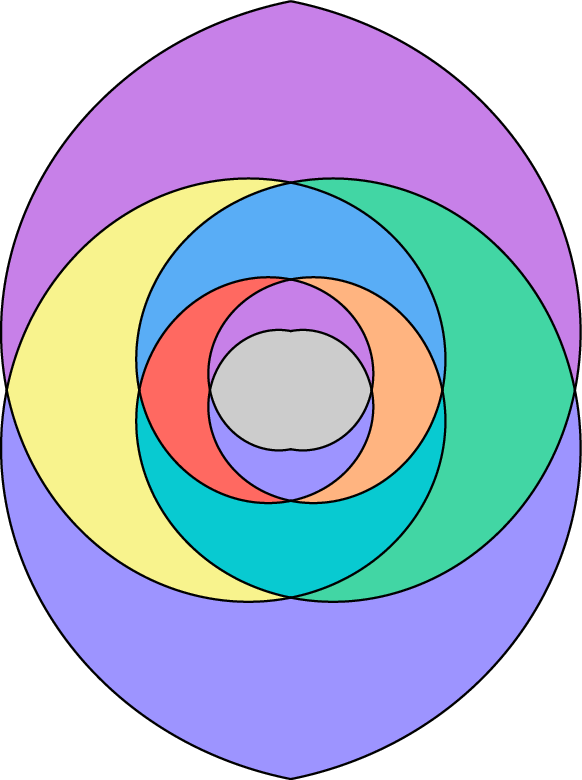}
           \put(-88, 110){$\widehat r_0$}
            \put(-88, 88){\scriptsize $\widehat r_{1}$}
            \put(-88, 135){\scriptsize $\widehat r_{2}$}
            \put(-123, 110){ $\widehat r_{3}$}
            \put(-60, 110){ $\widehat r_{4}$}
            \put(-88, 65){$\widehat r_{5}$}
            \put(-88, 152){$\widehat r_{6}$}
            \put(-26, 110){$\widehat r_{8}$}
            \put(-150, 110){$\widehat r_{7}$}
            \put(-88, 23){$\widehat r_{9}$}
            \put(-88, 195){$\widehat r_{10}$}
        \caption{The rectangular diagram $\Pi$ and the associated tiling of the surface $\widehat\Pi$}\label{fibred_unknot-fig}
    \end{figure}

\begin{figure}
\begin{tabular}{ccc}
    \includegraphics[scale = 6]{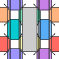}
&\hbox to1cm{}&
    \includegraphics[scale = 6]{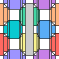}
\\
    $\Pi$&&$f_1(\Pi)$
\end{tabular}
    \caption{Trajectories of vertices of $\Pi$ under the deformation $f$, and the final rectangular diagram $f_1(\Pi)$}\label{deform-traj-fig}
\end{figure}

    For any $k\in\mathbb N$, denote by~$A_k$ the following collection of eight rectangles:
$$\bigcup_{i =8k-7}^{8k}  \{r_i\}.$$
One can see that~$\widehat A_k$ is an annulus for all~$k\in\mathbb N$,
and the surface~$\widehat D_k$, where
$$D_k = \{r_0\} \cup \bigcup_{i=1}^kA_i,$$
is a two-disc. Hence, the surface $\widehat\Pi = \bigcup_{n=1}^\infty\widehat D_n$ is an open two-disk.
One can also see that the annuli~$\widehat A_1,\widehat A_2,\widehat A_3,\ldots$ accumulate to the tube~$\widehat\Omega_{1/16}(R)$.
    
It is a direct check that
$$\begin{aligned}
    \Theta_-(\Pi, \epsilon) &= \Phi_-(\Pi, \epsilon)=\left\{\frac{3}{16}-\frac{1}{2^k}, \frac{11}{16}-\frac{1}{2^k}\right\}_{k\geqslant4},\\
    \Theta_+(\Pi, \epsilon) &= \Phi_+(\Pi, \epsilon)=\left\{\frac{5}{16}+\frac{1}{2^k}, \frac{13}{16} + \frac{1}{2^k}\right\}_{k\geqslant4}.    
    \end{aligned}$$
Define a positive deformation $f = (f^0, f^1)$ of $\Pi$ by
$$
    f^0=f^1 : \quad
        \left(\frac{4\pm1}{16}\pm\frac{1}{2^{k}}, t\right)\mapsto \frac{4\pm1}{16}\pm\frac{1}{2^{k - t}},\quad
        \left(\frac{12\pm1}{16}\pm\frac{1}{2^{k}}, t\right)\mapsto \frac{12\pm1}{16}\pm\frac{1}{2^{k - t}}, \quad  k=4, 5,\ldots.
$$
Figure~\ref{deform-traj-fig} demonstrates how the vertices of~$\Pi$ move under this deformation.
For any $i\geqslant9$, the image $f_1(r_{i})$ of~$r_i$ under this deformation  is the rectangle $r_{i-8}$.
So, each~$A_k$, $k\geqslant2$, is deformed to~$A_{k-1}$.

It follows from Proposition~\ref{thickened_surface-prop} that there is an immersion~$I_f:\widehat\Pi\times[0;1]\rightarrow\mathbb S^3$
that takes~$\widehat\Pi\times\{t\}$ to~$\widehat{f_t(\Pi)}$ for any~$t\in[0;1]$. Moreover, for any proper
subinterval~$[a;b]$ of~$[0;1]$, the restriction of this immersion to~$\widehat\Pi\times[a;b]$ is an embedding.

If we had~$\Pi=f_1(\Pi)$, we would be able to identify the discs~$\widehat\Pi\times\{0\}$ with~$\widehat\Pi\times\{1\}$
so as to make~$I_f$ an embedding~$\widehat\Pi\times\mathbb S^1\rightarrow\mathbb S^3$. But the equality~$\Pi=f_1(\Pi)$ does
not hold, and the image of~$I_f$ is not the entire complement of~$\overline{N_{1/16}(R)}$.
However, the difference between~$\Pi$ and~$f_1(\Pi)$ is not that big. Namely,
$$\Pi\setminus f_1(\Pi)=D_0,\quad
f_1(\Pi)\setminus\Pi=f_1(D_1),\quad
\text{and}\quad
\Pi\cap f_1(\Pi)=\bigcup_{k=1}^\infty A_k.$$

Thus, the coincident part of the open two-discs~$\widehat\Pi$ and~$\widehat{f_1(\Pi)}$ is the half-open
annulus~$\bigcup_{k=1}^\infty\widehat A_k$, whose boundary curve is glued up by
the disc~$\widehat D_0$ in~$\widehat\Pi$ and by~$\widehat{f_1(D_1)}$ in~$\widehat{f_1(\Pi)}$.
The two-discs~$\widehat D_0=\widehat r_0$ and~$\widehat{f_1(D_1)}$ enclose
an open three-ball disjoint from~$\widehat R$, which we denote by~$C$.

Thus, we have almost obtained a foliation of~$\mathbb S^3\setminus N_{1/16}(R)$ in which all leaves
exept~$\partial N_{1/16}(R)$ are open discs of the form~$\widehat{f_t(\Pi)}$,
and the only defect of the construction is an open three-ball~$C$, which is disjoint from these surfaces.
However, if we remove~$C$ from~$\mathbb S^3$ and
identify the discs~$\widehat D_0$ and~$\widehat{f_1(D_1)}$
by a homeomorphism identical on their common boundary, we will still get~$\mathbb S^3$, and
the discs~$\widehat{f_t(\Pi)}$ will become leaves of a genuine foliation in the complement of~$N_{1/16}(R)$.

This simple example demonstrates the difficulty with representing \emph{all} leaves of a foliation by rectangular diagrams---some
open subset of~$\mathbb S^3$ remains empty. However, the empty regions will have very special form that will allow
to deflate them without disturbing the topology of the three-sphere. A formal description
of those regions as well as a way to detect them by means of rectangular diagrams is the matter of the next section.

\section{Cavities}\label{cavit-sec}

\begin{defi}\label{cavity-def}
Let~$C\subset\mathbb S^3$ be an open subset of the three-sphere such that there exist a compact quasi-surface~$F\subset\mathbb S^3$,
a $1$-subcomplex~$\Gamma\subset F$ consisting of finitely many smooth arcs,
and a smooth isotopy~$\Psi:F\times[0;1]\rightarrow\mathbb S^3$ such that the following holds:
\begin{enumerate}
\item
$\Gamma\supset\partial F$;
\item
the restriction of~$\Psi$ to~$(F\setminus\Gamma)\times[0;1]$ is an embedding;
\item
$\Psi(p,t)=\Psi(p,0)$ for all~$(p,t)\in\Gamma\times[0;1]$, and the differential of~$\Psi(\cdot,t)$ is identical on~$\Gamma$;
\item
$C=\Psi\bigl((F\setminus\Gamma)\times(0;1)\bigr)$.
\end{enumerate}
Then the subset~$C$ is called \emph{a cavity}, and the projection map~$\mathfrak d_C:\mathbb S^3\setminus C\rightarrow(\mathbb S^3\setminus C)/{\sim}$,
where~$\sim$ is the following equivalence relation:
$$p\sim q\quad\text{iff}\quad p=q\text{ or }p\in F\text{ and }q=\Psi(p,1),$$
is referred to as \emph{deflating} the cavity~$C$.

If, additionally, $F_1,F_2\subset\mathbb S^3$ are two oriented quasi-surfaces such that
\begin{enumerate}
\item
$F_1\setminus F_2=F\setminus\Gamma$;
\item
$F_2\setminus F_1=\Psi((F\setminus\Gamma)\times\{1\})$;
\item
the positive normal to~$F_1$ at any point~$p\in F_1\setminus F_2$ points inward~$C$;
\item
the positive normal to~$F_2$ at any point~$p\in F_2\setminus F_1$ points outward~$C$
\end{enumerate}
we say that~$F_2$ is obtained from~$F_1$ by \emph{a positive leap over~$C$}.
\end{defi}

Observe that, in this definition, the deflating map~$\mathfrak d_C$ is not determined by~$C$ itself as the definition involves an isotopy~$\Psi$.
However, the obtained space~$(\mathbb S^3\setminus C)/{\sim}$ is clearly homeomorphic to~$\mathbb S^3$,
so~$\mathfrak d_C$ can be viewed as a map from~$\mathbb S^3\setminus C$ to~$\mathbb S^3$ defined up to
postcomposing with a self-homemorphism of~$\mathbb S^3$. Moreover, clearly, the arbitrariness in the definition
allows to make~$\mathfrak d_C$ identical on~$F$ and smooth. These properties of~$\mathfrak d_C$ will be assumed to hold in the sequel.

\begin{lemm}\label{h1...hk-lem}
Let~$F_0,F_1,\ldots,F_m$ be connected oriented quasi-surfaces in~$\mathbb S^3$, and let~$C_1,\ldots,C_m$ be cavities such that,
for each~$i=1,\ldots,m$,
\begin{enumerate}
\item
the quasi-surface~$F_i$ is obtained from~$F_{i-1}$ by a positive leap over~$C_i$\emph;
\item
$F_i\cap F_0\subset F_j$ for~$j<i$\emph;
\item
$C_i\cap F_0=\varnothing$.
\end{enumerate}
Denote the interior of the union~$\overline C_1\cup\ldots\cup\overline C_m$ by~$C$.
Then~$C$ is a cavity, and~$F_m$ is obtained from~$F_0$ by a positive leap over~$C$.
\end{lemm}
\begin{figure}[h]
    \begin{tabular}{c}
    \includegraphics[scale=1]{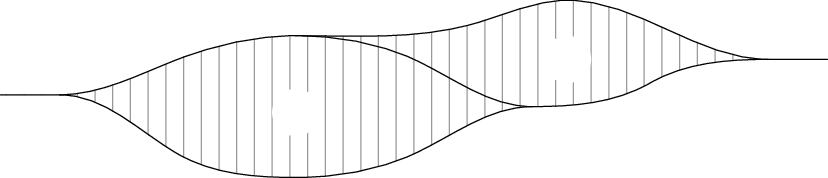}
    \put(-260, 29){$C_1$}
    \put(-130, 55){$C_2$}
    \put(-175, 15){$F_0$}
    \put(-175, 50){$F_1$}
    \put(-175,81){$F_2$}
    \\
    $\downarrow$\\
    \includegraphics[scale=1]{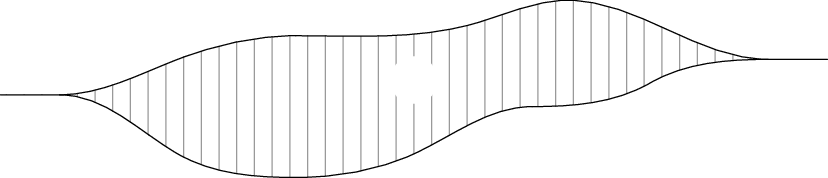}
    \put(-205, 44){$C$}  
    \end{tabular}
    \caption{Two cavities merged}
    \label{merged_cavities}
\end{figure}
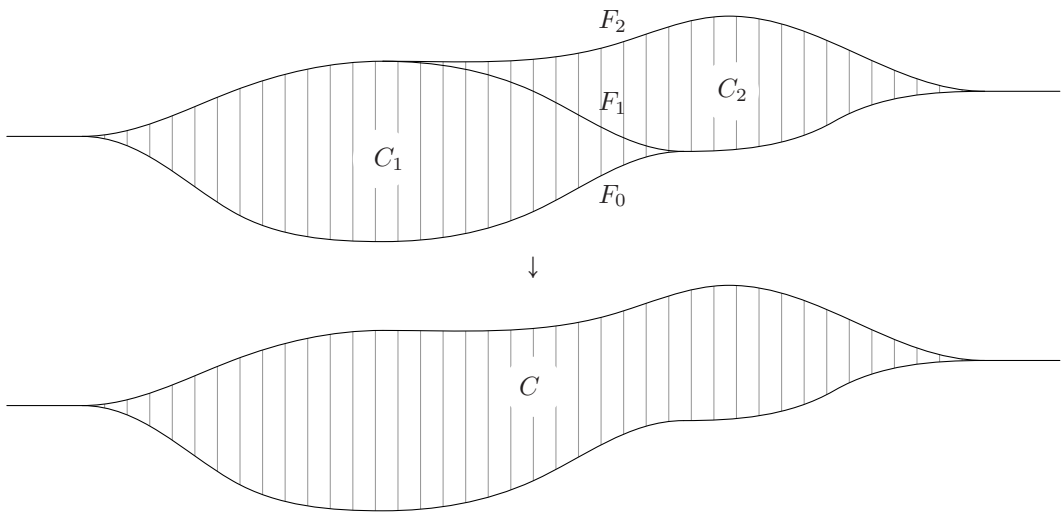

\begin{proof}
The statement is trivial in the case~$m=1$. If we prove it for~$m=2$, then the general case will follow by induction in~$m$.
So, we assume~$m=2$ in the sequel. This case is illustrated in Figure~\ref{merged_cavities}.

Let
$$\Psi_1:\overline{F_0\setminus F_1}\times[0;1]\rightarrow\mathbb S^3\quad
\text{and}\quad
\Psi_2:\overline{F_1\setminus F_2}\times[0;1]\rightarrow\mathbb S^3$$
be isotopies satisfying the requirements from Definition~\ref{cavity-def} for the cavities~$C_1$ and~$C_2$,
respectively. Extend~$\Psi_1$ and~$\Psi_2$ to the entire surfaces~$F_0$ and~$F_1$ by
$$\Psi_i(x,t)=x\quad\text{for all}\ x\in F_{i-1}\setminus F_i,\ t\in[0;1],\ i\in\{1,2\},$$
and define~$\Psi:F_0\times[0;1]\rightarrow\mathbb S^3$ as follows:
$$\Psi(x,t)=\left\{\begin{aligned}
&\Psi_1(x,2t)&\text{if }t\leqslant1/2,\\
&\Psi_2(\Psi_1(x,1),2t-1)&\text{if }t>1/2.
\end{aligned}\right.$$

Since~$C_1$ and~$C_2$ approach~$F_1$ from opposite sides and~$C_2\cap F_0=\varnothing$,
we have~$C_2\cap C_1=\varnothing$, and the map~$\Psi$ restricted to~$\overline{F_0\setminus F_2}\times[0;1]$ satisfies
the requirements of Definition~\ref{cavity-def} except that it is not an embedding on~$(F_0\setminus F_2)\times[0;1]$.
However, for every~$x\in F_0\setminus F_2$, the map~$\Psi$ takes the interval~$x\times[0;1]$
to its image monotonically, and the arcs~$\Psi(x\times[0;1])$ and~$\Psi(x'\times[0;1])$ are disjoint provided~$x\ne x'$.
Therefore, by a small perturbation of~$\Psi$ we can make it an embedding on~$(F_0\setminus F_2)\times[0;1]$
keeping it fixed on~$\overline{F_0\setminus F_2}\times\{0,1\}$.
\end{proof}

Now we define two types of modifications of rectangular diagrams of (quasi-)surfaces
that result in the corresponding surfaces leaping over cavities. These modifications
are called bubble moves and flypes, and the respective cavity is an open three-ball in each case.
These moves were introduced in~\cite{Distinguishing,Basic_moves} (for finite rectangular diagrams of surfaces without orientation).

\begin{defi}\label{bubble_move-def}
Let~$\Pi$ and~$\Pi'$ be rectangular diagrams of a quasi-surface, and let~$r,r_1,r_2,r_3$ be rectangles
such that
\begin{enumerate}
\item
$\Pi\setminus\Pi'=\{r\}$, $\Pi'\setminus\Pi=\{r_1,r_2,r_3\}$;
\item
$r_1\cup r_2\cup r_3=\overline{r\mathbin\triangle S}$, where~$S\subset\mathbb T^2$ is an annulus of the form
either~$[\theta_1;\theta_2]\times\mathbb S^1$
or~$\mathbb S^1\times[\varphi_1;\varphi_2]$ whose intersection with~$r$ is a smaller rectangle, and~$\triangle$ denotes the symmetric difference;
\item~$r_1,r_2\subset r$, $r_3\subset S$.
\end{enumerate}
Let also~$\epsilon$ and~$\epsilon'$ be orientations of~$\Pi$ and~$\Pi'$, respectively that agree on~$\Pi\cap\Pi'$
such that~$\epsilon(r)=\epsilon'(r_1)=\epsilon'(r_2)$.
Then we say that~$(\Pi',\epsilon')$ is obtained from~$(\Pi,\epsilon)$ (or~$\Pi'$ from~$\Pi$) by \emph{a bubble creation move},
and~$(\Pi,\epsilon)$ from~$(\Pi',\epsilon')$ (or~$\Pi$ from~$\Pi'$) by \emph{a bubble reduction move}. This is illustrated
in Figure~\ref{fig:Bubble_move}.
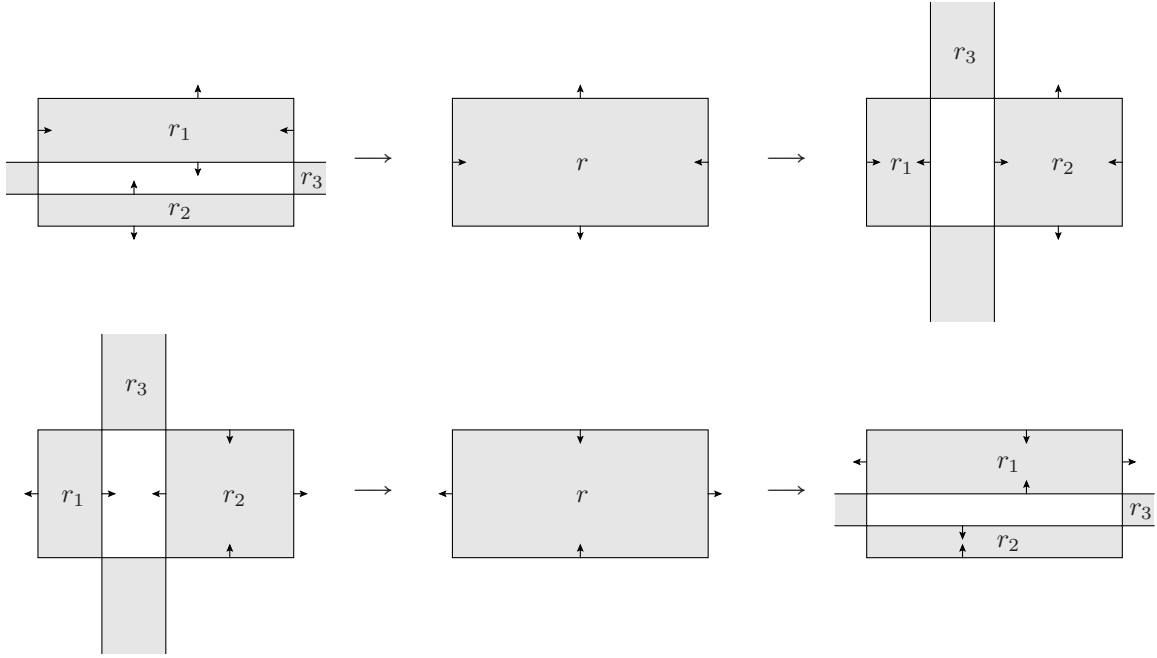
\begin{figure}[ht]
\begin{tabular}{ccccc}
\includegraphics[scale=0.85]{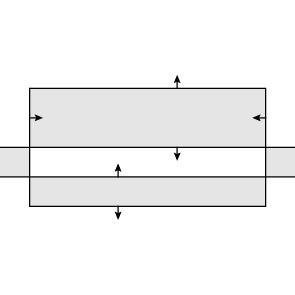}
\put(-60, 71){$r_1$}
\put(-10, 53){$r_3$}
\put(-60, 41){$r_2$}
&\raisebox{60pt}{$\longrightarrow$}
&\includegraphics[scale=0.85]{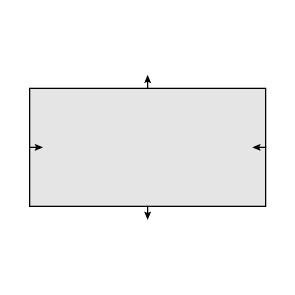}
\put(-62, 58){$r$}
&\raisebox{60pt}{$\longrightarrow$}
&\includegraphics[scale=0.85]{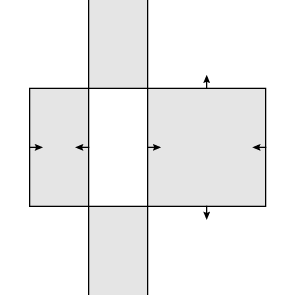}
\put(-100, 58){$r_1$}
\put(-39, 58){$r_2$}
\put(-76, 100){$r_3$}
\\
\includegraphics[scale=0.85]{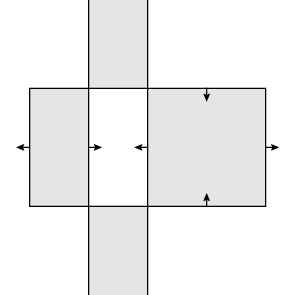}
\put(-100, 58){$r_1$}
\put(-39, 58){$r_2$}
\put(-76, 100){$r_3$}
&\raisebox{60pt}{$\longrightarrow$}
&\includegraphics[scale=0.85]{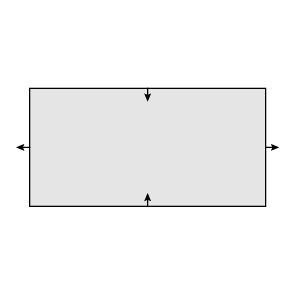}
\put(-62, 58){$r$}
&\raisebox{60pt}{$\longrightarrow$}
&\includegraphics[scale=0.85]{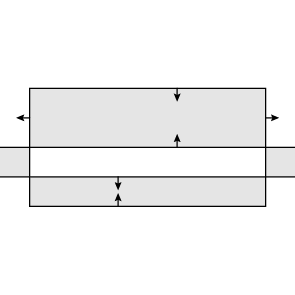}
\put(-60, 71){$r_1$}
\put(-10, 53){$r_3$}
\put(-60, 41){$r_2$}
\end{tabular}
    \caption{Positive bubble moves}
    \label{fig:Bubble_move}
\end{figure}

Let~$\delta$ be~$+1$ if~$S$ has the form~$[\theta_1;\theta_2]\times\mathbb S^1$,
and~$-1$ otherwise.
The bubble creation move~$(\Pi,\epsilon)\mapsto(\Pi',\epsilon')$ is called~\emph{positive} if~$\epsilon(r)=\delta$.
Otherwise, we call positive the bubble reduction move~$(\Pi',\epsilon')\mapsto(\Pi,\epsilon)$.
\end{defi}

\begin{lemm}\label{bubble_cavity-lem}
Let~$(\Pi_1,\epsilon_1)\mapsto(\Pi_2,\epsilon_2)$ be a positive bubble creation or reduction move.
Then~$\widehat\Pi_2$ is obtained from~$\widehat\Pi_1$ by a positive leap over a cavity
homeomorphic to an open three-ball.
\end{lemm}

\begin{proof}
Let~$\Pi,\Pi',\epsilon,\epsilon',r,r_1,r_2,r_3,S$, and~$\delta$ be as in Definition~\ref{bubble_move-def}.
It is a direct check that~$D=\{r\}$ and~$D'=\{r_1,r_2,r_3\}$
are rectangular diagrams representing two-discs, $\widehat D\cong\widehat D'\cong\mathbb D^2$.
These discs have common boundary, and their oriented tangent planes coincide
at every point of~$\partial\widehat D=\partial\widehat D'$.

Let~$f$ and~$f'$ be positive deformations of~$D$ and~$D'$
with respect to orientations inherited from~$\Pi$ and~$\Pi'$, respectively.
Suppose that~$\epsilon(r)=\delta$. Then it is another direct check that, for small enough~$t>0$
the rectangle~$f_t(r)$ is incompatible with~$r_i$, $i=1,2,3$, whereas~$f'_t(r_i)$ is
compatible and does not share a corner with~$r$, $i=1,2,3$. This means that a small push of~$\widehat D$ in the positive
transverse direction defined by the orientation~$\epsilon$ makes~$\widehat D$ intersect~$\widehat D'$ non-trivially,
whereas a small push of~$\widehat D'$ in the positive transverse direction makes it disjoint from~$\widehat D$.
This implies that~$\widehat D$ and~$\widehat D'$ have disjoint interiors and enclose an open three-ball~$C$
such that~$\widehat{f_t(D)}\cap C\ne\varnothing$ and~$\widehat{f'_t(D')}\cap C=\varnothing$ for small enough~$t>0$.

Let~$v$ be an arbitrary point in~$\mathbb T^2\setminus(r\cup S)$. Then the arc~$\widehat v$ is disjoint from~$\widehat D$
and~$\widehat D'$. Therefore, it is either contained in~$C$ or disjoint from~$C$.
One of the circles~$\mathbb S^1_{\tau=0}$ and~$\mathbb S^1_{\tau=1}$ has only two
intersection points with~$\widehat D\cup\widehat D'$ and no intersection with~$C$. Therefore,
we have~$\widehat v\cap C=\varnothing$.

This implies that, for any~$r'\in\Pi\cap\Pi'$, the tile~$\widehat r'$ is disjoint from~$C$.
Indeed, by hypothesis, $r'$ is compatible with all rectangles in~$D$ and~$D'$,
so~the interior of~$\widehat r'$ is disjoint from~$\widehat D\cup\widehat D'$.
Each vertex of~$r'$ is either outside~$r\cup S$ or coincides with a vertex of~$r$,
which implies~$\partial\widehat r'\cap C=\varnothing$. The inclusion~$\widehat r'\setminus\partial\widehat r'\subset C$
is impossible, since we have~$\widehat v\cap\widehat r'\ne\varnothing$,
$\widehat v\cap C=\varnothing$ provided that~$v\in r'\setminus(r\cup S)$.

We see that the bubble creation~$(\Pi,\epsilon)\mapsto(\Pi',\epsilon')$ is a positive leap over the cavity~$C$, which is an open three-ball.
The roles of the surface~$F$ and the graph~$\Gamma$ in Definition~\ref{cavity-def} are played by~$\widehat D$
and~$\partial\widehat D$, respectively, and the isotopy~$\Psi$ brings~$\widehat D$ to~$\widehat D'$.

The case~$\epsilon(r)=-\delta$ is similar and left to the reader. In this case, the bubble reduction~$(\Pi',\epsilon')\mapsto(\Pi,\epsilon)$
is a positive leap over the same~$C$.
\end{proof}

\begin{exam}
Consider again the construction from Section~\ref{discs-sec}. One can see that the original diagram~$\Pi$
can be obtained from the diagram~$f_1(\Pi)$ by the following sequence of four positive bubble reduction moves (see Figure~\ref{unknot_bubbles-fig}):

$$\begin{aligned}
f_1(\Pi)&\mapsto\Pi'\phantom{''} = \bigl(f_1(\Pi)\setminus\{f_1(r_0), f_1(r_1), f_1(r_2)\}\bigr)\cup\{r_{\mathrm{I}}\},\\
\Pi'&\mapsto\Pi''\phantom' = \bigl(\Pi'\setminus\{r_{\mathrm I}, f_1(r_3), f_1(r_4)\}\bigr)\cup\{r_{\mathrm{II}}\},\\
\Pi''&\mapsto\Pi''' = \bigl(\Pi''\setminus\{r_{\mathrm{II}}, f_1(r_5), f_1(r_6)\}\bigr)\cup\{r_{\mathrm{III}}\},\\
\Pi'''&\mapsto\Pi\phantom{'''} = \bigl(\Pi'''\setminus\{r_{\mathrm{III}}, f_1(r_7),f_1(r_8)\}\bigr)\cup\{r_0\},
\end{aligned}$$
where
$$\begin{aligned}
r_{\mathrm I}&=\left[\frac{1}{16}; \frac{15}{16}\right]\times\left[\frac{15}{16}; \frac{1}{16}\right],\\
r_{\mathrm{II}}&=\left[\frac{15}{16}; \frac{1}{16}\right]\times\left[\frac{9}{16}; \frac{7}{16}\right],\\
r_{\mathrm{III}}&=\left[\frac{5}{8}; \frac{3}{8}\right]\times\left[\frac{7}{16}; \frac{9}{16}\right].
\end{aligned}$$

\end{exam}

\begin{figure}[ht]
    \begin{tabular}{ccccc}
        \includegraphics[scale=4.2]{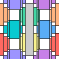}\put(-76,-13){$f_1(\Pi)$}&
       \raisebox{60pt}{$\longrightarrow$}&
        \includegraphics[scale=4.2]{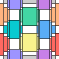}\put(-66,-13){$\Pi'$}&
        \raisebox{60pt}{$\longrightarrow$}&
        \includegraphics[scale=4.2]{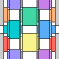}\put(-66,-13){$\Pi''$}\\
        &&&&\rotatebox{90}{$\longleftarrow$}\\
        &&\includegraphics[scale=4.2]{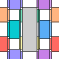}\put(-66,-13){$\Pi$}&
        \raisebox{60pt}{$\longleftarrow$}&
        \includegraphics[scale=4.2]{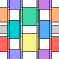}\put(-66,-13){$\Pi'''$}
    \end{tabular}
    \caption{Transforming $f_1(\Pi)$ to $\Pi$}\label{unknot_bubbles-fig}
\end{figure}

\begin{defi}\label{flype_move-def}
    Let $\Pi$ and $\Pi'$ be rectangular diagrams of a quasi-surface, and let $r_1$, $r_2$, $r_3$, $r_4$,
    $r_*$, $r_1'$, $r_2'$, $r_3'$, $r_4'$, $r_*'$ be rectangles such that
    \begin{enumerate}
        \item $\Pi\setminus\Pi' = \{r_1, r_2, r_3, r_4\}$, $\Pi'\setminus\Pi = \{r_1', r_2', r_3', r_4'\}$;
        \item the rectangle $r_1$ overlays~$r_4$;
        \item the rectangle $r_4'$ overlays~$r_1'$;
        \item the rectangles $r_i$ and $r_{i + 1}$ share a vertex for $i = 1, 2, 3$;
        \item the rectangles $r_i'$ and $r_{i + 1}'$ share a vertex for $i = 1, 2, 3$;
        \item $\partial r_* \subset \partial(r_1\cup r_2 \cup r_3 \cup r_4)$, $\partial r_*' \subset \partial(r_1' \cup r_2' \cup r_3' \cup r_4')$;
        \item $r_1\cup r_* \cup r_3 = r_1'\cup r_*'\cup r_3'$, $r_4\cup r_*\cup r_2 = r_4'\cup r_*' \cup r_2'$.
    \end{enumerate}
    Let $\epsilon$ and $\epsilon'$ be orientations of $\Pi$ and $\Pi'$, respectively, such that 
$$\epsilon(r) = \epsilon'(r) \text{ for } r\in \Pi\cap\Pi', \quad\text{and}\quad
        \epsilon(r_i) = \epsilon'(r_i') \text{ for } i=1, 2, 3, 4.$$
    Then we say that $(\Pi', \epsilon')$ is obtained from $(\Pi, \epsilon)$ by a \emph{flype} (see Figure~\ref{flype-fig}). If $\epsilon(r_1) = -1$, then the flype $(\Pi, \epsilon)\mapsto (\Pi', \epsilon')$ is called \emph{positive}.

    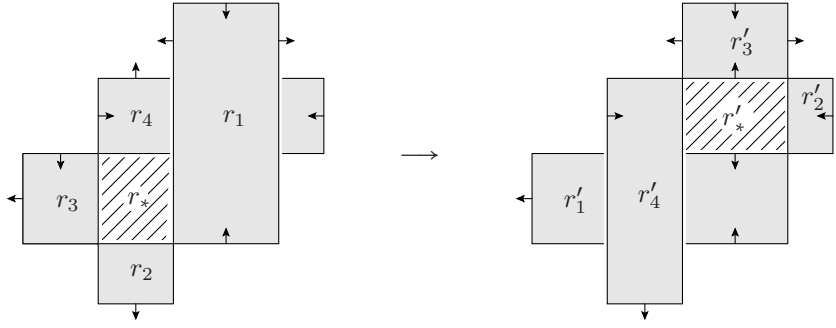
\begin{figure}
        \includegraphics{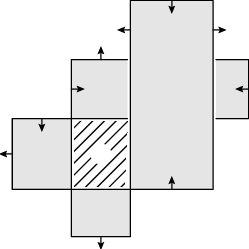}\hskip1cm
        \raisebox{60pt}{$\longrightarrow$}\hskip1cm
        \includegraphics{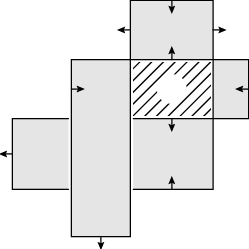}
    \put(-232, 75){$r_1$}
    \put(-267, 75){$r_4$}
    \put(-267, 17){$r_2$}
    \put(-295, 43){$r_3$}
    \put(-268, 44){$r_*$}
    \put(-103, 43){$r_1'$}
    \put(-13, 82){$r_2'$}
    \put(-40, 103){$r_3'$}
    \put(-75, 45){$r_4'$}
    \put(-43, 74){$r_*'$}
        \caption{A positive flype}\label{flype-fig}
    \end{figure}
\end{defi}

\begin{lemm}\label{flype-hol-lem}
    Let~$(\Pi_1,\epsilon_1)\mapsto(\Pi_2,\epsilon_2)$ be a positive flype.
    Then~$\widehat\Pi_2$ is obtained from~$\widehat\Pi_1$ by a positive leap over a cavity
    that is homeomorphic to an open three-ball.
\end{lemm}

\begin{proof}
The proof is similar to that of Lemma~\ref{bubble_cavity-lem}, in which one puts~$D=\{r_1,r_2,r_3,r_4\}$
and~$D'=\{r_1',r_2',r_3',r_4'\}$. We leave it to the reader.
\end{proof}

\section{Representing a foliation by a rectangular diagram}\label{foli-sec}
\begin{defi}
Two packs of rectangles~$P$ and~$P'$ are said to be \emph{almost compatible} if they satisfy
all the requirements for compatible packs of rectangles (see Definition~\ref{compat-pack-def})
except that the first rectangle of one pack may not be compatible with the last
rectangle of the other.
\end{defi}

One can see that if~$P$ and~$P'$ are almost compatible packs of rectangles
such that~$r_{\min}(P)$ is incompatible with~$r_{\max}(P')$, then
these rectangles share
a part of the boundary including
at least one non-degenerate arc. In particular, we may have~$r_{\min}(P)=
r_{\max}(P')$, in which case the union of the two packs is also a pack of rectangles.
The rectangles~$r_{\min}(P)$ and~$r_{\max}(P')$ may intersect in a rectangle,
in which case~$P$ and~$P'$ must be both positive or both negative.
Otherwise, the interiors of~$r_{\min}(P)$ and~$r_{\max}(P')$
are disjoint, and then one of the packs~$P$ and~$P'$ is positive,
and the other is negative.
Examples of almost compatible packs of rectangles are shown in Figure~\ref{almost-compat-pack-fig}.
\begin{figure}[ht]
\includegraphics[scale=0.3]{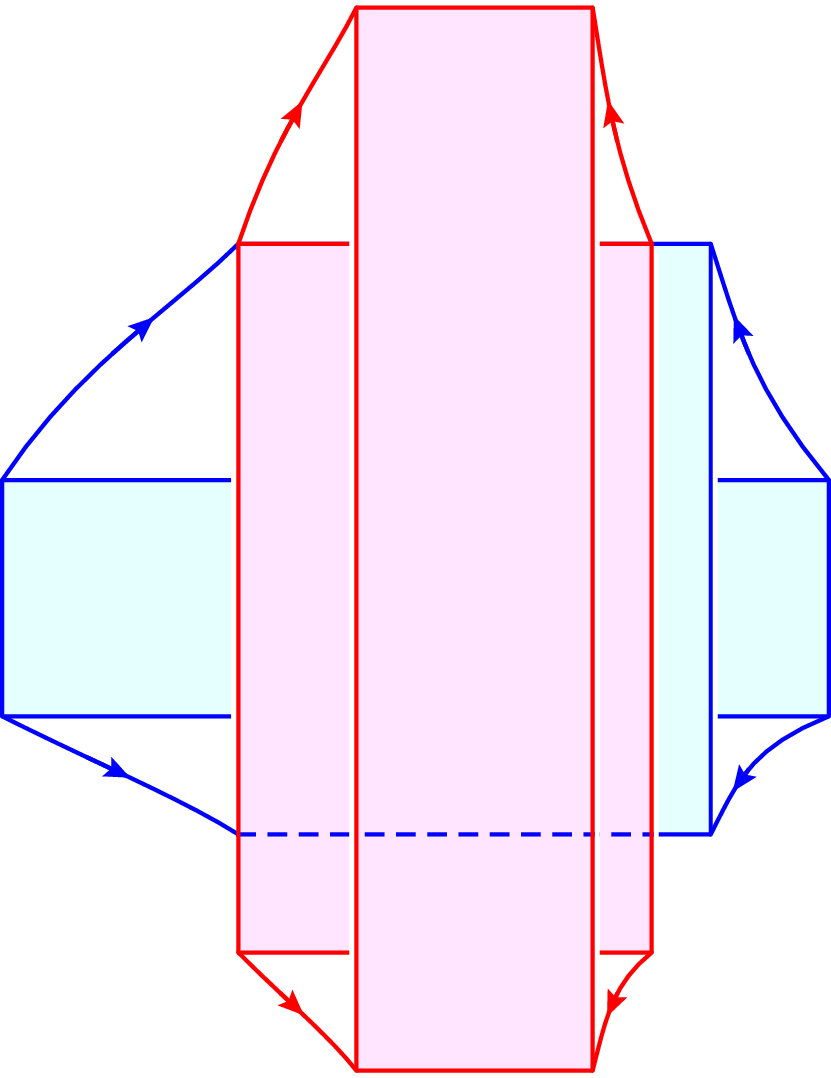}
\hskip1cm
\includegraphics[scale=0.3]{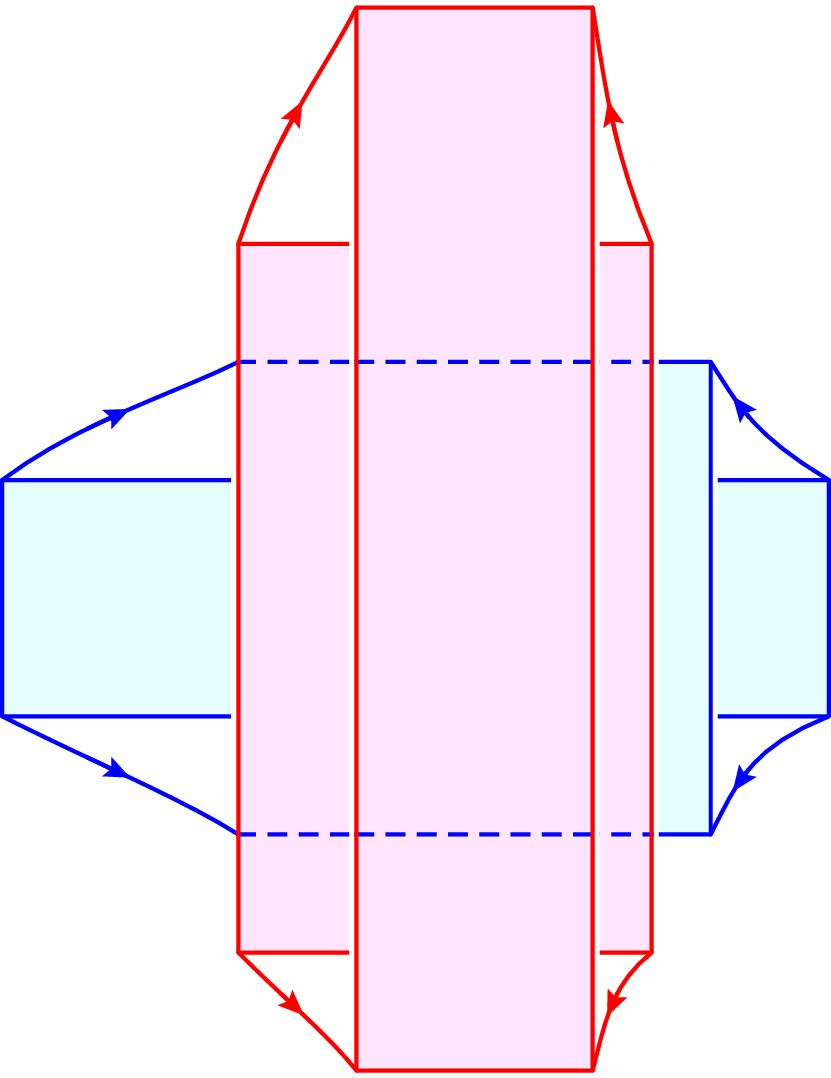}
\hskip1cm
\includegraphics[scale=0.3]{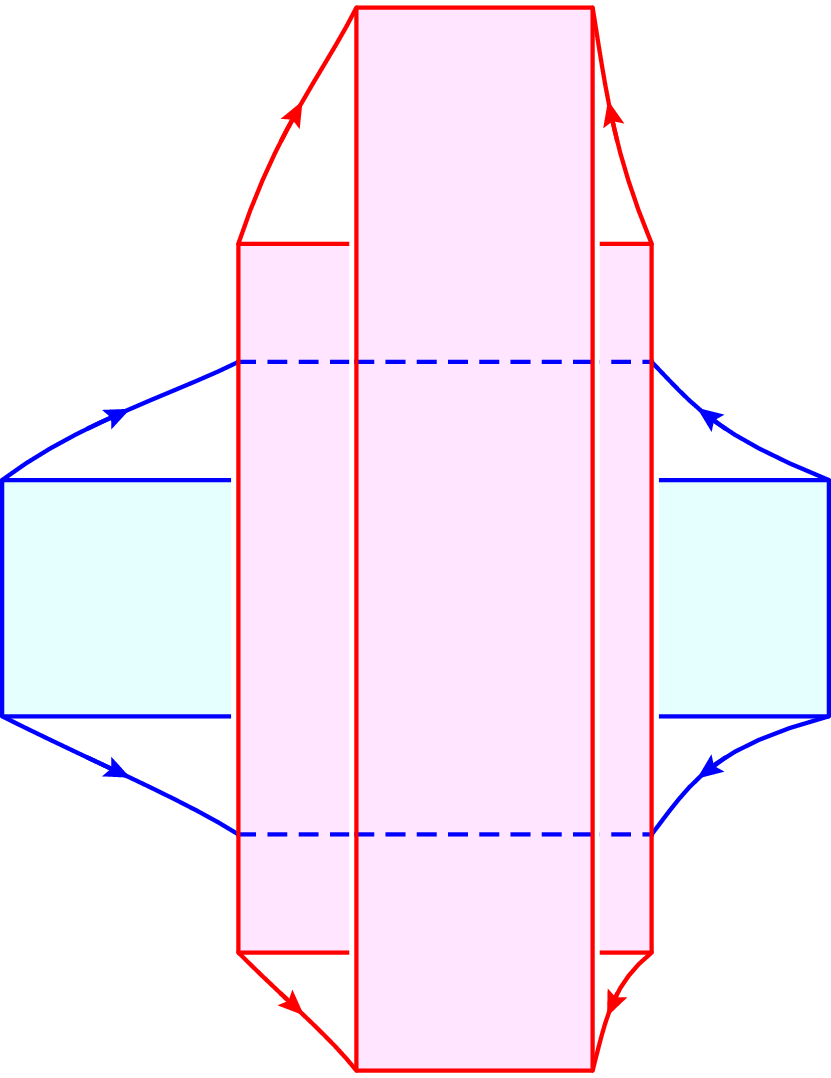}
\vskip1cm

\includegraphics[scale=0.3]{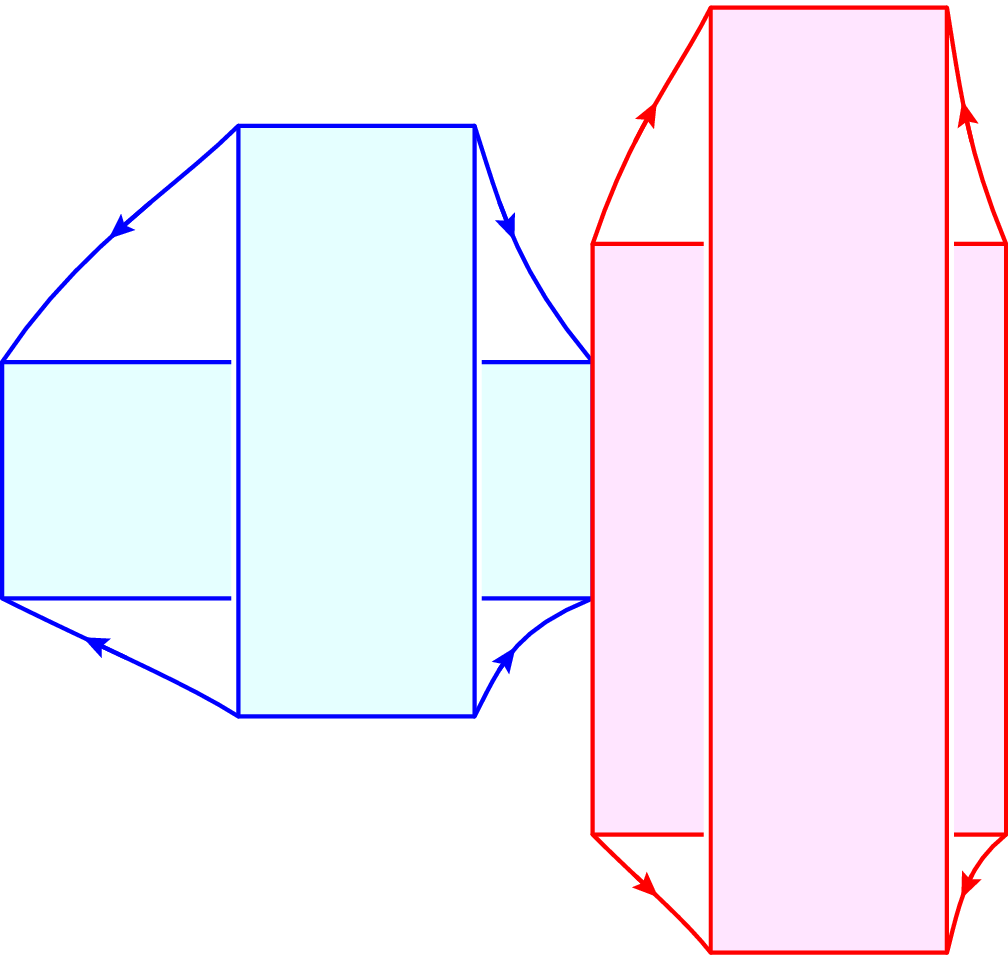}
\hskip1cm
\includegraphics[scale=0.3]{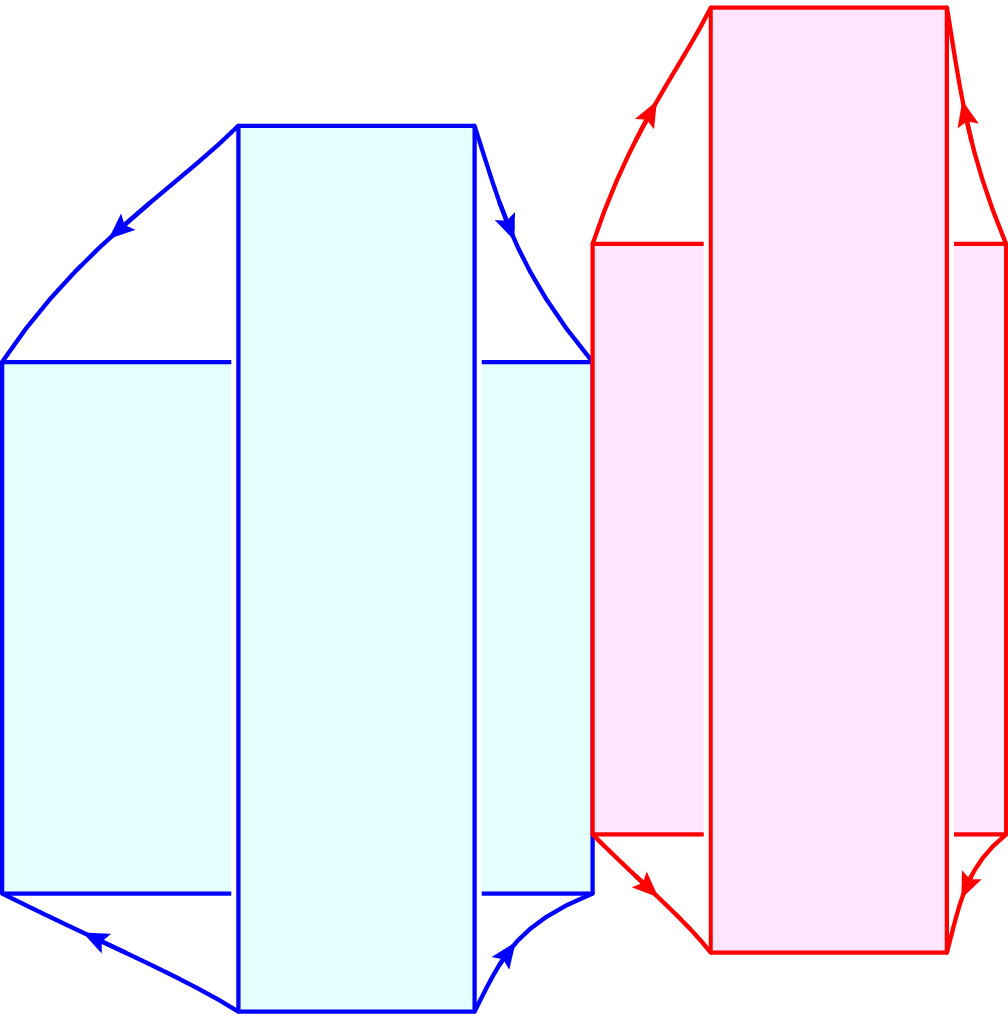}
\caption{Almost compatible packs of rectangles}\label{almost-compat-pack-fig}
\end{figure}

One can also see that if~$\{P_1,\ldots,P_k\}$ is a finite collection of pairwise almost
compatible packs of rectangles, then
$$\{r_{\mathrm{min}}(P_i)\}_{i=1,\ldots,k}\quad\text{and}\quad
\{r_{\mathrm{max}}(P_i)\}_{i=1,\ldots,k}$$
are oriented rectangular diagrams of quasi-surfaces with the orientations of rectangles given by
the orientations of the respective packs.

\begin{defi}
Let~$R$ be a rectangular diagram of a link. By a \emph{rectangular diagram of a foliation
in the complement of~$\widehat R$} we mean a finite collection~$\{P_1,\ldots,P_k\}$
of packs of rectangles such that:
\begin{enumerate}
\item
every two packs~$P_i$, $P_j$, $i\ne j$, are almost compatible;
\item
the rectangular diagram of a quasi-surface~$\{r_{\mathrm{min}}(P_i)\}_{i=1,\ldots,k}$
is obtained from~$\{r_{\mathrm{max}}(P_i)\}_{i=1,\ldots,k}$ by a sequence
of positive flypes and positive bubble moves, followed by the addition
of the diagram of a tube around~$\widehat R$ of the form~$\Omega_\varepsilon(R)$ for some~$\varepsilon>0$;
\item
there is no non-empty rectangular diagram of a surface~$\Pi$ with~$\Pi\subset\{r_{\mathrm{max}}(P_i)\}_{i=1,\ldots,k}$
and~$\partial\Pi=\varnothing$.
\end{enumerate}
When~$\Xi$ is a rectangular diagram of a foliation, we denote by~$\Pi_{\mathrm{min}}(\Xi)$ and~$\Pi_{\mathrm{max}}(\Xi)$
the following collections of rectangles, each of which is a rectangular diagram of a quasi-surface:
$$\Pi_{\mathrm{min}}(\Xi)=\{r_{\mathrm{min}}(P):P\in\Xi\},\quad
\Pi_{\mathrm{max}}(\Xi)=\{r_{\mathrm{max}}(P):P\in\Xi\}.$$
\end{defi}

\begin{prop}\label{catr=cabl-prop}
Let~$R$ be a rectangular diagram of a link, and let~$\Xi$ be a rectangular diagram of a foliation
in the complement of~$\widehat R$. Then
$$\bigcup_{P\in\Xi}\catr P=\bigcup_{P\in\Xi}\cabl P,\quad
\bigcup_{P\in\Xi}\catl P=\bigcup_{P\in\Xi}\cabr P.$$
Both these subsets of~$\mathbb T^2$ are finite unions of pairwise disjoint sloped arcs.
\end{prop}

\begin{proof}
Let~$\varepsilon>0$ be such that~$\Pi_{\min}(\Xi)$ contains the diagram~$\Omega_\varepsilon(R)$.

It follows from the definition of (almost) compatibility of packs of rectangles that
each connected component of~$\bigcup_{P\in\Xi}\cabl P$ is either a sloped arc or a union of sloped arcs that forms a
simple closed curve. The latter case, however, is impossible for~$\bigcup_{P\in\Xi}\cabl P$
cannot cross the boundary of the region
\begin{equation}\label{domain-tube-eq}
\bigcup_{(\theta,\varphi)\in R}\left([\theta-\varepsilon;\theta+\varepsilon]\times\mathbb S^1\cup
\mathbb S^1\times[\varphi-\varepsilon;\varphi+\varepsilon]\right)\subset\mathbb T^2.
\end{equation}
Indeed, $\bigcup_{P\in\Xi}\cabl P$ cannot intersect the interior of the rectangles
constituting~$\Omega_\varepsilon(R)$, and can only pass through their corners.
But since all rectangles in~$\Omega_\varepsilon(R)$ are the first rectangles
of some packs from~$\Xi$, no connected component of~$\bigcup_{P\in\Xi}\cabl P$ can pass through their corners
in the inward direction with respect to the domain~\eqref{domain-tube-eq}.

Thus, every connected component of~$\bigcup_{P\in\Xi}\cabl P$
is a sloped arc. By symmetry, the same holds for~$\bigcup_{P\in\Xi}\cabr P$,
$\bigcup_{P\in\Xi}\catl P$, and~$\bigcup_{P\in\Xi}\catr P$.

Let~$\alpha$ be a connected component of~$\bigcup_{P\in\Xi}\cabl P$, and let~$v$
be its endpoint. This means that~$v=\cabl r_{\mathrm{max}}(P)$ for some~$P\in\Xi$,
but~$v\notin\{\cabl r:r\in\Pi_{\mathrm{min}}(\Xi)\}$. Bubble moves and flypes preserve
the boundary of the involved rectangular diagrams of quasi-surfaces, and we have~$\partial\Omega_\varepsilon(R)=\varnothing$. Therefore,
$$\partial\Pi_{\mathrm{min}}(\Xi)=\partial\Pi_{\mathrm{max}}(\Xi)\not\ni v.$$
This means that~$v\in\{\catr r:r\in\Pi_{\mathrm{max}}(\Xi)\}$
and~$v\notin\{\catr r:r\in\Pi_{\mathrm{min}}(\Xi)\}$.

Thus, $v$ is the endpoint of a connected component of~$\bigcup_{P\in\Xi}\catr P$.
Similarly, the endpoint of any connected component of~$\bigcup_{P\in\Xi}\catr P$
is the endpoint of a connected component of~$\bigcup_{P\in\Xi}\cabl P$.
The same is true about the starting points. This implies that~$\bigcup_{P\in\Xi}\cabl P=\bigcup_{P\in\Xi}\catr P$.

The second equality is proved similarly.
\end{proof}

\begin{prop}\label{C-N-empty-prop}
Let~$R$ be a rectangular diagram of a link, and let~$\Xi$ be a rectangular diagram of a foliation
in the complement of~$\widehat R$. Then there is a cavity~$C$ and a positive~$\varepsilon$
such that~$C\cap N_\varepsilon(R)=\varnothing$, $C\cup N_\varepsilon(R)=\mathbb S^3\setminus\bigcup_{P\in\Xi}\widehat P$, and~$\partial(C\cup N_\varepsilon(R))=\widehat\Pi_{\mathrm{min}}(\Xi)\cup\widehat\Pi_{\mathrm{max}}(\Xi)$.
Such~$C$ is clearly unique.
\end{prop}

\begin{proof}
The statement is non-trivial only when~$\Pi_{\max}(\Xi)\ne\varnothing$, which is assumed below. This also
implies~$\partial\Pi_{\max}(\Xi)\ne\varnothing$.

By the definition of a rectangular diagram of a foliation, there exists a sequence of positive bubble moves and flypes
$$\Pi_{\max}(\Xi)=\Pi_0\mapsto\Pi_1\mapsto\ldots\mapsto\Pi_m$$
such that~$\Pi_{\min}(\Xi)=\Pi_m\cup\Omega_\varepsilon(R)$.
It follows from Lemmas~\ref{bubble_cavity-lem} and~\ref{flype-hol-lem} that, for each~$i=1,2,\ldots,m$,
the transformation~$\widehat\Pi_{i-1}\mapsto\widehat\Pi_i$ is a positive leap over a cavity, which we denote by~$C_i$.
We claim that the cavities~$C_i$, $i=1,\ldots,m$, and the solid torus~$N_\varepsilon(R)$
are pairwise disjoint.

Indeed, the union
\begin{equation}\label{3chain-eq}
\overline N_\varepsilon(R)\cup\bigcup_{i=1}^m\overline C_i\cup\bigcup_{j=1}^k\widehat P_j
\end{equation}
represents a three-chain with trivial boundary, so its homology class has the form~$l\cdot[\mathbb S^3]$.
A non-empty intersection of~$C_i$ with~$C_j$, where~$j\ne i$, or with~$N_\varepsilon(R)$
would imply~$l>1$.

Let~$v$ be a point in~$\partial\Pi_{\max}(\Xi)$. Due to symmetry, we may assume that~$v=\catr r_0$
for some~$r_0\in\Pi_{\max}(\Xi)$. Then each of~$\Pi_i$, $i=1,2,\ldots,m$, contains a rectangle~$r_i$
with~$\catr r_i=v$. It follows from Proposition~\ref{catr=cabl-prop} that there is a pack of rectangles~$P\in\Xi$
such that~$v\in\cabl(P)$, and~$\{\cabl r_{\min}(P),\cabl r_{\max}(P)\}\not\ni v$. All this implies that there is a
neighborhood~$U$ of~$\widehat v$ such that the cavities~$C_i$,
$i=1,\ldots,m$ are disjoint from the intersection~$U\cap\widehat P$. Therefore, this intersection participates
in the three-chain~\eqref{3chain-eq} with multiplicity one, hence $l=1$.
\end{proof}

\begin{defi}
Let~$\Xi$ be a rectangular diagram of a foliation in the complement of~$\widehat R$,
where~$R$ is a rectangular diagram of a link, and let~$C$ and~$\varepsilon$ be as in Proposition~\ref{C-N-empty-prop}.
Choose a deflating map~$\mathfrak d_C:\mathbb S^3\setminus C\rightarrow\mathbb S^3$ so
that it is identical on~$N_\varepsilon(R)$.

Propositions~\ref{thickened_surface-prop} and~\ref{C-N-empty-prop}, together with the discussion in
Section~\ref{pack-sec}, imply that the images of the foliations~$\mathscr F(P)$, $P\in\Xi$,
under the deflating map~$\mathfrak d_C$ compile into a well defined foliation on~$\mathbb S^3\setminus N_\varepsilon(R)$.
This foliation will be said to be \emph{associated with~$\Xi$} and denoted by~$\widehat\Xi$.
\end{defi}

Recall (see~\cite{gabai83}) that a compact leaf of a codimension-one foliation is said to be \emph{of depth~$0$},
and then, inductively, a leaf~$F$ is \emph{of depth~$k$}, $k\in\mathbb N$, if~$\overline F\setminus F$
consists of leaves of depth~$<k$. A \emph{finite depth foliation} is a codimension-one foliation~$\mathscr F$
for which there is~$n\in\mathbb N$ such that every leaf of~$\mathscr F$ has a depth not exceeding~$n$.
The smallest such~$n$ is then called the \emph{depth} of~$\mathscr F$.

Gabai proves in~\cite{gabai83} that any properly embedded oriented surface~$F$ in a compact irreducible oriented
$3$-manifold~$M$ can be included as a leaf into a co-oriented finite depth taut codimension-one foliation on~$M$
provided that~$F$ minimizes the Thurston norm. This means, in particular, that, for any non-split link~$L\subset\mathbb S^3$
and any minimal genus Seifert surface~$F$ for~$L$, the intersection~$F\cap M$ can be a leaf of such a foliation
on~$M=\mathbb S^3\setminus N(L)$, where~$N(L)$ is an open tubular neighborhood of~$L$.

The foliations considered by Gabai are assumed to be transverse to the boundary~$\partial M$, which is a union of tori,
and such a foliation is said to be \emph{taut} if there is a closed curve in~$M$ intersecting all leaves
of the foliation transversely.

Clearly, if a foliation~$\mathscr F$ on~$M$ is transverse to~$\partial M$, then it can be modified
near~$\partial M$ to become tangent to~$\partial M$, that is to say, to have each component of~$\partial M$
as a leaf. If~$\mathscr F$ is taut, then, after making it tangent to~$\partial M$, it remains taut in
the following sense.

\begin{defi}
Let~$\mathscr F$ be a codimension-one foliation on a compact $3$-manifold~$M$ such that~$\mathscr F$
is tangent to~$\partial M$. $\mathscr F$ is said to be \emph{taut} if there is a closed curve in~$M$
intersecting all leaves contained in~$M\setminus\partial M$ transversely.
\end{defi}

If a foliation~$\mathscr F$ is tangent to~$\partial M$ and taut in the sense of this definition,
then a taut foliation transverse to the boundary can be obtained by removing an open collar neighborhood of~$\partial M$.
So, there is no essential difference between $\partial M$-transverse and $\partial M$-tangent settings
in the context of taut foliations. (However, the leaves of depth~$k$ with non-empty boundary in the $\partial M$-transverse settings
become depth-$(k+1)$ ones in the $\partial M$-tangent settings, and the boundary components become depth~$0$ leaves.)

For the reasons mentioned in Section~\ref{discs-sec}, in the context of rectangular diagrams, it is more natural to
work in the $\partial M$-tangent settings.
The following statement is the main result of the present paper.

\begin{theo}\label{diagram-exist-thm}
Let~$R$ be a rectangular diagram of a link, and let~$\varepsilon>0$ be such that the tube diagram~$\Omega_\varepsilon(R)$
is defined.
Let also~$\mathscr F$ be a co-orientable finite depth taut codimension-one foliation on~$\mathbb S^3\setminus N_\varepsilon(R)$
such that each connected component of~$\widehat\Omega_\varepsilon(R)$ is a leaf of~$\mathscr F$ co-oriented outward~$N_\varepsilon(R)$.
Then there exists a rectangular diagram~$\Xi$ of a foliation in the complement of~$\widehat R$
such that~$\widehat\Xi$ is isotopic to~$\mathscr F$ relative to~$\Omega_\varepsilon(R)$.
\end{theo}

The proof will be given in Section~\ref{normal-sec}.

\section{Examples}\label{examp-sec}
Before proceeding with the proof of Theorem~\ref{diagram-exist-thm}
we demonstrate how a rectangular diagram of a foliation
may look in practice.

\begin{defi}
A rectangular diagram of a foliation~$\Xi$ is said to be \emph{reduced}
if there are no two packs of rectangles~$P',P''\in\Xi$ such that
\begin{equation}\label{min=max-eq}
r_{\min}(P')=r_{\max}(P'').
\end{equation}
\end{defi}

Clearly, if~\eqref{min=max-eq} holds, then the same foliation is represented
by the diagram~$(\Xi\setminus\{P',P''\})\cup\{P'\cup P''\}$, which has
fewer packs of rectangles. So, Theorem~\ref{diagram-exist-thm}
can be strengthen by requiring~$\Xi$ to be reduced.

\begin{prop}
Let~$\Xi$ be a reduced rectangular diagram of a foliation in the complement
of a link. Then~$\Xi$ can be recovered from the union of oriented sloped
curves
\begin{equation}\label{all-sloped-curves-eq}
\bigcup_{P\in\Xi}(\cabl P\cup\cabr P).
\end{equation}
\end{prop}

\begin{proof}
Denote the union of curves~\eqref{all-sloped-curves-eq} by~$\Gamma$.
It follows from Propositions~\ref{pack-defined-by-arcs-prop}
and~\ref{catr=cabl-prop} that~$\Xi$ is the set of maximal packs of rectangles~$P$
such that~$\cabl P\cup\cabr P\cup\catl P\cup\catr P\subset\Gamma$,
and the orientations of the corner arcs of~$P$ agree with the orientations of
the respective components of~$\Gamma$.
\end{proof}

Thus, in order to specify a reduced rectangular diagram of a foliation~$\Xi$, we just
need to provide the union of arcs~\eqref{all-sloped-curves-eq}.

\begin{exam}
We start from a trivial example, which
was considered in detail in Section~\ref{discs-sec}, the Reeb foliation
in the complement of an unknot. All rectangles participating
in the discussed rectangular diagrams of individual leaves
can be collected into nine packs of rectangles, the union
of the corner arcs of which is shown in Figure~\ref{reeb-fig}.
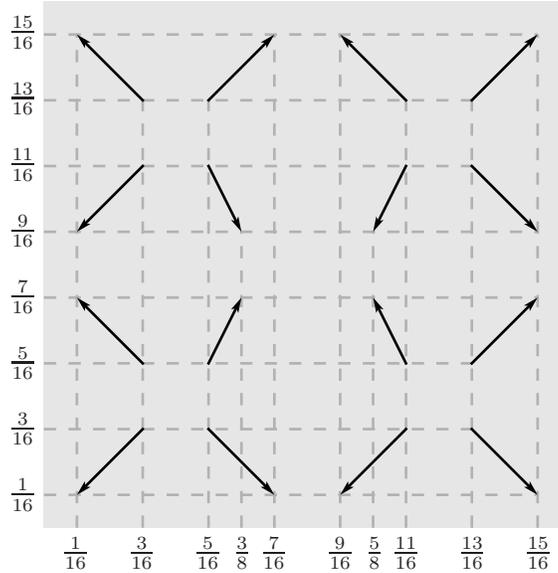
\begin{figure}
    \includegraphics[scale=7]{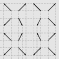}
    \put(-197,-5){$\tiny\frac1{16}$}
    \put(-172,-5){$\tiny\frac3{16}$}
    \put(-147,-5){$\tiny\frac5{16}$}
    \put(-132,-5){$\tiny\frac38$}
    \put(-122,-5){$\tiny\frac7{16}$}
    \put(-97,-5){$\tiny\frac9{16}$}
    \put(-82,-5){$\tiny\frac58$}
    \put(-72,-5){$\tiny\frac{11}{16}$}
    \put(-47,-5){$\tiny\frac{13}{16}$}
    \put(-22,-5){$\tiny\frac{15}{16}$}
    \put(-217,17){$\tiny\frac1{16}$}
    \put(-217,42){$\tiny\frac3{16}$}
    \put(-217,67){$\tiny\frac5{16}$}
    \put(-217,92){$\tiny\frac7{16}$}
    \put(-217,117){$\tiny\frac9{16}$}
    \put(-217,142){$\tiny\frac{11}{16}$}
    \put(-217,167){$\tiny\frac{13}{16}$}
    \put(-217,192){$\tiny\frac{15}{16}$}
    \caption{Rectangular diagram of the Reeb foliation in the complement of an unknot}\label{reeb-fig}
\end{figure}
\end{exam}

Note that the information about orientations of the arcs in the diagram
of a foliation is redundant. Indeed, an orientation of any of them
prescribes the orientations of the others, so there are only
two possible ways to orient all arcs in a consistent manner. And which
one of the two is correct can be determined by checking
which of the surfaces~$\widehat\Pi_{\min}(\Xi)$
and~$\widehat\Pi_{\max}(\Xi)$ contains a tube around the
given link. If the orientation is correct, this must be~$\widehat\Pi_{\min}(\Xi)$.

So, in further examples we omit arrows in the diagrams of foliations.

\begin{exam}
(A diagram of a foliation in the complement of the Figure Eight knot.)
A more complicated example is shown in Figure~\ref{fig-8-foli-fig},
where a rectangular diagram of a foliation is presented in the top picture
by the union of corner arcs, and the bottom pictures shows
the respective rectangular diagrams~$\Pi_{\min}$ (on the left) and~$\Pi_{\max}$ (on the right).
The green rectangles in the diagrams are positive,
and red ones are negative.
\begin{figure}[ht]
    \includegraphics[scale=.2]{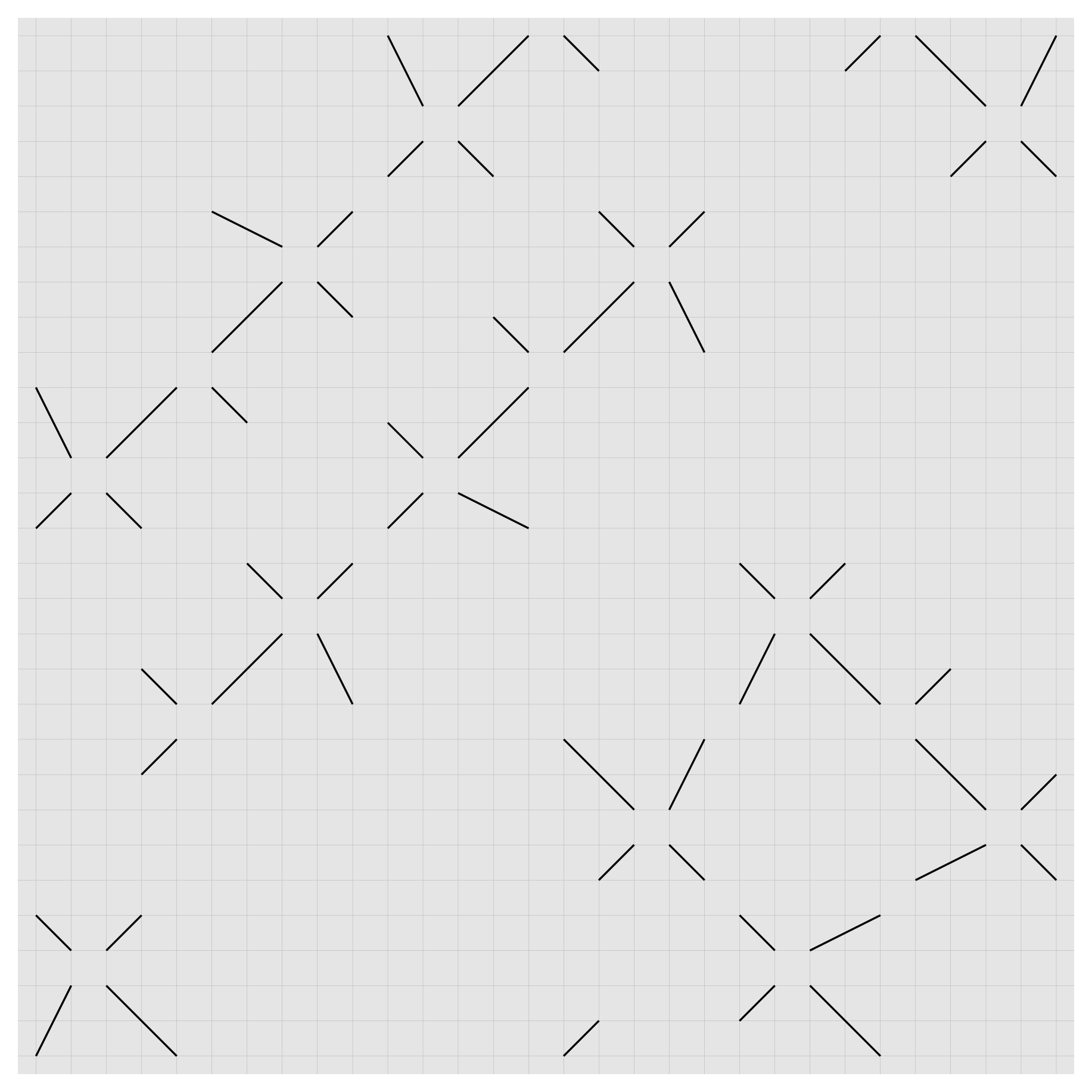} 
    
    \includegraphics[scale=.2]{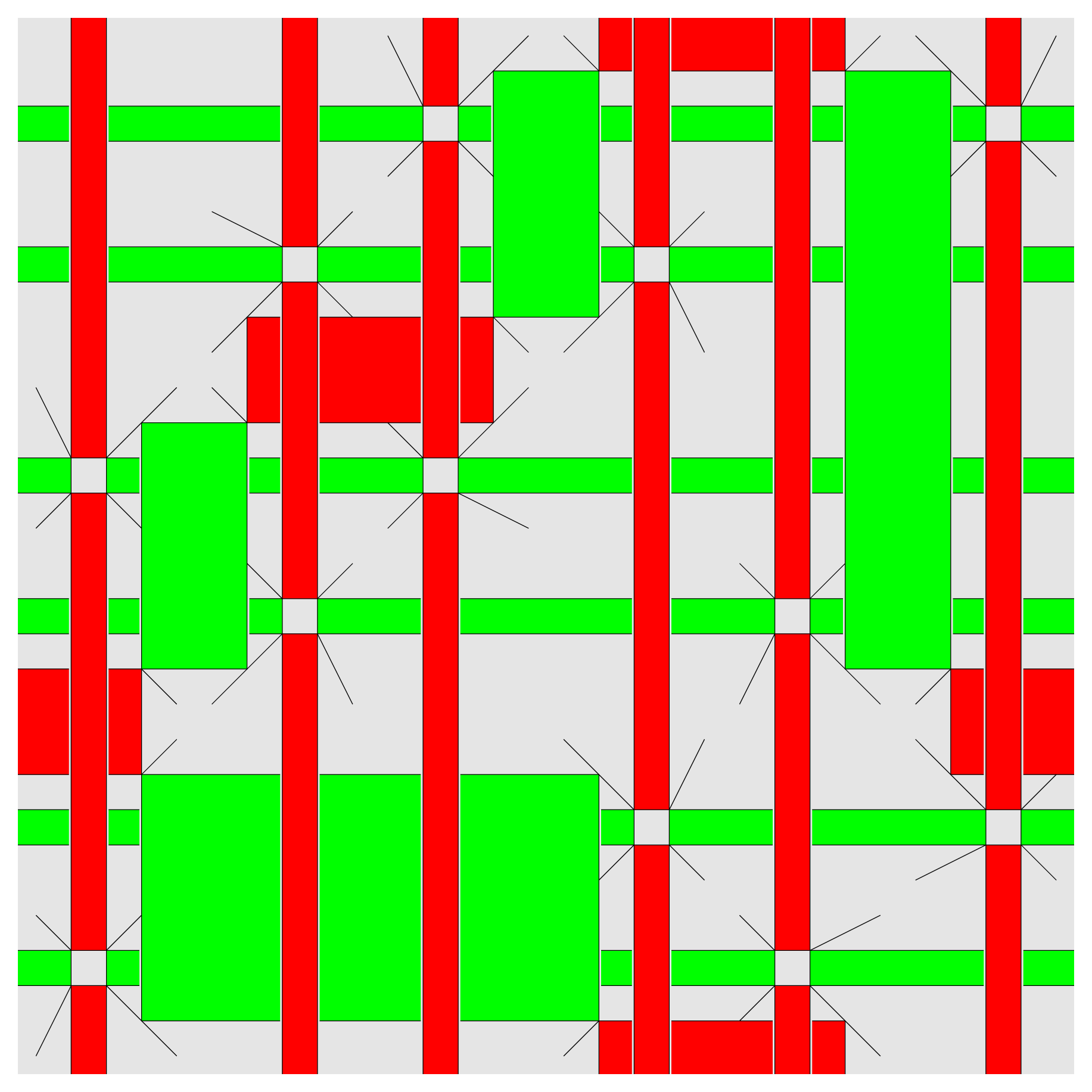}   \hskip1cm
    \includegraphics[scale=.2]{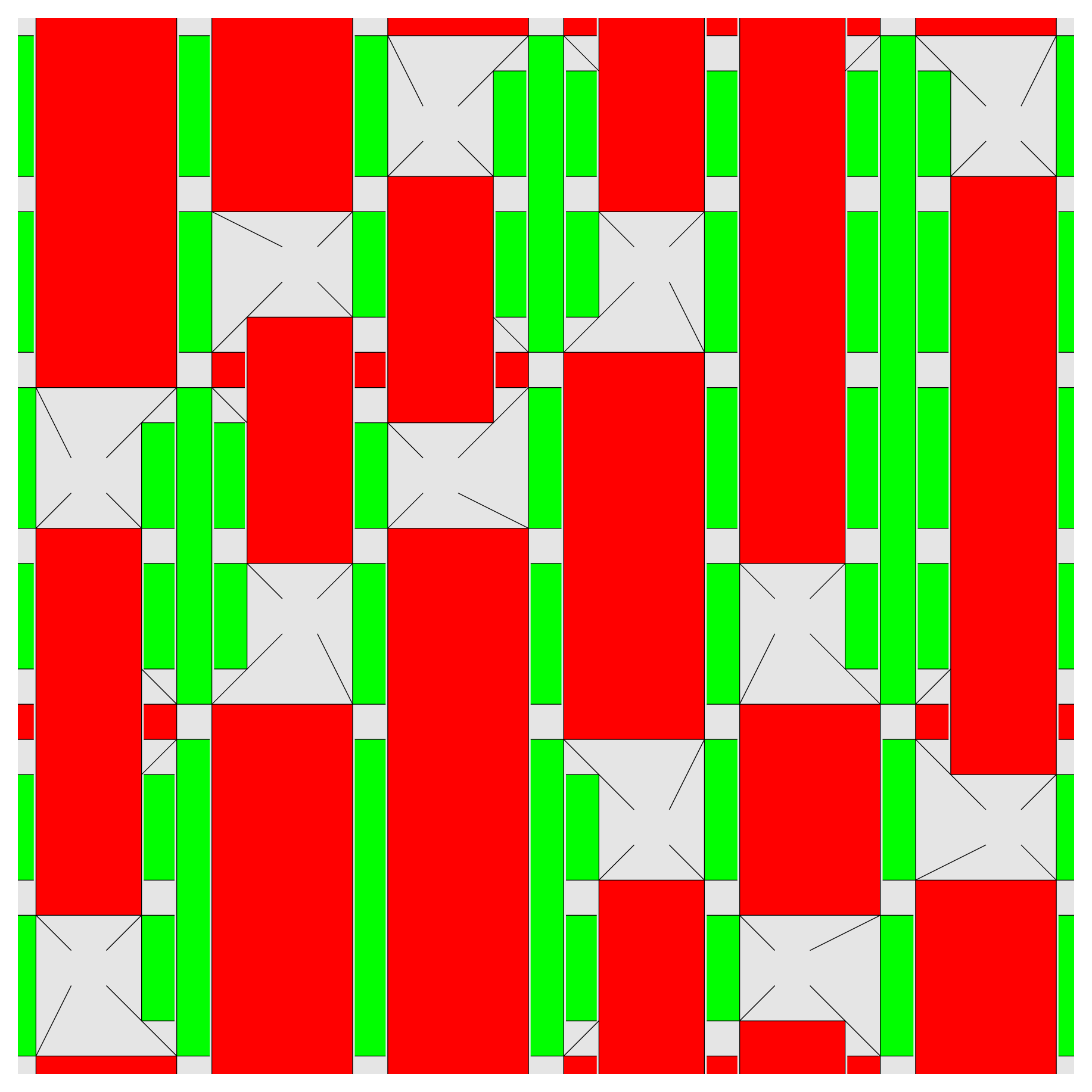} 
    \caption{A diagram of a foliation in the Figure Eight knot complement}\label{fig-8-foli-fig}
\end{figure}

As in the previous example, all the corner arcs of the diagram are just straight line segments,
which is due to the fact that this foliation in the complement of the Figure Eight knot
is obtained from a fibration by spinning the leaves around the boundary of a tubular neighborhood of the knot. This can be seen as follows.

Clearly, $\Pi_{\min}$ consists of a rectangular diagram of a connected surface
and a rectangular diagram of a tube around the knot. Let~$\Pi_0$ be the former, and~$\Omega$ the latter.
One can verify that~$\widehat\Pi_0$ is a genus-one Seifert surface for the Figure Eight knot,
and that~$\partial\Pi_0=\partial\Pi_{\min}=\partial\Pi_{\max}$.

Now, similarly to the case of a Reeb component described in Section~\ref{discs-sec},
in each pack of rectangles~$P$ such that~$r_{\min}(P)\in\Omega$, we can choose
an infinite sequence of rectangles converging to~$r_{\min}(P)$ so that
all the selected rectangles form a rectangular diagram~$A$ of a half-open annulus
such that~$\partial A=\partial\Pi_0$, and the annulus~$\widehat A$ is wound onto the tube~$\widehat\Omega$.

It can then be verified that~$\Pi_{\max}\cup A$ is obtained from~$\Pi_0\cup A$ by a positive
deformation, and that~$\Pi_0$ is obtained from~$\Pi_{\max}$ by a sequence of positive flypes and
positive bubble reduction moves.
\end{exam}

\begin{exam}
Shown in Figure~\ref{5_2-fig} is a rectangular diagram of a foliation
in the complement of the knot~$5_2$. This time, the knot is not fibered,
and foliation has depth two (in the classical settings, depth one), the minimal
possible depth of a taut foliation in this case. Again, the top picture
represents the union of all corner arcs of the diagram, and the two
bottom pictures shows the diagrams~$\Pi_{\min}$ and~$\Pi_{\max}$.
Now the surface~$\widehat\Pi_{\max}$ has two connected components
one of which is a Seifert surface for the knot~$5_2$ and the other is a two-disc.
Accordingly, the cavity complementary to the union of all the `curved cubes' associated
with the packs of rectangles from the diagram is homeomorphic
to~$\widehat\Pi_{\max}\times(0;1)$, and thus has two connected components.
\begin{figure}[ht]
\includegraphics[scale=.5]{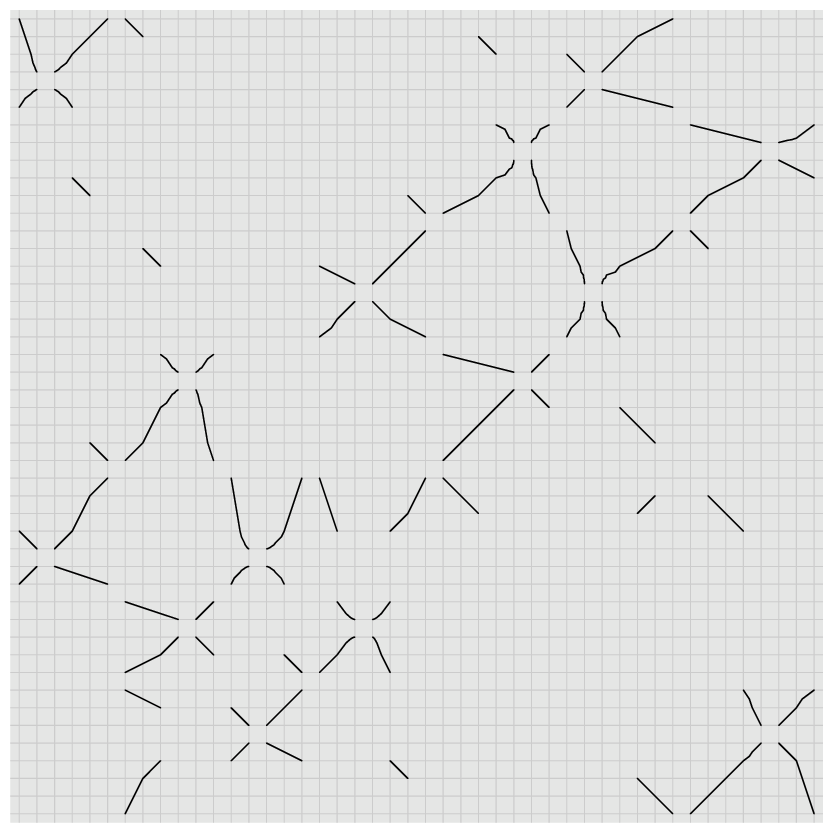}

\includegraphics[scale=.5]{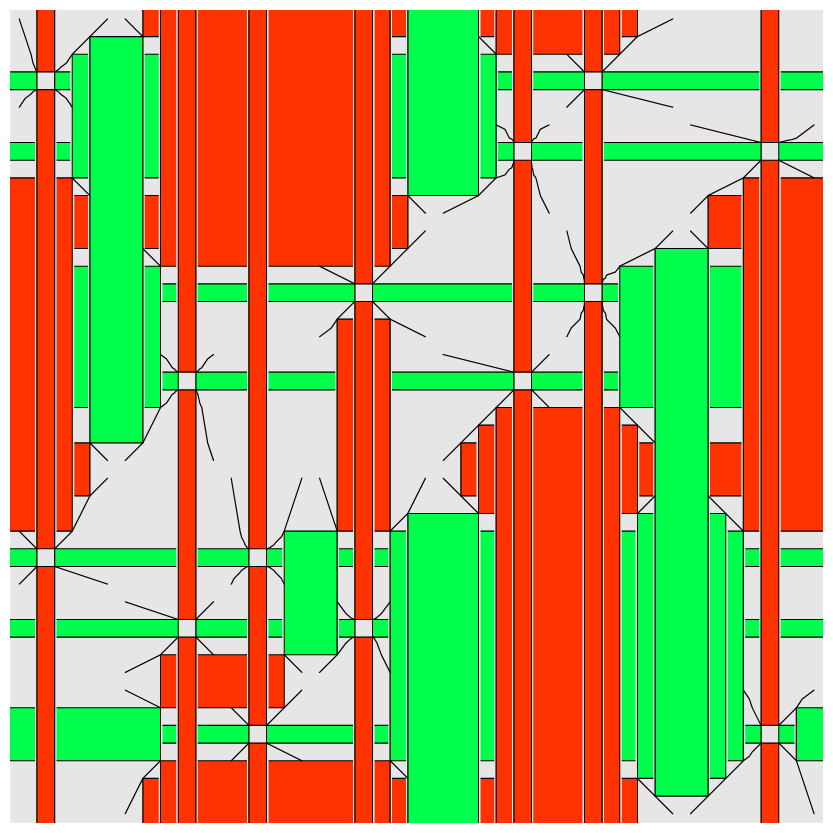}
\hskip1cm
\includegraphics[scale=.5]{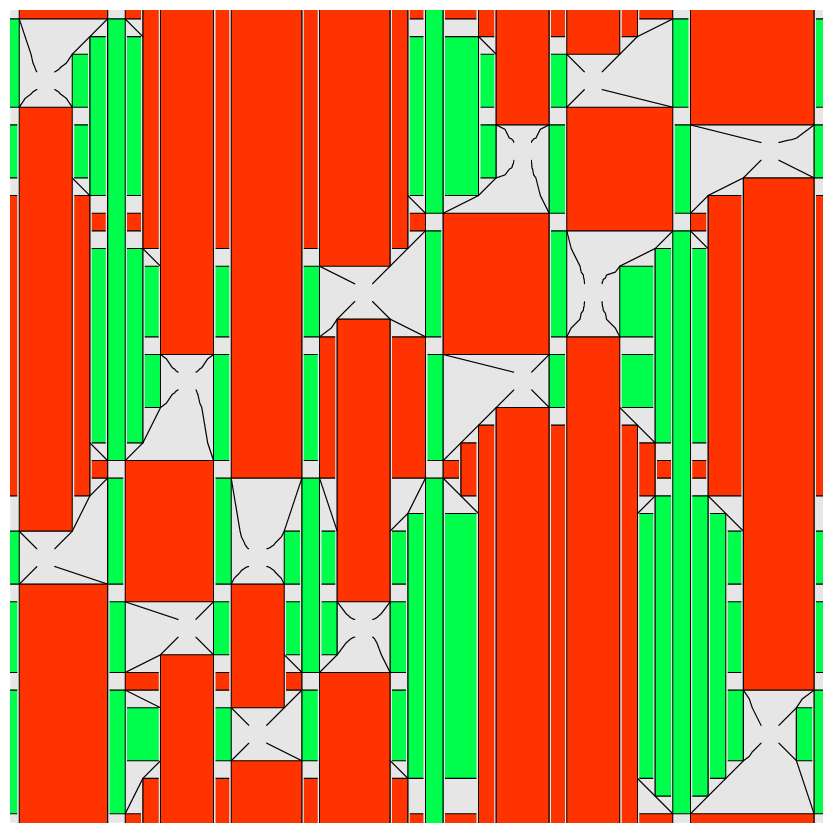}
\caption{A rectangular diagram of a taut foliation in the complement to
the knot~$5_2$}\label{5_2-fig}
\end{figure}
\end{exam}

\section{Normalizing a foliation}\label{normal-sec}
To prove Theorem~\ref{diagram-exist-thm}, we use the classical Kneser--Haken theory of normal
surfaces, which is tightly related with rectangular diagrams of surfaces in the
case when~$\mathbb S^3$ is triangulated in a special way. We will also make use of
the thin position technique, which is due to D.\,Gabai and A.\,Thompson, and
the concept of an almost normal surface introduced by~H.\,Rubinstein.

Recall that by a \emph{normal disc} in a $3$-simplex~$\Delta$ one means a properly embedded disc~$d\subset\Delta$ disjoint from
the vertices of~$\Delta$ and intersecting edges of~$\Delta$ transversely so that each edge meets~$\partial d$ at most once.
Two normal discs in~$\Delta$ are said to be \emph{equivalent} if they intersect the same edges of~$\Delta$.
In each $3$-simplex, there are seven equivalence classes of normal discs, four classes of \emph{triangles}
and three classes of \emph{quadrilaterals},
which are named according to the number of their intersections with the edges of the triangulation.

Let~$T$ be a triangulation of a $3$-manifold~$M$.
A surface~$F$ embedded in~$M$ is said to be \emph{normal with respect to~$T$} (or simply \emph{normal} when~$T$
is understood) if every connected component
of the intersection of~$F$ with every $3$-simplex~$\Delta$ of~$T$ is a normal disc in~$\Delta$.

Let a triangulation~$T$ be fixed. Two normal surfaces~$F,F'\subset M$ are said to be~\emph{normal isotopic}
if there is an ambient isotopy in~$M$ bringing~$F$ to~$F'$ and preserving all simplexes of~$T$,
and \emph{normal parallel} if there exists an embedding~$\psi:F\times[0;1]\rightarrow M$
such that~$\psi(F\times 0)=F$, $\psi(F\times 1)=F'$ and the surface~$\psi(F\times t)$ is normal for all~$t\in[0;1]$.

With every pair~$(X,Y)$ in which~$X$ and~$Y$ are finite subsets of~$\mathbb S^1$ each having at least three points,
we associate a triangulation~$T(X,Y)$ as follows. First,
we triangulate two copies of~$\mathbb S^1$ so that~$X$ is the set of vertices in
the first copy, and~$Y$ is the set of vertices in the second one. Then we take for~$T(X,Y)$ the join
of the two triangulations, which yields a triangulation of~$\mathbb S^3=\mathbb S^1*\mathbb S^1$.
In other words, $T(X,Y)$ is defined by requiring that every set of
the form~$\overline{\alpha*\beta}\subset\mathbb S^1*\mathbb S^1$, where~$\alpha$
and~$\beta$ are connected components of~$\mathbb S^1\setminus X$ and~$\mathbb S^1\setminus Y$,
respectively, is a $3$-simplex of~$T(X,Y)$.

\begin{prop}\label{rect=>norm-prop}
Let~$\Pi$ be a rectangular diagram of a surface with~$\partial\Pi=\varnothing$,
and let~$X,Y\subset\mathbb S^1$ be finite subsets disjoint from~$\Theta(\Pi)$ and~$\Phi(\Pi)$, respectively,
such that every rectangle in~$\Pi$ contains at least one point from~$X\times Y\subset\mathbb T^2$. Then the surface~$\widehat\Pi$
is normal with respect to~$T(X,Y)$.
\end{prop}

\begin{proof}
Let~$\Delta$ be a $3$-simplex of~$T(X,Y)$, and let~$r=[\theta_1;\theta_2]\times[\varphi_1;\varphi_2]$
be an element of~$\Pi$ such that the tile~$\widehat r$
has non-empty intersection with the interior of~$\Delta$. Let also~$r_\Delta=[\theta';\theta'']\times[\varphi';\varphi'']$
be the rectangle for which~$\Delta=[\theta';\theta'']*[\varphi';\varphi'']$. By construction, $\theta'$ and~$\theta''$
(respectively, $\varphi'$ and~$\varphi''$) are two consecutive points of~$X$ (respectively, $Y$)
with respect to the circular order in~$\mathbb S^1$. Since~$X\cap\Theta(\Pi)=Y\cap\Phi(\Pi)=\varnothing$,
we have~$\{\theta',\theta''\}\cap\{\theta_1,\theta_2\}=\{\varphi',\varphi''\}\cap\{\varphi_1,\varphi_2\}=\varnothing$.

The interior of the tile~$\widehat r$ is transverse to all arcs of the form~$\widehat v$, where~$v\in\mathbb T^2$. Therefore,
if the intersection~$r_\Delta\cap r$ is connected, then the intersection~$\Delta\cap\widehat r$ is a disc.
Moreover, since the only possible intersection of~$\partial\widehat r$ with~$\partial\Delta$
occurs at~$\mathbb S^1_{\tau=0}\cup\mathbb S^1_{\tau=1}$, the disc~$\widehat r$ intersects all faces of~$\Delta$ transversely.

Since, by hypothesis, $r$ contains at least one point from~$X\times Y$, one of the corners of~$r_\Delta$
lies in the interior of~$r$. So, $r$ contains either all four, or exactly two, or just one corner of~$r_\Delta$.
Up to symmetries, there are the following four cases to consider, which are illustrated in Figure~\ref{r-delta-fig}.

\begin{figure}[ht]
    \begin{tabular}{ccc}
    Case~1&\hbox to 8mm{}&Case 2\\[2mm]
    \includegraphics[scale=.7]{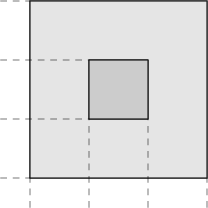}\put(-37,38){$r_\Delta$}
    \put(-51,18){$r$}\hskip1cm\includegraphics[scale=.4]{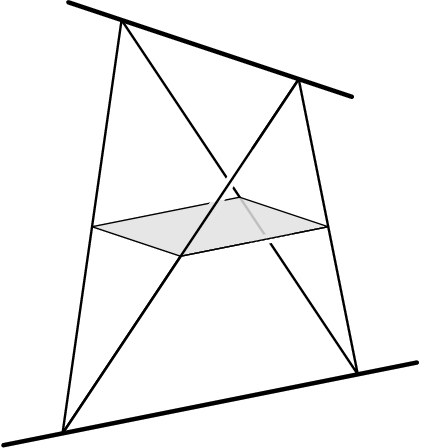}
    \put(-173, -10){$\theta_1$}
    \put(-113, -10){$\theta_2$}
    \put(-153, -10){$\theta'$}
    \put(-133, -10){$\theta''$}
    \put(-193, 8){$\varphi_1$}
    \put(-193, 68){$\varphi_2$}
    \put(-193, 28){$\varphi'$}
    \put(-193, 48){$\varphi''$}
    \put(-72, -7){$\varphi'$}
    \put(-14, 6){$\varphi''$}
    \put(-24.5, 74){$\theta''$}
    \put(-59, 85){$\theta'$}
    \put(3, 14){$\mathbb S^1_{\tau=0}$}
    \put(-10, 64){$\mathbb S^1_{\tau=1}$}
    &&
    \includegraphics[scale=.7]{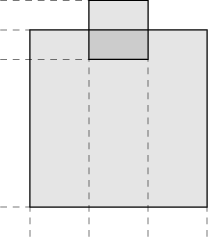}\put(-37,74){$r_\Delta$}
    \put(-51,18){$r$}\hskip1cm\includegraphics[scale=.4]{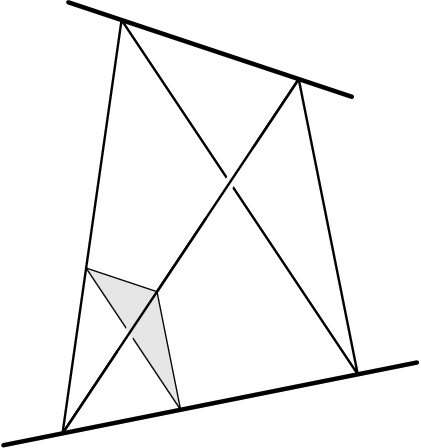}
    \put(-173, -10){$\theta_1$}
    \put(-113, -10){$\theta_2$}
    \put(-153, -10){$\theta'$}
    \put(-133, -10){$\theta''$}
    \put(-193, 8){$\varphi_1$}
    \put(-193, 68){$\varphi_2$}
    \put(-193, 58){$\varphi'$}
    \put(-193, 78){$\varphi''$}
    \put(-72, -7){$\varphi'$}
    \put(-49, -2){$\varphi_2$}
    \put(-13, 6){$\varphi''$}
    \put(-24.5, 74){$\theta''$}
    \put(-59, 85){$\theta'$}
    \put(3, 14){$\mathbb S^1_{\tau=0}$}
    \put(-10, 64){$\mathbb S^1_{\tau=1}$}
    \\[8mm]
    Case~3&&Case 4\\[2mm]
    \includegraphics[scale=.7]{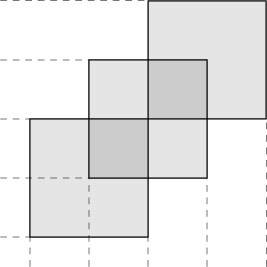}\put(-35,38){$r_\Delta$}
    \put(-71,18){$r$}\put(-13,75){$r'$}\hskip1cm\includegraphics[scale=.4]{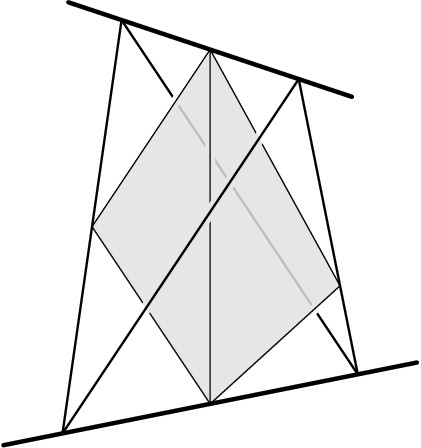}
    \put(-193, -10){$\theta_1$}
    \put(-173, -10){$\theta'$}
    \put(-113, -10){$\theta_3$}
    \put(-153, -10){$\theta_2$}
    \put(-133, -10){$\theta''$}
    \put(-213, 8){$\varphi_1$}
    \put(-213, 68){$\varphi''$}
    \put(-213, 28){$\varphi'$}
    \put(-213, 48){$\varphi_2$}
    \put(-213, 88){$\varphi_3$}
    \put(-43, -1){$\varphi_2$}
    \put(-42.5, 80.5){$\theta_2$}
    \put(-72, -7){$\varphi'$}
    \put(-13, 6){$\varphi''$}
    \put(-24.5, 74){$\theta''$}
    \put(-59, 85){$\theta'$}
    \put(3, 14){$\mathbb S^1_{\tau=0}$}
    \put(-10, 64){$\mathbb S^1_{\tau=1}$}
    &&
    \includegraphics[scale=.7]{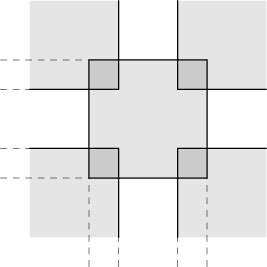}\put(-44,50){$r_\Delta$}
    \put(-71,18){$r$}
    \put(-53, -10){$\theta_2$}
    \put(-34, -10){$\theta_1$}
    \put(-65, -10){$\theta'$}
    \put(-22, -10){$\theta''$}
    \put(-102,28){$\varphi'$}
    \put(-102,40){$\varphi_2$}
    \put(-102,59){$\varphi_1$}
    \put(-102,70){$\varphi''$}
     \hskip1cm\raisebox{16mm}{impossible}\hbox to 8mm{}
    \end{tabular}
    \caption{Possible (and impossible) intersections of~$r_\Delta$ with~$r$}\label{r-delta-fig}
    \end{figure}
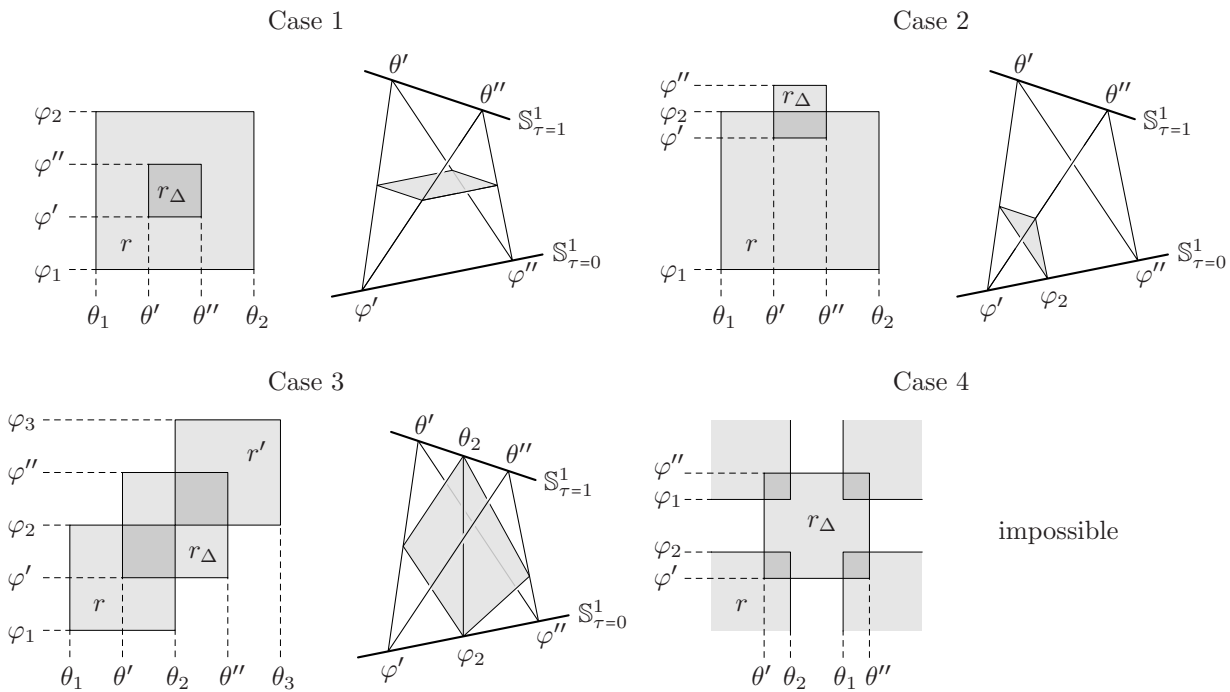

\smallskip\noindent\emph{Case 1}: $r_\Delta\subset r\setminus\partial r$.
The intersection~$\Delta\cap\widehat r$ is a disc that meets each arc of the form~$\widehat v$, where~$v\in r_\Delta$,
exactly once, and does not meet the two edges of~$\Delta$ contained in~$\mathbb S^1_{\tau=0}\cup\mathbb
S^1_{\tau=1}$.
Hence, $\Delta\cap\widehat r$ is a normal disc in~$\Delta$ equivalent to the one defined by the equation~$\tau=1/2$.

\smallskip\noindent\emph{Case 2}: $r_\Delta\cap r=[\theta';\theta'']\times[\varphi';\varphi_2]$.
The boundary of~$\widehat r$ meets~$\partial\Delta$ only at the point~$\varphi=\varphi_2$ of~$\mathbb S^1_{\tau=0}$,
so~$\partial(\Delta\cap\widehat r)\subset\partial\Delta$.
The intersection~$\Delta\cap\widehat r$ is a disc that meets each of the edges~$\widehat{(\theta',\varphi')}$,
$\widehat{(\theta'',\varphi')}$, and~$\Delta\cap\mathbb S^1_{\tau=0}$ exactly once,
and does not meet the other three edges of~$\Delta$. Hence, $\Delta\cap\widehat r$ is normal.

\smallskip\noindent\emph{Case 3}: $r_\Delta\cap r=[\theta';\theta_2]\times[\varphi';\varphi_2]$.
In this case, there must be another rectangle~$r'\in\Pi$ that has the form~$[\theta_2;\theta_3]\times[\varphi_2;\varphi_3]$
and intersects~$r_\Delta$ in the rectangle~$[\theta_2;\theta'']\times[\varphi_2;\varphi'']$.
Each of the intersections~$\Delta\cap\widehat r$ and~$\Delta\cap\widehat r'$ is a disc, and the boundaries
of these discs share the arc~$\widehat{(\theta_2;\varphi_2)}$. One can see that the union of these two discs
is also a disc whose boundary is contained in~$\partial\Delta$ and
meets each of the edges~$\widehat{(\theta';\varphi')}$, $\widehat{(\theta'';\varphi'')}$,
$\Delta\cap\mathbb S^1_{\tau=0}$, and~~$\Delta\cap\mathbb S^1_{\tau=1}$ exactly once,
and does not meet the other two edges of~$\Delta$. Hence, $\Delta\cap(\widehat r\cup\widehat r')$
is a normal disc.

\smallskip\noindent\emph{Case 4}: all four corners of~$r_\Delta$ are contained in~$r$, but~$r_\Delta\not\subset r$.
There must be another rectangle~$r'\in\Pi$ that shares a corner with~$r$. In order to be compatible with~$r$,
this rectangle must be contained in~$r_\Delta$, which contradicts the condition~$r'\cap(X\times Y)\ne\varnothing$.
Therefore, this case is impossible.

\smallskip
We see that, for any $3$-simplex~$\Delta$ of~$T(X,Y)$,
all connected components of the intersection~$\Delta\cap\widehat\Pi$ are normal discs. Thus, $\widehat\Pi$ is normal.
\end{proof}

\begin{prop}\label{normal=>rect-prop}
Let~$X,Y\subset\mathbb S^1$ be finite subsets with~$|X|,|Y|\geqslant3$, and let~$F\subset\mathbb S^3$
be a normal surface with respect to~$T(X,Y)$ such that every connected component of~$F$
has a non-empty intersection with both circles~$\mathbb S^1_{\tau=0}$ and~$\mathbb S^1_{\tau=1}$. Then~$F$
is normal isotopic to a surface of the form~$\widehat\Pi$, where~$\Pi$ is a rectangular diagram of a surface.
\end{prop}

\begin{proof}
We call a normal disc~$d$ of~$F$ \emph{vertical} if~$d$ intersects both circles~$\mathbb S^1_{\tau=0}$
and~$\mathbb S^1_{\tau=1}$. One can see that this can only happen when~$d$ is a quadrilateral,
and the two corners of~$d$ that lie on these circles are opposite to each other.

Let~$d_1,d_2,d_3,\ldots$ be all vertical normal discs of~$F$ (there may be an infinite number of them), and, for each~$i$,
let~$\alpha_i$ be a proper arc in~$d_i$ having endpoints at~$d_i\cap\mathbb S^1_{\tau=0}$ and~$d_i\cap\mathbb S^1_{\tau=1}$. We may assume without loss of generality that the surface~$F$
is transverse to all arcs of the form~$\widehat v$, where~$v\in\mathbb T^2$, outside the union~$\bigcup_i\alpha_i$,
whereas the arcs~$\alpha_i$ have the form~$\widehat v$.

Indeed, let us endow each $3$-simplex of~$T(X,Y)$ with an affine structure so that each arc of the form~$\widehat v$,
where~$v\in\mathbb T^2$, is a straight line segment. Then~$F$ is normal isotopic to a surface~$F'$ every
normal disc in which is either an affine triangle or a quadrilateral composed of two affine triangles in a $3$-simplex of~$T(X,Y)$,
where the normal isotopy can be chosen to be fixed on the $1$-skeleton of~$T(X,Y)$.
Moreover, vertical normal discs of~$F'$ can be chosen
in the form of the union of two affine triangles sharing a side of the form~$\widehat v$ for some~$v\in\mathbb T^2$.
The surface~$F'$ is transverse to the arcs of the form~$\widehat v$ everywhere except at a union of arcs
that can be taken for~$\alpha_i$.

In what follows we assume that~$\alpha_i=\widehat v_i$ for some~$v_i\in\mathbb T^2$, and~$F\setminus\bigcup_i\alpha_i$
is transverse to all arcs of the form~$\widehat v$.

Let~$V$ be a connected component of~$F\setminus\bigcup_i\alpha_i$. Define a map~$\omega:V\rightarrow\mathbb T^2$
by demanding that~$\widehat{\omega(p)}\cap V=p$ for all~$p\in V$. The pullpack of the standard locally
Euclidean metric on~$\mathbb T^2=\mathbb S^1\times\mathbb S^1$
is a locally Euclidean metric on~$V$. Let~$\widetilde V$ be the minimal compactification of~$V$
making this metric complete and keeping it locally Euclidean, and extend the map~$\omega$ to~$\widetilde V$ by continuity.
Clearly, $\omega$ is a locally isometric immersion~$\widetilde V\rightarrow\mathbb T^2$.

For every $3$-simplex~$\Delta$ of~$T(X,Y)$,
the closure of each connected component of the intersection~$V\cap\Delta$ is either a non-vertical normal disc
or one half of a vertical quadrilateral. This means that the image of~$V\cap\Delta$ under~$\omega$
has the form~$I\times J$, where each of~$I$ and~$J$ is either a closed or half-open interval of~$\mathbb S^1$.
Thus, the intersection of~$V$ with the $2$-skeleton of~$T(X,Y)$ cuts~$V$ into pieces isometric to such
subsets of~$\mathbb T^2$.

This implies that~$\widetilde V$ can be cut into finitely many rectangles, and the angle at any breaking point
of~$\partial\widetilde V$ is~$\pi/2$. Therefore, $\widetilde V$
is isometric either to a rectangle in the Euclidean plane, or a flat annulus with geodesic
boundary, or a flat two-torus. Now we claim that~$\omega$ is actually an embedding, and the latter
two cases (an annulus and a two-torus) are impossible.

Indeed, the vertical and horizontal arcs in~$\omega(\partial\widetilde V)$ come from the intersections of~$\overline V$
with~$\mathbb S^1_{\tau=0}$ and~$\mathbb S^1_{\tau=1}$, respectively. The equality~$\partial\widetilde V=\varnothing$
would mean that~$\overline V\cap(\mathbb S^1_{\tau=0}\cup\mathbb S^1_{\tau=1})=\varnothing$,
which implies~$\partial\overline V=\varnothing$. So, $V=\overline V$ would be a connected component of~$F$
disjoint from~$\mathbb S^1_{\tau=0}\cup\mathbb S^1_{\tau=1}$, which is assumed not to be the case.

If the entire~$\omega(\partial\widetilde V)$ is horizontal (respectively, vertical) in~$\mathbb T^2$,
then~$\overline V\cap\mathbb S^1_{\tau=1}=\varnothing$ (respectively, $\overline V\cap\mathbb S^1_{\tau=0}=\varnothing$),
which implies that~$V=\overline V$ is a spherical component of~$F$ intersecting only one of the
circles~$\mathbb S^1_{\tau=0}$ and~$\mathbb S^1_{\tau=1}$, and this again contradicts the hypothesis of the proposition.

Thus, $\widetilde V$ is a rectangle. The length of the sides of~$\widetilde V$ corresponding to
the intersections of~$\overline V$ with~$\mathbb S^1_{\tau=0}$ (respectively, $\mathbb S^1_{\tau=1}$) is equal to the variation of the
$\theta$-coordinate (respectively, $\varphi$-coordinate) on the boundary of a small neighborhood of such an intersection in~$V$, which implies that it is smaller than~$2\pi$. Therefore, $\omega$ is an embedding, and~$\omega(\widetilde V)$ is a rectangle in~$\mathbb T^2$.
Denote it by~$r(V)$.

Now let~$\Pi$ be the set of all rectangles of the form~$r(V)$, where~$V$ runs over the set of connected
components of~$F\setminus\bigcup_i\alpha_i$. One can see that the surface~$\widehat\Pi$ is normal parallel to~$F$.
\end{proof}

Thus, a normal surface with respect to a triangulation of the form~$T(X,Y)$ can always be viewed as one represented
by a rectangular diagram of a surface~$\Pi$ unless the surface is a torus disjoint from~$\mathbb S^1_{\tau=0}\cup\mathbb S^1_{\tau=1}$
or has a spherical component disjoint from one of the circles~$\mathbb S^1_{\tau=0}$
and~$\mathbb S^1_{\tau=1}$. The converse is true only under the additional assumptions that~$X\cap\Theta(\Pi)=Y\cap\Phi(\Pi)=\varnothing$
and no rectangle of the diagram is disjoint from~$X\times Y$. If the latter condition does not hold, then~$\widehat\Pi$
is not normal with respect to~$T(X,Y)$. However, using rectangular diagrams, we can take control over the normalization
process of such a surface.

To this end, we introduce more general transformations of rectangular diagrams of surfaces,
which allow to implement a normalization in the sense of the normal surface theory,
and then show that they can be decomposed into bubble moves and flypes.

All rectangular diagrams of surfaces considered in the sequel are assumed to
be oriented, and the considered transformations are divided into positive and negative
accordingly.
A transformation~$\Pi\mapsto\Pi'$ of rectangular diagrams of surfaces
is said to be \emph{positive} if it results in `pushing $\widehat\Pi$ in the positive direction'.
Formally, this means that there is a submanifold~$M$ of~$\mathbb S^3$ which satisfies~$\partial M=\widehat\Pi''-\widehat\Pi$ provided
that it is endowed with the induced orientation and viewed as a three-chain, where~$\Pi''$ is a diagram obtained from~$\Pi'$ by a small positive deformation.
The intersection of all such manifolds~$M$ will be referred to
as the \emph{bordism associated with the transformation~$\Pi\mapsto\Pi'$}. The
inverse of a positive transformation is called \emph{negative}.

There is one exceptional case of a transformation considered below that can be positive and negative simultaneously.
This occurs when~$\Pi'=\varnothing$ and~$\widehat\Pi$ is a two-sphere. In this case,
we will silently assume that such transformations are endowed with a choice of an associated bordism.

\begin{defi}
Two distinct rectangles~$r=[\theta_1;\theta_2]\times[\varphi_1;\varphi_2]$
and~$r=[\theta_1';\theta_2']\times[\varphi_1';\varphi_2']$ are said to be
\emph{$\theta$-companions} (respectively, \emph{$\varphi$-companions})
of one another if~$\theta_1'=\theta_1$ and~$\theta_2'=\theta_2$ (respectively,
$\varphi_1'=\varphi_1$ and~$\varphi_2'=\varphi_2$).
\end{defi}

\begin{defi}\label{gen-wrinkle-def}
Let~$\Pi$ be a rectangular diagram of a surface, and let~$r=[\theta_1;\theta_2]\times[\varphi_1;\varphi_2]\in\Pi$
be a rectangle such that~$(\theta_1;\theta_2)\cap\Theta(\Pi)=\varnothing$. Suppose that~$\partial\Pi=\varnothing$
and that there is no $\theta$-companion of~$r$ in~$\Pi$. Suppose also that~$[\theta_2;\theta_1]\times[\varphi_2;\varphi_1]$
is not a rectangle of~$\Pi$.
Let~$r_0=r,{}r_1,r_2,\ldots,r_m$ be all rectangles of~$\Pi$ having at least one side in~$\theta_1\times\mathbb S^1$
or~$\theta_2\times\mathbb S^1$.

Let~$\chi_*:\mathbb T^2\rightarrow\{0,1\}$ be the characteristic function of the annulus~$[\theta_1;\theta_2]\times\mathbb
S^1$,
and let~$\chi_i$, $i=0,{}1,\ldots,m$, be the characteristic functions of the respective rectangles~$r_i$.
Then one can see that the closure of the set
$$\left\{p\in\mathbb T^2:\chi_*(p)+\sum_{i=0}^m\chi_i(p)\equiv 1\,(\mathrm{mod}\,2)\right\}$$
is a union of $m-1$ rectangles whose interiors are
pairwise disjoint. Let~$r_1',\ldots,r_{m-1}'$ be these rectangles,
and let $\Pi'=(\Pi\setminus\{r_0,r_1,\ldots,r_m\})\cup\{r_1',\ldots,r_{m-1}'\}$. We orient~$\Pi'$
so that the signs of the common rectangles of~$\Pi$ and~$\Pi'$ agree.
Then we say that~$\Pi\mapsto\Pi'$ is a \emph{generalized vertical wrinkle reduction}, and the inverse
operation a \emph{generalized vertical wrinkle creation}.

\emph{Horizontal generalized wrinkle reduction} and \emph{creation} operations are defined similarly
with the exchange of the roles of the coordinates~$\theta$ and~$\varphi$.
\end{defi}

An example of a generalized wrinkle reduction is shown in Figure~\ref{gen-wrinkle-fig}.
\begin{figure}[ht]
    \includegraphics[scale=0.65]{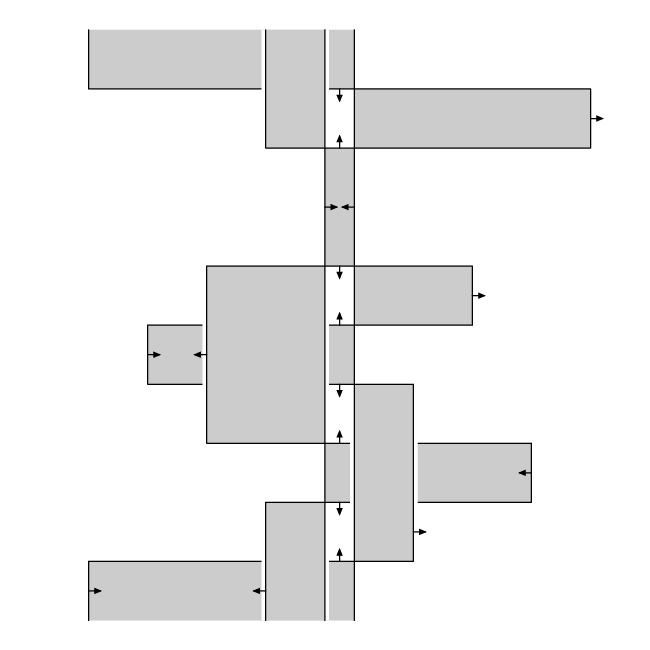}\put(-99,145){$r$}
    \raisebox{90pt}{$\longrightarrow$}
    \includegraphics[scale=0.65]{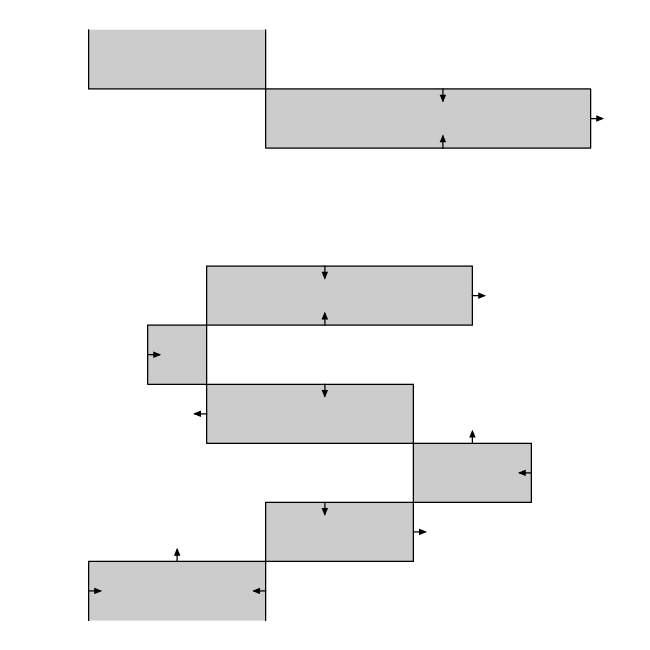}
    \caption{Generalized positive wrinkle reduction move}\label{gen-wrinkle-fig}
\end{figure}

Now we claim that each generalized wrinkle reduction move is either positive or negative. This can
be verified directly by examining the change in the associated surface, and can also be seen from the following lemma.

\begin{lemm}\label{wrinkle-decomp-lem}
A generalized wrinkle reduction \emph(creation\emph) can be decomposed
into a sequence of flypes followed by a bubble reduction move so that all these moves are either positive
or all negative.
\end{lemm}

\begin{proof}
Due to symmetry, it suffices to consider the case when~$\Pi\mapsto\Pi'$ is a vertical wrinkle reduction move, and the rectangle~$r$ from
Definition~\ref{gen-wrinkle-def} is positive. The number~$m$ from this definition
will be referred to as the \emph{complexity} of the generalized wrinkle reduction~$\Pi\mapsto\Pi'$.
The proof is by induction in~$m$. One can see that~$m$ is always even.

If~$m=2$, then~$\Pi\mapsto\Pi'$ is a positive bubble reduction move, and we are done.

Suppose that~$m>2$, and let~$r_1$, $r_2$, $r_3$, and~$r_4$ be rectangles from~$\Pi$
such that~$\catr r_1=\cabl r$, $\catl r_2=\cabr r$, $\catl r_3=\cabr r_1$, and~$\catr r_4=\cabl r_2$.
We cannot have~$r_3=r_4$, since otherwise~$\Pi\mapsto\Pi'$ would not be a generalized
wrinkle reduction. Therefore, one of the following must hold:
\begin{enumerate}
\item
$r_1$ overlays~$r_4$;
\item
$r_2$ overlays~$r_3$.
\end{enumerate}
In both cases, the two overlapping rectangles make~$\Pi$ suitable for a positive flype~$\Pi\mapsto\Pi''$.
Then~$\Pi''\mapsto\Pi'$ is a generalized vertical wrinkle reduction of complexity~$m-2$,
and the new rectangle~$r$ is again positive. Hence, we have the induction step.
\end{proof}

Recall that a \emph{compressing disc} for a surface~$F$ embedded into a three-manifold~$M$ is
a two-disc~$d$ embedded in~$M$ such that~$d\cap F=\partial d$ and~$d\cap\partial F=\varnothing$.
For each compressing disc one defines a \emph{compression} of~$F$, which consists in removing a small neighborhood of~$\partial d$ from~$F$ and gluing up
the two new boundary components by two-discs `parallel' to~$d$. A surface~$F$ is \emph{incompressible}
if any compression results in an isotopy and the addition of a spherical component to~$F$ bounding a three-ball
whose interior is disjoint from~$F$.

\begin{defi}
Let~$\Pi$ be a rectangular diagram of a surface, and let~$r\notin\Pi$ be a rectangle
such that its sides are also sides of four distinct rectangles from~$\Pi$, and no vertex of~$\Pi$ lies in the interior
of~$r$.
In other words, there are~$\theta_1,\theta_2,\theta_3,\theta_4,\varphi_1,\varphi_2,\varphi_3,\varphi_4\in\mathbb S^1$ such that:
\begin{enumerate}
\item
$\theta_1,\theta_2,\theta_3,\theta_4$ ($\varphi_1,\varphi_2,\varphi_3,\varphi_4$) follow in~$\mathbb S^1$
in the indicated order;
\item
no vertex of~$\Pi$ lies in~$(\theta_2;\theta_3)\times(\varphi_2;\varphi_3)$;
\item
the rectangles
$$r_1=[\theta_1;\theta_2]\times[\varphi_2;\varphi_3],\
r_2=[\theta_2;\theta_3]\times[\varphi_1;\varphi_2],\
r_3=[\theta_2;\theta_3]\times[\varphi_3;\varphi_4],\
r_4=[\theta_3;\theta_4]\times[\varphi_2;\varphi_3]$$
belong to~$\Pi$.
\end{enumerate}
Let~$\Pi'=(\Pi\setminus\{r_1,r_2,r_3,r_4\})\cup\{r',r''\}$, where
$$r'=[\theta_1;\theta_4]\times[\varphi_2;\varphi_3],\
r''=[\theta_2;\theta_3]\times[\varphi_1;\varphi_4].$$
Then we say that~$\Pi\mapsto\Pi'$ is an (\emph{ordinary}) \emph{compression}.
\end{defi}

The compression operation is illustrated in Figure~\ref{compres-fig}.
One can see, that if~$\Pi\mapsto\Pi'$ is a compression, then~$\widehat\Pi\mapsto\widehat\Pi'$
is a compression in the topological sense. Clearly, each compression is either positive or negative.
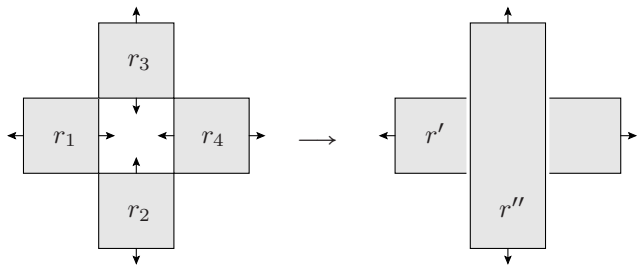
\begin{figure}
    \includegraphics[scale=1]{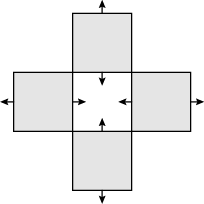}\put(-81,47){$r_1$}
    \put(-53,19){$r_2$}
    \put(-53,76){$r_3$}\put(-25,47){$r_4$}
    \hskip0.4cm
    \raisebox{1.6cm}{$\longrightarrow$}
    \hskip0.4cm
    \includegraphics[scale=1]{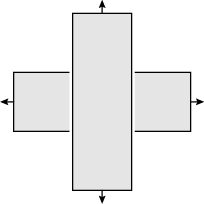}\put(-81,47){$r'$}\put(-53,19){$r''$}
    \caption{A positive compression}\label{compres-fig}
\end{figure}

\begin{lemm}\label{connected-sum-lem}
Let~$\Pi'$ and~$\Pi''$ be oriented rectangular diagrams of surfaces such that~$\Pi'\cap\Pi''=\varnothing$,
$\partial\Pi'=\partial\Pi''=\varnothing$,
and~$\Pi'\cup\Pi''$ is a rectangular diagram of a surface. Then, possibly after flipping the orientation
of~$\Pi''$, there exists a rectangular diagram of a surface~$\Pi$ such that~$\Pi\mapsto\Pi'\cup\Pi''$
is a compression.
\end{lemm}

\begin{proof}
Since~$\partial\Pi'=\partial\Pi''=\varnothing$, there must be rectangles~$r'\in\Pi'$ and~$r''\in\Pi''$
such that their interiors have a non-empty intersection. Let~$v\in r'\cap r''$ be a point
not belonging to~$\Theta(\Pi'\cup\Pi'')\times\Phi(\Pi'\cup\Pi'')$. Then the arc~$\widehat v$
pierces both surfaces~$\widehat\Pi'$ and~$\widehat\Pi''$. Rechoose
the pair of rectangles~$r'\in\Pi'$ and~$r''\in\Pi''$ so that~$\widehat v$ pierces~$\widehat r'$
and~$\widehat r''$ in points that are next to each other, that is, so that
the open subarc of~$\widehat v$ between these points is disjoint from~$\widehat\Pi'\cup\widehat\Pi''$.

Due to this choice of~$r'$ and~$r''$ the corners of the rectangle~$r'\cap r''$
are disjoint from all rectangles from~$\Pi'\cup\Pi''$ other than~$r'$ and~$r''$.
The closure of the symmetric difference~$r'\triangle r''$ is a union of four
rectangles, which we denote by~$r_1$, $r_2$, $r_3$, and~$r_4$. One can
now take~$\bigl((\Pi'\cup\Pi'')\setminus\{r',r''\}\bigr)\cup\{r_1,r_2,r_3,r_4\}$
for~$\Pi$, and the passage~$\Pi\mapsto\Pi'\cup\Pi''$ will be a compression.
\end{proof}

\begin{defi}
Let~$\Pi$ be a rectangular diagram of a surface that can be decomposed into a disjoint
union of two diagrams of surfaces~$\Pi'$ and~$\Pi''$ such that~$\widehat\Pi''$
is a sphere that bounds a three-ball disjoint from~$\widehat\Pi'$.
Then the passage from~$\Pi$ to~$\Pi'$ is referred to as a \emph{spherical component removal}.
If~$\Pi''$ consists of just two rectangles, then this spherical component is
said to be \emph{trivial}.
\end{defi}

\begin{defi}\label{gen-compres-def}
Let~$\Pi$ be a rectangular diagram of a surface, and let~$r=[\theta_1;\theta_2]\times[\varphi_1;\varphi_2]\in\Pi$
be a rectangle such that~$(\theta_1;\theta_2)\cap\Theta(\Pi)=\varnothing$. Suppose that~$\partial\Pi=\varnothing$
and that there is at least one $\theta$-companion of~$r$ in~$\Pi$. Suppose also that
no rectangle of~$\Pi$ has the form~$[\theta_2;\theta_1]\times[\varphi';\varphi'']$. Then define~$\Pi'$ as in Definition~\ref{gen-wrinkle-def}
with the only difference that the number of new rectangles~$r_i'$ is now~$m-1-2k$, where~$k$ is the number
of $\theta$-companions of~$r$. Then we say that~$\Pi\mapsto\Pi'$ is a \emph{generalized
vertical compression}.

A \emph{generalized horizontal compression} is defined similarly with the roles of~$\theta$ and~$\varphi$
exchanged.
\end{defi}

\begin{lemm}\label{gen-compres-decomp-lem}
A positive \emph(respectively, negative\emph) generalized compression can be decomposed
into a sequence of transformations
that include flypes, ordinary compressions, and a bubble reduction move so that all these moves are positive
\emph(respectively, negative\emph).
\end{lemm}

\begin{proof}
The proof is parallel to that of Lemma~\ref{wrinkle-decomp-lem}.
In addition to Cases~(1) and~(2) we may now have~$r_3=r_4$, in which case we apply
an ordinary compression to make the induction step. The details are left to the reader.
\end{proof}

Let~$X,Y\subset\mathbb S^3$ be subsets with~$|X|,|Y|\geqslant3$, and let~$\Pi$
be a rectangular diagram of a surface such that~$\Theta(\Pi)\cap X=\Phi(\Pi)\cap Y=\varnothing$
and~$\partial\Pi=\varnothing$.

\begin{defi}
We say that a rectangle~$r=[\theta_1;\theta_2]\times[\varphi_1;\varphi_2]\in\Pi$
is \emph{$\theta$-thin} (respectively, \emph{$\varphi$-thin}) (with respect to~$T(X,Y)$) if
$(\theta_1;\theta_2)$ is disjoint from~$\Theta(\Pi)\cup X$ (respectively,
$(\varphi_1;\varphi_2)$ is disjoint from~$\Phi(\Pi)\cup Y$).
In both cases we say that it is \emph{thin}.

A rectangle~$r=[\theta_1;\theta_2]\times[\varphi_1;\varphi_2]$ of an oriented
rectangular diagram of a surface~$(\Pi,\epsilon)$ is called \emph{positively pre-thin}
(respectively, \emph{negatively pre-thin}) if either~$(\theta_1;\theta_2)\cap X=\varnothing$ and~$\epsilon(r)=1$
or~$(\varphi_1;\varphi_2)\cap Y=\varnothing$ and~$\epsilon(r)=-1$
(respectively, $(\theta_1;\theta_2)\cap X=\varnothing$ and~$\epsilon(r)=-1$
or~$(\varphi_1;\varphi_2)\cap Y=\varnothing$ and~$\epsilon(r)=1$).

If at least one of the conditions~$(\theta_1;\theta_2)\cap X=\varnothing$
and~$(\varphi_1;\varphi_2)\cap Y=\varnothing$ holds, then~$r$ is called \emph{pre-thin}.
\end{defi}

If~$r\in\Pi$ is a pre-thin rectangle, then the intersection of~$\widehat r$ with a simplex of~$T(X,Y)$ meets
one of the edges twice, so $\widehat\Pi$ is not normal. Conversely, if~$\widehat\Pi$
is not normal and~$\partial\Pi=\varnothing$, then~$\Pi$ contains a pre-thin rectangle.
In this case, we define a normalization procedure for~$\Pi$
as follows.

\begin{defi}\label{normalize-def}
Suppose that~$\Pi$ contains a pre-thin rectangle
with respect to~$T(X,Y)$. Then, by the `innermost' argument, $\Pi$ contains a thin
rectangle~$r=[\theta_1;\theta_2]\times[\varphi_1;\varphi_2]$ and one of the following cases occurs,
in each of which we define a \emph{normalization step}~$\Pi\mapsto\Pi'$.

\smallskip\noindent\emph{Case 1}: $r$ is a part of a trivial spherical component of~$\Pi$. We remove 
this spherical component and take the result for~$\Pi'$.

\smallskip\noindent\emph{Case 2}: either $r$ is $\theta$-thin and has no $\theta$-companion in~$\Pi$,
or~$r$ is $\varphi$-thin and has no $\varphi$-companion in~$\Pi$. In this case, we define~$\Pi\mapsto\Pi'$
to be the corresponding generalized wrinkle reduction.

\smallskip\noindent\emph{Case 3}: either~$r$ is $\theta$-thin and has a $\theta$-companion in~$\Pi$,
or~$r$ is $\varphi$-thin and has a $\varphi$-companion in~$\Pi$. If~$\Pi$ contains also
a rectangle~$r'$ of the form~$[\theta_2;\theta_1]\times[\varphi';\varphi'']$ or~$[\theta';\theta'']\times[\varphi_2;\varphi_1]$,
respectively, we let~$\Pi\mapsto\Pi'$ be the bubble reduction that removes~$r'$. Otherwise, 
we define~$\Pi\mapsto\Pi'$
to be the corresponding generalized compression.

\smallskip
A sequence~$\Pi\mapsto\Pi_1\mapsto\Pi_2\mapsto\ldots\mapsto\Pi_n$
of normalization steps is referred as a \emph{normalization procedure for~$\Pi$}. If~$\widehat\Pi_n$
is normal with respect to~$T(X,Y)$, we say that \emph{$\Pi$ normalizes to~$\Pi_n$}.
A normalization procedure is \emph{positive} (respectively, \emph{negative}) if the transformations
at all steps of the procedure are positive (respectively, negative).
In the case of
a positive (respectively, negative)
normalization we also say that~$\Pi$ normalizes to any other surface obtained from~$\Pi_n$ by a
positive (respectively, negative) deformation.\end{defi}

When~$\Pi$ is finite, any normalization procedure for~$\Pi$ reduces the number of rectangles
in the diagram, hence it eventually stops. However, we are going to use normalization
in the case of infinite rectangular diagrams, which requires more care.

\begin{lemm}\label{only-positive-lem}\emph{(i)}
Suppose that~$\Pi$ contains a positively pre-thin rectangle, but does not contain a negatively pre-thin one. Then any normalization procedure for~$\Pi$ is positive.

\emph{(ii)}
Suppose that~$\Pi$ contains exactly one pre-thin rectangle \emph(which may be $\theta$-thin
and $\varphi$-thin simultaneously\emph), and let~$\Pi\mapsto\Pi'$
be a positive normalization step. Then any normalization procedure for~$\Pi$
starting from this step is positive.

The same statements are true after exchanging `positive' and `negative'.
\end{lemm}

\begin{proof}
One establishes, by a routine check of several cases, that:\begin{enumerate}
\item
to admit a positive (respectively, negative) normalization step an oriented rectangular diagram must
have a positively (respectively, negative) pre-thin rectangle;
\item
 a positive normalization step never
produces a negatively pre-thin rectangles.
\end{enumerate}
We leave it to the reader.
\end{proof}

For a short while, we forget about triangulations and prove the following statement.

\begin{lemm}\label{no-compres-lem}
Let~$\Pi_0\mapsto\Pi_1\mapsto\Pi_2\mapsto\ldots\mapsto\Pi_{n-1}\mapsto\Pi_n$ be a sequence of
positive moves that may include flypes, bubble moves, compressions, and removals of
trivial spherical components, and let~$M$ be a submanifold of~$\mathbb S^3$ containing
all the bordisms associated with the moves~$\Pi_{i-1}\mapsto\Pi_i$. Suppose that the surfaces~$\widehat\Pi_0$ and~$\widehat\Pi_n$
are connected and incompressible in~$M$ or homeomorphic to the two-sphere.
Suppose also that the interiors of the bordisms associated with the moves~$\Pi_{i-1}\mapsto\Pi_i$, $i=1,\ldots,n$,
are pairwise disjoint.

Then there is another sequence~$\Pi_0\mapsto\Pi_1'\mapsto\Pi_2'\mapsto\ldots\mapsto\Pi_{n-1}'\mapsto\Pi_n$
of positive moves that include only flypes and bubble moves, and is such that
the union of the bordisms associated with the new sequence is the same
as that for the original one.
\end{lemm}

\begin{proof}
The proof is by induction in the pair~$(k,l)$, where~$k$ is the number of compressions
in the sequence, and~$l$ is the number of moves occurring after the last compression (if~$k=0$, then~$l$
is also set to zero). In the case~$k=0$ there is nothing to prove, this is the induction base.

The point is that, in most cases, compressions can be postponed while
other transformations moved forward.
More precisely, let~$\Pi_i\mapsto\Pi_{i+1}$ be the last compression in the original sequence of moves.
It must be followed by another move~$\Pi_{i+1}\mapsto\Pi_{i+2}$, since~$\widehat\Pi_n$ is connected and~$\widehat\Pi_0$ is incompressible.
In most cases,
the moves can be exchanged in the sense that there exists a move~$\Pi_i\mapsto\Pi_{i+1}'$ of the same
type as~$\Pi_{i+1}\mapsto\Pi_{i+2}$ such that~$\Pi_{i+1}'\mapsto\Pi_{i+2}$ is a compression. For instance,
this is always the case when~$\Pi_{i+1}\mapsto\Pi_{i+2}$ is a flype. (It is crucial here that
all moves in question are positive, and the bordisms associated to them have pairwise disjoint
interiors.)

There are only three exceptions in which such `commutation' of moves does not apply, namely,
when $\Pi_{i+1}\mapsto\Pi_{i+2}$ is one of the following:
\begin{enumerate}
\item
a removal of a trivial spherical component created by the compression;
\item
a bubble creation move that splits one of the rectangles created by the compression;
\item
a bubble reduction move that removes one of the rectangles created by the compression.
\end{enumerate}
In the first case, one can obtain~$\Pi_{i+2}$ by two consecutive bubble reduction moves.

In the second case, $\Pi_{i+2}$ is obtained by first applying a generalized wrinkle creation move,
and then a compression.

In the third case, the diagram~$\Pi_i$ admits two positive compressions, each of which
can be followed by a bubble reduction. Find the longest subsequence~$\Pi_i\mapsto\Pi_{i+1}\mapsto\ldots\mapsto\Pi_j$ 
in which every move except~$\Pi_i\mapsto\Pi_{i+1}$ is a bubble reduction removing a rectangle
created at the previous step. Geometrically, this means that~$\widehat\Pi_i$ has a `thin tube of length~$j-i-1$'
that can be cut at~$j-i$ different places by a compression, and the rest of the tube
is shrunk by bubble reduction moves. We then look at the move~$\Pi_j\mapsto\Pi_{j+1}$ and
proceed similarly to the previous cases viewing the removal of the entire tube
as a single operation. This gives the induction step.

One can see that every `exchange of moves' preserve the union of bordisms associated to them.
\end{proof}

Now we make use of a powerful tool of the normal surface theory, which is almost normal surfaces,
and look at them from the point of view of the rectangular diagram context. Almost normal surfaces
were introduced by H.\,Rubinstein in~\cite{rubinstein}. We utilize the version of the concept
that was used by A.\,Thompson in~\cite{thompson} and the idea of thin position, which is due to D.\,Gabai~\cite{gabai-87}.

Recall that a surface~$F$ in a three-manifold~$M$ is said to be \emph{almost normal} with respect to a
triangulation~$T$ of~$M$ if any connected component of the intersection of~$F$ with any three-simplex of~$T$
is a normal disc, with exactly one exception which is a disc intersecting the one-skeleton of~$T$
exactly eight times in such a way that four edges of the three-simplex containing this disc
are met once, and the other two edges twice. Such a disc is referred as an \emph{octagon}.
The notion of \emph{normal isotopic} surfaces is extended to almost normal surfaces
accordingly.

\begin{prop}\label{almost-norm-rect-prop}
Let~$X,Y\subset\mathbb S^1$ be finite susbsets with~$|X|,|Y|\geqslant3$, and let~$F\subset\mathbb S^3$
be an almost normal oriented surface with respect to~$T(X,Y)$ such that every connected component of~$F$
has a non-empty intersection with both circles~$\mathbb S^1_{\tau=0}$ and~$\mathbb S^1_{\tau=1}$. Then one of
the following holds true\emph:
\begin{enumerate}
\item
$F$ is normal isotopic to a surface of the form~$\widehat\Pi$, where~$\Pi$ is a rectangular diagram of a surface
having exactly one pre-thin rectangle, which is both $\theta$-thin and $\varphi$-thin\emph;
\item
there exist a surface~$F'$ normal isotopic to~$F$ and a positive flype~$\Pi\mapsto\Pi'$ such that~$F'$
is isotopic to~$\widehat\Pi$ and~$\widehat\Pi'$ and
contained in the bordism associated with the flype. Moreover, $\Pi$ has no positively pre-thin rectangles, and~$\Pi'$
has no negatively pre-thin rectangles.
\end{enumerate}
\end{prop}

\begin{proof}
Let~$\Delta=[\theta_1;\theta_2]*[\varphi_1;\varphi_2]$ be the simplex of~$T(X,Y)$
containing the octagon of~$F$.
There are two cases to consider.

\smallskip\noindent\emph{Case 1}: the octagon of~$F$ intersects each of the circles~$\mathbb S^1_{\tau=0,1}$
twice. Let~$\theta_0\in(\theta_1;\theta_2)$ and~$\varphi_0\in(\varphi_1;\varphi_2)$
be points on~$\mathbb S^1_{\tau=1}$ and~$\mathbb S^1_{\tau=0}$, respectively, separating the two
intersections of the boundary of the octagon with the respective arcs. Then the surface~$F$
is normal isotopic (with respect to~$T(X,Y)$) to a surface~$F'$ that is normal with respect to~$T(X\cup\{\theta_0\},Y\cup\{\varphi_0\})$;
see Figure~\ref{octagon-1-fig}.
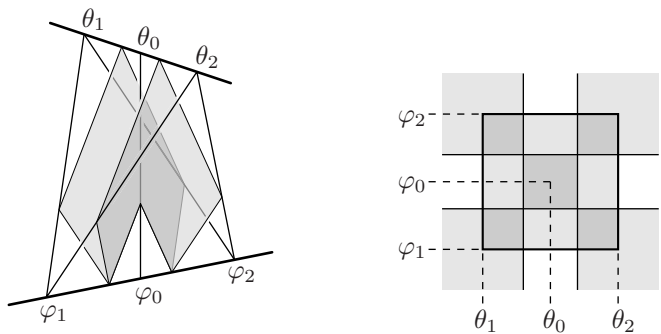
\begin{figure}[ht]
 \includegraphics[scale=.5]{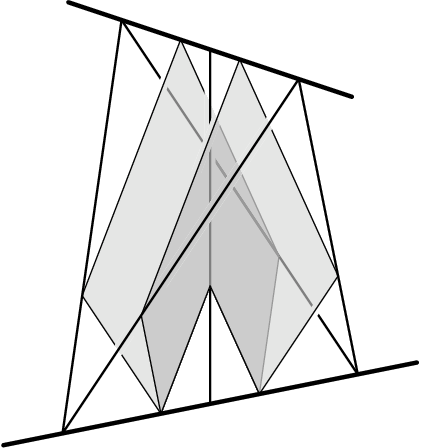}
 \put(-89,-4){$\varphi_1$}\put(-18,11){$\varphi_2$}\put(-53,3){$\varphi_0$}
 \put(-73,106){$\theta_1$}\put(-30,92){$\theta_2$}\put(-52,99){$\theta_0$}
 \hskip2cm    \includegraphics[scale=1.8]{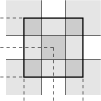}
 \put(-70,-8){$\theta_1$}\put(-44,-8){$\theta_0$}\put(-18,-8){$\theta_2$}
 \put(-99,20){$\varphi_1$}\put(-99,45){$\varphi_0$}\put(-99,70){$\varphi_2$}
\caption{The octagon in Case~1}\label{octagon-1-fig}
\end{figure}
By Proposition~\ref{normal=>rect-prop}, the surface~$F$ is normal isotopic to a surface of the form~$\widehat\Pi$,
where~$\Pi$ is a rectangular diagram of a surface. One can see that~$\Pi$ will have exactly
one pre-thin rectangle, which is both $\theta$-thin and $\varphi$-thin.

\smallskip\noindent\emph{Case 2}: the octagon of~$F$ intersects each of the circles~$\mathbb S^1_{\tau=0,1}$
just once. We proceed similarly to the proof of Proposition~\ref{normal=>rect-prop}.
The octagon can also be assumed to contain an arc of the form~$\widehat v$, which
cuts it into halves. Unlike the case of a vertical normal disc, the projection of each half
on~$\mathbb T^2$ is not a rectangle, but a `boomerang-shaped' hexagon;
see Figure~\ref{octagon-2-fig}.
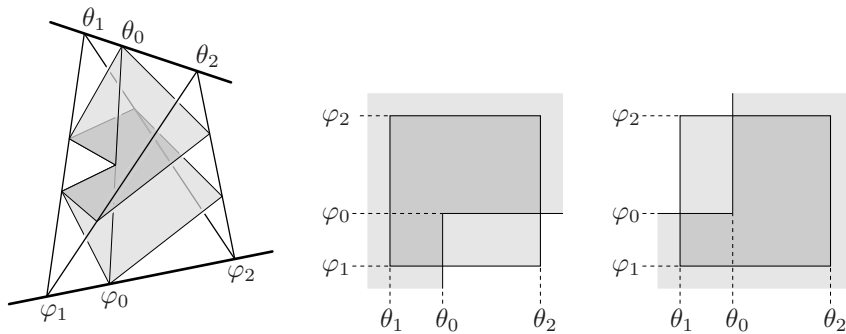
\begin{figure}[ht]
\includegraphics[scale=.5]{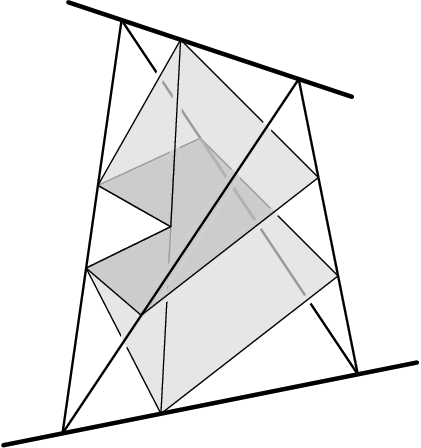}
 \put(-89,-4){$\varphi_1$}\put(-18,11){$\varphi_2$}\put(-66,1){$\varphi_0$}
 \put(-73,106){$\theta_1$}\put(-30,92){$\theta_2$}\put(-58,101){$\theta_0$}
\hskip1cm
\includegraphics[scale=1]{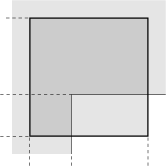}
 \put(-70,-8){$\theta_1$}\put(-49,-8){$\theta_0$}\put(-12,-8){$\theta_2$}
 \put(-92,13){$\varphi_1$}\put(-92,33){$\varphi_0$}\put(-92,70){$\varphi_2$}
\hskip1cm
\includegraphics[scale=1]{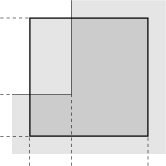}
 \put(-70,-8){$\theta_1$}\put(-49,-8){$\theta_0$}\put(-12,-8){$\theta_2$}
 \put(-92,13){$\varphi_1$}\put(-92,33){$\varphi_0$}\put(-92,70){$\varphi_2$}
\caption{The octagon in Case~2 and the projections of its two halves on~$\mathbb T^2$}\label{octagon-2-fig}
\end{figure}
Thus, for every connected component~$V$ of~$F\setminus\bigcup_i\alpha_i$ with two exceptions,
one can define a rectangle~$r(V)$ exactly as in the proof of Proposition~\ref{normal=>rect-prop}.
For two exceptional components, the same procedure produces non-convex
hexagons having five right angles. These hexagons intersect as shown in Figure~\ref{boomerangs-fig}.
\begin{figure}[ht]
\includegraphics{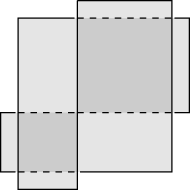}
\caption{The exceptional hexagonal regions}\label{boomerangs-fig}
\end{figure}

Now let~$\Pi\mapsto\Pi'$ be a flype such that the common part of~$\Pi$ and~$\Pi'$ is the set of all
rectangles of the form~$r(V)$, where~$V$ runs over the set of all non-exceptional connected
components of~$F\setminus\bigcup_i\alpha_i$, and the `boomerangs' are,
in the notation of Definition~\ref{flype_move-def}, the following union of rectangles:
$$r_1\cup r_* \cup r_3 = r_1'\cup r_*'\cup r_3'\text{ and }r_4\cup r_*\cup r_2 = r_4'\cup r_*' \cup r_2'.$$

It is not hard to verify that~$F$ is now isotopic to a surface contained
in the bordism associated to the flype~$\Pi\mapsto\Pi'$. We choose the orientation of~$\Pi$
and~$\Pi'$ so that it agrees with that of~$F$ on~$\Pi\cap\Pi'$. Exchanging~$\Pi$
and~$\Pi'$ if necessary, we may assume that the flype~$\Pi\mapsto\Pi'$ is positive.

By construction, the only rectangles of~$\Pi$ and~$\Pi'$ that may be pre-thin are those that are involved
in the flype, namely, this can only be rectangles~$r_2$, $r_3$, $r_2'$, and~$r_3'$, in the notation of Definition~\ref{flype_move-def}.
It is an easy check that the rectangles~$r_2,r_3\in\Pi$ can only be negatively pre-thin, and~$r_2',r_3'\in\Pi'$
can only be positively pre-thin.
\end{proof}

\begin{prop}\label{shrink-sphere-prop}
Let~$\Pi$ be an oriented rectangular diagram of a two-sphere. Then there is a sequence~$\Pi=\Pi_0\mapsto\Pi_1\mapsto\Pi_2\mapsto\ldots
\mapsto\varnothing$ of positive moves that may include flypes, bubble  moves, compressions, removals of trivial spherical components, and inverses of compressions and trivial spherical component removals.
\end{prop}

\begin{proof}
Let~$B\subset\mathbb S^3$ be the three-ball defined by~$\partial B=-\widehat\Pi$.
Choose a triangulation of~$\mathbb S^3$ of the form~$T(X,Y)$ so that~$\widehat\Pi$ is normal
with respect to it. Define \emph{the complexity}~$c(\Pi)$ of~$\Pi$ as the maximal number~$n$
of normal two-spheres~$S_1,S_2,\ldots,S_n\subset B{}\setminus\partial B$ that are pairwise disjoint, pairwise normally non-parallel,
and not normally parallel to~$\widehat\Pi$. We assume that~$X$ and~$Y$ are chosen to minimize
this complexity.

We proceed by induction in the just defined complexity of~$\Pi$.
Suppose that~$c(\Pi)=0$. According to Thompson~\cite{thompson} there is an almost normal
sphere~$S$ contained in~$B$. It follows from Proposition~\ref{almost-norm-rect-prop} and Lemma~\ref{only-positive-lem}
that there are rectangular diagrams~$\Pi'$ and~$\Pi''$ such that the following holds:
\begin{enumerate}
\item
$\Pi'$ normalizes negatively to~$\Pi$ (or~$\widehat\Pi'$ is normal parallel to~$\widehat\Pi$ from the beginning);
\item
$\Pi''$ normalizes positively to~$\varnothing$;
\item
either~$\Pi'=\Pi''$ or~$\Pi'\mapsto\Pi''$ is a positive flype.
\end{enumerate}
The claim now follows from Lemmas~\ref{wrinkle-decomp-lem} and~\ref{gen-compres-decomp-lem}.
The induction base is done.

Now suppose that~$c(\Pi)>0$. There are two cases to consider.

\smallskip\noindent\emph{Case 1}: $B$ contains a vertex of~$T(X,Y)$. Then there is a positive
bubble creation move~$\Pi\mapsto\Pi'$ that results in~$\widehat\Pi$ jumping over
a vertex of~$T(X,Y)$. The diagram~$\Pi'$ may have positively pre-thin rectangles (but
not negatively pre-thin ones). In this case, let it normalize positively
to some diagram~$\Pi''$. Otherwise, put~$\Pi''=\Pi'$. 

The diagram~$\Pi''$ will have smaller complexity, so the induction step follows from
Lemmas~\ref{wrinkle-decomp-lem} and~\ref{gen-compres-decomp-lem}.

\smallskip\noindent\emph{Case 2}: $B$ contains no vertices of~$T(X,Y)$. Let~$S$ be an innermost normal
two-sphere contained in~$B$. Since there are no vertices of the triangulation in~$B$, this sphere
must intersect both circles~$\mathbb S^1_{\tau=0}$ and~$\mathbb S^1_{\tau=1}$. Therefore,
we may assume that~$S=\widehat\Sigma$, where~$\Sigma$ is a rectangular diagram of a surface.

$S=\widehat\Sigma$ bounds a ball~$B'\subset B$, which we endow with the induced orientation from~$\mathbb S^3$.
Orient~$\Sigma$ so that~$\partial B'=\widehat\Sigma$.

By  Lemma~\ref{connected-sum-lem} there is a rectangular diagram of a surface~$\Pi'$ such that~$\widehat\Pi'$
is contained in~$\overline{B\setminus B'}$ and~$\Pi'\mapsto\Pi\cup\Sigma$ is a negative compression.
By the induction base, the negative removal of a spherical component~$\Pi\cup\Sigma\mapsto\Pi$
can be decomposed into negative moves that include flypes, bubble moves, compressions, and removals of
trivial spherical components. If the surface~$\widehat\Pi'$ is normal with respect to~$T(X,Y)$, then we put~$\Pi''=\Pi'$.
Otherwise, normalize~$\Pi'$ to another diagram~$\Pi''$. $\Pi'$ can have only positively pre-thin
rectangles, hence the normalization is positive. The diagram~$\Pi''$ is clearly simpler than~$\Pi$, which gives the induction step in this case.
\end{proof}

Proposition~\ref{shrink-sphere-prop} and Lemma~\ref{no-compres-lem} imply the following statement.

\begin{prop}\label{only-flypes-bubble-prop}
\emph{(i)} Let~$\Pi$ be an oriented rectangular diagram of a two-sphere. Then there is a sequence
of positive flypes, positive bubble moves, and a single positive removal of a trivial component that produces~$\varnothing$
from~$\Pi$.

\emph{(ii)}
Let~$\Pi$ be a rectangular diagram of a connected surface, and let~$\Pi\mapsto\Pi'$ and~$\Pi'\mapsto\Pi''$
be a positive compression and a positive removal of a spherical component, respectively. Then one
can obtain~$\Pi''$ from~$\Pi$ by means of a sequence of positive flypes and positive bubble moves.
\end{prop}

We are ready to proceed with the proof of the main result.

\begin{proof}[Proof of Theorem~\ref{diagram-exist-thm}]
First, we explain the strategy. Let~$k$ be the depth of~$\mathscr F$. Take for~$X_0$ and~$Y_0$
minimal subset of~$\mathbb S^1$ such that~$R\subset X_0\times Y_0$.

For each~$j=1,2,\ldots,k$, inductively, we define~$X_j,Y_j\subset\mathbb S^1$ so that any leaf
of depth~$j$ can be normalized with respect to~$T(X_j,Y_j)$. To do so, we use the previously normalized
leaves of depth~$\leqslant j-1$. We choose a maximal family of leaves of depth~$\leqslant j$ that can be normalized
to become pairwise normally non-parallel surfaces, normalize them, and proceed to the next~$j$.

Then we insert cavities so that all remaining leaves of~$\mathscr F$ are normalized with
respect to~$T(X_k,Y_k)$ except those that branch at a cavity. (The exceptional,
branching, leaves can also be said to be normalized, but in a non-unique way.) Finally,
all rectangles of the corresponding rectangular diagrams of surfaces are collected
into finitely many packs.

There are a few issues to address. First, we explain how to normalize individual leaves of~$\mathscr F$.
We call two leaves~$F$ and~$F'$ \emph{equivalent} if~$F=F'$ or there is an
embedding~$\iota:{}F\times[0;1]\rightarrow\mathbb S^3$ that takes~$F\times0$ and~$F\times1$
to~$F$ and~$F'$, respectively, and each arc of the form~$x\times[0;1]$ to a transversal of~$\mathscr F$. Clearly, in this case, $F$ and~$F'$ have the same depth, and~$\overline F\setminus F=\overline F'\setminus F'$.
Any leaf in~$\iota(F\times(0;1))$ which is not equivalent to~$F$ is said to be \emph{subordinary to~$F$ and~$F'$}.
Leaves of~$\mathscr F$ that are not subordinary to any other leaf are called \emph{principal}.

The general case can be reduced to the one in which all leaves are principal, so we put this additional assumption
for time being.

At the $j$th step of the procedure described above, we are going to normalize at least one leaf from
each equivalence class of leaves of depth~$j$. By this we mean that the foliation~$\mathscr F$ is modified
by an isotopy relative to the union of previously normalized leaves so that the selected leaves become normal with
respect to~$T(X_j,Y_j)$ and, after the modification, have the form~$\widehat\Pi$, where~$\Pi$ is a rectangular
diagram of a surface.

Since~$\mathscr F$ is taut, the only depth~$0$
leaves of~$\mathscr F$ are the connected components of~$\widehat\Omega_\varepsilon(R)$,
which are normal from the beginning with respect to~$T(X_0,Y_0)$ and are represented by a rectangular diagram.

Let~$1\leqslant j\leqslant k$.
Suppose that a representative of every equivalence class of leaves of depth~$\leqslant j-1$
has been normalized with respect to~$T(X_{j-1},Y_{j-1})$. We denote the union of the leaves
that have already been normalized
by~$L_{j-1}$. By a small deformation of~$\mathscr F$ fixed on~$L_{j-1}$ we can ensure
that~$\mathscr F$ has only finitely many tangencies with the circles~$\mathbb S^1_{\tau=0}$
and~$\mathbb S^1_{\tau=1}$. Define~$X_j\supset X_{j-1}$ and~$Y_j\supset Y_{j-1}$ so that these
points become vertices of~$T(X_j,Y_j)$. Note that we modify~$\mathscr F$ at every step by an isotopy,
so new points of tangency with the circles~$\mathbb S^1_{\tau=0,1}$ can arise.

Call an intersection point of a leaf~$F$ with the circle~$\mathbb S^1_{\tau=0}$ or~$\mathbb S^1_{\tau=1}$
\emph{positive} if the co-orientation of~$F$ at this point coincides
with the orientation of the circle, and \emph{negative} otherwise. Due to our choice of~$X_j$ and~$Y_j$,
for any leaf~$F$ of~$\mathscr F$ and any edge~$e$ of~$T(X_j,Y_j)$
contained in~$\mathbb S^1_{\tau=0}\cup\mathbb S^1_{\tau=1}$, either all intersections of~$F$ with~$e\setminus\partial e$ are positive, or all are negative. We then call the edge~$e$ positive or negative, accordingly.

The leaves in~$L_{j-1}$ are not tangent to the circles~$\mathbb S^1_{\tau=0}$ and~$\mathbb S^1_{\tau=1}$,
hence they are normal with respect to the new triangulation~$T(X_j,Y_j)$. Therefore,
there is an open neighborhood~$U_j$
of $L_{j-1}$ such that, for any leaf~$F$ of~$\mathscr F$
and any $3$-simplex of~$T(X_j,Y_j)$, any connected component of~$F\cap\Delta$ contained in~$U_j$
is a normal disc.

Since~$L_{j-1}$ contains at least one representative from each equivalence class of leaves of
depth at most~$j-1$, by a deformation of~$\mathscr F$ fixed on~$L_{j-1}$ we may also ensure that any leaf of depth smaller
than~$j$ that is contained in the closure of a depth~$j$ leaf is also contained in~$U_j$.

Let~$F$ be a leaf of depth~$j$. As follows from the above discussion, $F$ can be divided into
two parts with common boundary, $F'$ and~$F''$ say,
such that~$F'$ is covered by normal discs, and~$F''$ is compact. Moreover,
we can ensure that~$\partial F'=\partial F''$
consists of arcs of of the form~$\widehat v$, that is, the common boundary of~$F'$ and~$F''$ is a link
represented by a rectangular diagram (see the proof of Proposition~\ref{normal=>rect-prop}). Also,
we may assume that~$F'$ has the form~$\widehat\Pi'$, where~$\Pi'$ is a rectangular diagram of a surface.

Now we claim that~$F''$ is isotopic relative~$F'\cup L_{j-1}$ to a surface associated with a rectangular
diagram~$\Pi''$. This follows from a slightly strengthen version of~\cite[Proposition~5]{Representability}. Namely,
the isotopy in~\cite{Representability} is assumed to be fixed on a graph of the form~$\bigcup_{v\in A}\widehat v$,
where~$A$ is a finite subset of~$\mathbb T^2$, whereas here we require it to be fixed on a union of surfaces.
The proof, however, requires only a small modification at the place, where finger moves are used. Namely, one
has to define more general finger moves, which could push the surface not only along the radial direction,
but along any simple arc contained in a single page, and then choose such arcs to avoid~$F'\cup L_{j-1}$.

Thus, the entire leaf~$F$ can be deformed into a surface of the form~$\widehat\Pi$ without
disturbing the leaves contained in~$L_{j-1}$. Now we run a normalization procedure for~$\Pi$
with respect to~$(X_j,Y_j)$ according to Definition~\ref{normalize-def}. The procedure arrives
at a $T(X_j,Y_j)$-normal surface in finitely many steps for the following reason.

Each normalization step reduces a pair of intersections of the surface~$\widehat\Pi$ with an edge of~$T(X_j,Y_j)$
contained in~$\mathbb S^1_{\tau=0}\cup\mathbb S^1_{\tau=1}$, one of which is positive and the other negative.
The surface~$F$ does not have positive intersections with negative edges and negative intersections
with positive edges, hence~$\widehat\Pi$ may have only finitely many such `incorrect' intersections, since it differs from~$F$
only in a compact part.
Therefore, only finitely many steps of the normalization can be made.
Since the folitation~$\mathscr F$ is taut, the leaf~$F$ is incompressible in~$\mathbb S^3\setminus N_\varepsilon(R)$, so the normalization will result in a surface isotopic to~$F$
relative to~$L_{j-1}$.

In a similar fashion, we can normalize, with respect to~$T(X_j,Y_j)$, any finite collection of depth~$j$ leaves of~$\mathscr F$.
By a generalization of the classical Kneser's argument~\cite{kneser}, under the assumption
that all leaves are principal, the total number of depth~$j$
leaves that can be normalized simultaneously with respect to a fixed triangulation so that no two become normal parallel
is bounded by a number depending only on the triangulation.

Indeed, let~$N$ be a union of finitely many depth~$j$ leaves of~$\mathscr F$ that are normal with respect to~$T(X_j,Y_j)$,
and suppose that no two leaves in~$N$ are normal parallel. In order to obtain Kneser's estimate on the number
of connected components of~$N$ it suffices to show that there is no path-connected
component~$M$ of~$\mathbb S^3\setminus(N_\varepsilon(R)\cup N)$ such that each connected component of the intersection of~$M$
with every three-simplex~$\Delta$ of the triangulation is bounded in~$\Delta$ by two `parallel' normal discs.
If such a component exists, it has the form~$\iota(F\times(0;1))$, where~$F$ is a leaf of~$\mathscr F$,
and~$\iota:F\times[0;1]\rightarrow\mathbb S^3\setminus N_\varepsilon(R)$ is an immersion such that~$\iota(F\times0)=
\iota(F\times1)=F$. Then the image~$\iota(F\times[0;1])$ has the form~$\iota'(F'\times(0;1))$, where~$F'$
is another leaf of~$\mathscr F$, and~$\iota':F'\times[0;1]\rightarrow\mathbb S^3\setminus N_\varepsilon(R)$
is an embedding such that~$\iota'(F'\times0)=F'$. This means that the leaf~$F$ is subordinate to~$F'$, a contradiction.

Thus, we end up with the triangulation~$T(X_k,Y_k)$ and a maximal collection of pairwise normally non-parallel leaves
normalized with respect to~$T(X_k,Y_k)$, whose union is denoted by~$L_k$ and contains a representative of
every equivalence class of leaves. The closure~$\overline L_k$ also is a union of normalized leaves
(some of which may be normally parallel).

Now let~$M$ be a connected component of~$\mathbb S^3\setminus(N_\varepsilon(R)\cup\overline L_k)$.
Since~$L_k$ contains a representative from each equivalence class of leaves, the manifold~$M$
has the form~$F\times(0;1)$, where~$F$ is an abstract connected surface, and there is an immersion~$\psi:F\times[0;1]\rightarrow\mathbb
S^3$ that takes~$F\times(0;1)$ onto~$M$ homeomorphically. We assume this immersion is orientation-preserving,
and let~$F_0$ and~$F_1$ be the images of~$F\times0$ and~$F\times1$, respectively. (We may have~$F_0=F_1$.
Otherwise, $\psi$ is an embedding.) By construction, there are rectangular diagrams
of surfaces~$\Pi_0$ and~$\Pi_1$ such that~$F_0=\widehat\Pi_0$ and~$F_1=\widehat\Pi_1$.

We now claim that there are rectangular diagrams of surfaces~$\Pi_0'$ and~$\Pi_1'$ such
that:
\begin{enumerate}
\item
the associated surfaces~$\widehat\Pi_0'$ and~$\widehat\Pi_1'$ are contained in~$M$
and are normal parallel to~$\widehat\Pi_0$ and~$\widehat\Pi_1$,
respectively;
\item
$\Pi_1'$ is obtained from~$\Pi_0'$
by a sequence of positive flypes and positive bubble moves.
\end{enumerate}
To see this, there are a few cases to consider.

First, suppose that~$M$ contains a vertex~$q$ of~$T(X_k,Y_k)$.
We say that~$q$ is \emph{visible from~$\Pi_0$} (respectively, from~$\Pi_1$) if there is
an arc~$\alpha\subset\overline M$ contained in an edge of the triangulation~$T(X_k,Y_k)$ of the form~$\widehat v$, $v\in X_k\times Y_k$,
such that one of the endpoints of~$\alpha$ is~$q$, and the other lies in~$\Pi_0$ (respectively, in~$\Pi_1$).

Suppose that~$M$ contains a vertex of~$T(X_k,Y_k)$ visible from~$\Pi_0$.
For a small enough~$\delta>0$ the surface~$\psi(F\times\delta)$ is a leaf close to~$F_0$, and hence normal.
We take for~$\Pi_0'$ the rectangular diagram constructed from~$\psi(F\times\delta)$
as in the proof of Proposition~\ref{normal=>rect-prop}. Clearly, this diagram is obtained from~$\Pi_0$
by a small positive deformation.

Since~$M$ contains a vertex of the triangulation visible from~$\Pi_0$,
$\Pi_0'$ admits a positive bubble creation move~$\Pi_0'\mapsto\Pi_0''$
that results in the surface~$\widehat\Pi_0'$ jumping over a vertex of~$T(X_k,Y_k)$ while staying disjoint from~$F_1$.
The diagram~$\Pi_0''$ may have positively pre-thin rectangles, but not negatively pre-thin ones.
Let~$\Pi_1'$ be
a diagram obtained from~$\Pi_0''$ by a positive normalization (or let~$\Pi_1'=\Pi_0''$
if~$\widehat\Pi_0''$ is normal). By Lemmas~\ref{wrinkle-decomp-lem}, \ref{gen-compres-decomp-lem},
and~\ref{no-compres-lem}, $\Pi_1'$
can be obtained from~$\Pi_0'$ by a sequence of positive flypes and positive bubble moves.

Note that the bordisms associated with all the involved flypes and bubble moves are contained in~$M$.
Indeed, since~$\widehat\Pi_0''$ is disjoint from~$F_1$,
the union~$\Pi_0''\cup\Pi_1$
is also a rectangular diagram of a surface. This implies that~$\Pi_1'\cup\Pi_1$ is obtained from~$\Pi_0'\cup\Pi_1$
by a positive normalization procedure. Therefore, the moves used to transform~$\Pi_0'$ to~$\Pi_1'$ never produce a surface
that intersects~$\widehat\Pi_1$.

Since~$\widehat\Pi_1'$ is normal, it must be normal parallel to one of the leaves contained in~$L_k$.
Since~$\widehat\Pi_1'\subset M$, and~$\Pi_1'$ is obtained from~$\Pi_0'$ by a non-trivial sequence
of positive moves, $\widehat\Pi_1'$ can be normal parallel only to~$\widehat\Pi_1$.

If~$M$ does not contain a vertex of~$T(X_k,Y_k)$ visible from~$\Pi_0$, then it must contain a vertex visible from~$\Pi_1$.
In this case, we proceed similarly to the previous one, exchanging the roles of~$\Pi_0$ and~$\Pi_1$: push~$\Pi_1$ into~$M$, then apply a negative
bubble creation move, and then normalize negatively to a surface that is normal parallel to~$\widehat\Pi_0$.

The case when~$M$ contains a vertex of~$T(X_k,Y_k)$ is done.

Now suppose that~$M$ contains no vertices of~$T(X_k,Y_k)$.
The surface~$F$ is not compact, but `at infinity' the manifold~$M$ becomes `thin', which formally means that~$F$
can be divided into two parts~$F'$ and~$F''$ such that, for all~$x\in[0;1]$, $\psi(F'\times x)$ is covered by normal discs contained in~$\psi(F\times x)$, and~$F''$
is compact. Moreover, we may assume that the preimage of the one-skeleton of~$T(X_k,Y_k)$ under~$\psi|_{F'\times[0;1]}$
consists of arcs of the form~$p\times[0;1]$.

Now we consider all possible ways to change the foliation~$\mathscr F$ by an isotopy supported
in~$\psi(F''\times(0;1))$ and apply the thin position argument from~\cite{gabai-87,thompson}.
The feature of the present situation compared to~\cite{gabai-87,thompson} is that
the \emph{width} of the foliation is defined by counting only those intersections of leaves
with the edges of~$T(X_k,Y_k)$ that are contained in~$\psi(F''\times(0;1))$.
Following the lines of the proof of Lemma~4 in~\cite{thompson}, one can show that one of the following
holds:
\begin{enumerate}
\item
$\mathscr F$ can be changed by an isotopy supported in~$\psi(F''\times(0;1))$ so that
a leaf~$L$ of~$\mathscr F$ contained in~$M$ becomes almost normal with respect to~$T(X_k,Y_k)$;
\item
there exists a $T(X_k,Y_k)$-normal sphere contained in~$M$.
\end{enumerate}

Suppose that the former occurs. Then it follows from Proposition~\ref{almost-norm-rect-prop}
and Lemma~\ref{only-positive-lem}
that there exist rectangular diagrams~$\Pi_0''$ and~$\Pi_1''$ such that:
\begin{enumerate}
\item
$\Pi_0''$ normalizes negatively to a rectangular diagram~$\Pi_0'$;
\item
$\Pi_1''$ normalizes positively to a rectangular diagram~$\Pi_1'$;
\item
either~$\Pi_0''\mapsto\Pi_1''$ is a positive flype with associated bordism containing a surface
which is normal isotopic to~$L$, or~$\Pi_0''=\Pi_1''$ and~$\widehat\Pi_0''$ is normal isotopic to~$L$.
\end{enumerate}
Lemmas~\ref{wrinkle-decomp-lem}, \ref{gen-compres-decomp-lem}, and~\ref{no-compres-lem}
imply that~$\Pi_1'$ is obtained from~$\Pi_0'$ by a sequence of positive
moves that include only flypes and bubble moves.

Similarly to the previous case, one can show that~$\Pi_0''$ and~$\Pi_1''$
can be chosen so as to satisfy the additional condition~$\widehat\Pi_0'',\widehat\Pi_1''\subset M$,
which then implies~$\widehat\Pi_0',\widehat\Pi_1'\subset M$.

The surfaces~$\widehat\Pi_0'$ and~$\widehat\Pi_1'$ are normal, so they are
normal parallel to some surfaces contained in~$L_k$. Since~$\widehat\Pi_0',\widehat\Pi_1'\subset M$
they can be normal parallel only to~$F_0$ and~$F_1$. The case
when~$\widehat\Pi_0'$ is normal parallel to~$F_1$ is impossible, since
then~$\Pi_0'\cup\Pi_1$ would not admit positive moves leaving~$\Pi_1$ intact. Therefore,
$\widehat\Pi_0'$ is normal parallel to~$F_0$. For a similar reason, $\widehat\Pi_1'$ is
normal parallel to~$F_1$.

We are done with the case of an almost normal leaf~$L$.

It remains to consider the case when~$M$ contains a normal two-sphere. Let~$\Sigma$
be a rectangular diagram representing such a sphere. Suppose that~$\widehat\Sigma$
is `visible' from~$\widehat\Pi_0$ in the sense that that there is a subarc~$\alpha$
of an arc of the form~$\widehat v$, $v\in\mathbb T^2$, such that~$\alpha\subset M$
and~$\partial\alpha\subset\widehat\Pi_0\cup\widehat\Sigma$. Take for~$\Pi_0'$
a rectangular diagram obtained from~$\Pi_0$ by a small positive deformation
such that~$\widehat\Pi_0'$ is contained in~$M$ and separates~$F_0$ from~$\widehat\Sigma$.

Using Lemma~\ref{connected-sum-lem}, find a rectangular diagram of a surface~$\Pi_0''$
such that~$\Pi_0''\mapsto\Pi_0'\cup\Sigma$ is a negative compression. Let~$\Pi_0''$
normalize positively to~$\Pi_1'$. Similarly to the previous cases, we will have
that~$\widehat\Pi_1'$ is normal parallel to~$F_1$, and~$\Pi_1'$ is obtained from~$\Pi_0'$
by a sequence of positive flypes and positive bubble moves.

If~$\widehat\Sigma$ is not visible from~$\widehat\Pi_0$, it must be visible from~$\widehat\Pi_1$,
in which case we proceed similarly exchanging the roles of~$\Pi_0$ and~$\Pi_1$.

We are done with proving the existence of~$\Pi_0'$ and~$\Pi_1'$ having the announced properties.
Those properties imply that~$\overline M$ is a disjoint union of three parts~$M_0$, $M_1$, and~$C(M)$ such
that~$C(M)$ is a cavity, $M_0$ is enclosed by~$F_0$ and~$\widehat\Pi_0'$,
and~$M_1$ is enclosed by~$F_1$ and~$\widehat\Pi_1'$.
Moreover, each of the parts~$M_0$ and~$M_1$ can be filled with a codimension-one
foliation all leaves of which are represented by rectangular diagrams of surfaces obtained
from~$\Pi_0$ or~$\Pi_1'$, respectively, by positive deformations. The agregate of these two foliations
turn into the original foliation~$\mathscr F|_{\overline M}$ after deflating~$C(M)$.

Repeat the procedure described above for every path-connected component~$M$ of~$\mathbb S^3\setminus(\overline L_k\cup N_\varepsilon(R))$.
We obtain finitely many continuous one-parametric families~$\{\Pi_i,t\}_{t\in[0;1]}$ of rectangular diagrams of surfaces and a cavity~$C$ (that
may have several connected components) such that all surfaces~$\widehat\Pi_{i,t}$ become, after deflating~$C$, leaves
of a foliation~$\mathscr F'$ in~$\mathbb S^3\setminus N_\varepsilon(R)$ isotopic to~$\mathscr F$ relative to~$\widehat\Omega_\varepsilon(R)$.

Let~$\mathscr R$ be the union of all the diagrams~$\Pi_{i,t}$. Introduce an equivalence relation on~$\mathscr R$
as follows. Two rectangles~$r_1,r_2\in\mathscr R$ are said to be equivalent if there is a point~$v\in X_k\times Y_k$
such that~$v\in (r_1\setminus\partial r_1)\cap(r_2\setminus\partial r_2)$ and the subarc of~$\widehat v$ with endpoints~$\widehat v\cap\widehat r_1$ and~$\widehat v\cap\widehat r_2$
is contained in~$\bigcup_{r\in\mathscr R}\widehat r$. Since~$C$ has only finitely many components, $\mathscr R$
splits into finitely many equivalence classes, which we denote by~$P_1,P_2,\ldots,P_k$. Order rectangles
in each class as follows. Let~$r_1,r_2\in P_i$
be such that~$r_1$ overlays~$r_2$, and let~$v$ be as above. Let~$\alpha$ be the subarc of~$\widehat v$
with endpoints~$p_1=\widehat v\cap\widehat r_1$ and~$p_2=\widehat v\cap\widehat r_2$. Orient~$\alpha$
according to the coorientation of the foliation and put~$r_1>r_2$ in~$P_i$ if~$\partial\alpha=p_1-p_2$,
and~$r_1<r_2$ otherwise.

One can see that~$P_i$, $i=1,\ldots,k$, are packs of rectangles, and the collection of them forms
a rectangular diagram~$\Xi$ of a foliation such that~$\widehat\Xi$ is isotopic to~$\mathscr F$
relative to~$\widehat\Omega_\varepsilon(R)$.

The proof of Theorem~\ref{diagram-exist-thm} is now complete in the case when all leaves of the given
foliation are principal. Now we reduce the general case to this one.

Let~$k_0$ be the smallest integer such that all leaves of~$\mathscr F$ of depth~$\leqslant k_0$ are principal. The proof
is by induction in the pair~$(k-k_0,l)$, where~$l$ is defined below.
The induction base, $k=k_0$, has already been settled. Suppose that~$k_0<k$. For the same
reason as before, the number of equivalence classes of depth~$k_0$ leaves is finite. We denote by~$l$
the number of equivalence classes of depth~$k_0$ leaves~$F$ such that there is a leaf subordinate to~$F$.

Suppose that there is a leaf subordinate to a leaf~$F$ of depth~$k_0$. Then there is
submanifold~$M\subset\mathbb S^3\setminus N_\varepsilon(R)$ that contains all leaves equivalent to~$F$
and has the form~$\psi(F\times[0;1])$,
where~$\psi$ is an embedding such that~$\psi(F\times0)$ and~$\psi(F\times1)$ are leaves of~$\mathscr F$ equivalent to~$F$,
and all arcs of the form~$\psi(x\times[0;1])$ are transverse to~$\mathscr F$.

Now let~$\mathscr F'$ be the foliation that coincides with~$\mathscr F$ outside~$M$ and has the product
structure~$F\times[0;1]$ induced by~$\psi$ inside~$M$. Removing the interior of~$M$ from~$\mathbb S^3$ and identifying~$\psi(F\times0)$
with~$\psi(F\times1)$ yields another finite depth taut foliation in~$\mathbb S^3\setminus N_\varepsilon(R)$, which we denote by~$\mathscr F''$.
The foliation~$\mathscr F''$ is `simpler' than~$\mathscr F$ in the sense that the value of~$k-k_0$ for~$\mathscr F''$ is not larger
than that for~$\mathscr F$, and at least one of~$k-k_0$ and~$l$ is strictly smaller.

So, by the induction hypothesis, $\mathscr F''$ can be represented by a rectangular diagram, which means that
all leaves of~$\mathscr F''$ can be normalized,  after inserting some cavities, with respect to a triangulation of the form~$T(X,Y)$.
The foliation~$\mathscr F'$ can be obtained from~$F''$ by the Denjoy trick, which replaces a single leaf~$F$ with~$F\times[0;1]$.
Therefore, all leaves of~$F'$ can also be normalized, after inserting the same cavities, so that all leaves in~$F\times[0;1]$
are normal parallel to each other. Then the foliation of~$F\times[0;1]$ can be replaced by~$\mathscr F|_M$ so
that all leaves remain normal, which yields a normalization of all leaves of~$\mathscr F$. The induction step follows,
which completes the proof of Theorem~\ref{diagram-exist-thm} in the general case.
\end{proof}

\section{General finite depth foliations on~$\mathbb S^3$}\label{finite-depth-sec}
In the proof of Theorem~\ref{diagram-exist-thm}, the tautness of the foliation~$\mathscr F$ was used twice:
first, when we claim that the only depth zero leaves are boundary components of the link complement,
and second, when we conclude that all leaves are incompressible in the link complement. It follows from Novikov's theorem~\cite{novikov65}
that for the latter it suffices to demand only that~$\mathscr F$ is Reebless. The former condition is
actually not important. Indeed, if~$\mathscr F$ is Reebless and has compact leaves other than the boundary components, then
they can be easily normalized with respect to~$T(X_0,Y_0)$. So, the tautness hypothesis in
Theorem~\ref{diagram-exist-thm} can be weakened to the Reebless one.

The Reebless condition can also be dropped due to the following reason. The given foliation
can be deformed so that the core of each Reeb component is a knot of the form~$\widehat K$,
where~$K$ is a rectangular diagram of a knot. The union of cores of all Reeb components
can be added to the given link, and the foliation obtained from the given one
by drilling out the Reeb components can be represented by a rectangular diagram.
Each Reeb component, even knotted, of the original foliation can then also be represented
by a rectangular diagram in a way similar to the one described in Section~\ref{discs-sec}
for the case of an unknotted Reeb component.

Thus, with a little additional work, we have the following.

\begin{theo}
Assertion of Theorem~\ref{diagram-exist-thm} holds without the tautness hypothesis on the foliation~$\mathscr F$.
\end{theo}

It is plausible that the finite depth hypothesis in Theorem~\ref{diagram-exist-thm} can also be dropped, but
the proof will require a different method. We are going to address this question in a future work.

\end{document}